\documentclass[12pt,dvips]{amsart}
\usepackage{amsfonts, amssymb, latexsym, epsfig}

\newcommand\addtag{\refstepcounter{equation}\tag{\theequation}}
\newtheorem{Theorem}{Theorem}[section]
\newtheorem{Proposition}[Theorem]{Proposition}
\newtheorem{Lemma}[Theorem]{Lemma}
\newtheorem{Claim}[Theorem]{Claim}

\newtheorem{Corollary}[Theorem]{Corollary}

\newtheorem{Defclaim}[Theorem]{Definition-Claim}
\newtheorem{Main Conjecture}[Theorem]{Main Conjecture}

\newtheorem{Definition}[Theorem]{Definition}

\theoremstyle{remark}
\newtheorem{Example}[Theorem]{Example}

\newcommand\iso{{\cong}}

\newcommand\olambda{{\overline\lambda}}
\newcommand\omu{{\overline\mu}}
\newcommand\onu{{\overline\nu}}
\newcommand\oalpha{{\overline\alpha}}
\newcommand\obeta{{\overline\beta}}
\newcommand\okappa{{\overline\kappa}}

\newcommand\otau{{\overline\tau}}

\newcommand{\excise}[1]{}

\newcommand\ourflag{{Fl}_{1,n-1;n}}
\theoremstyle{plain}

\newtheorem{Fact}[Theorem]{Fact}

\newcommand\leftass{\sctableau{{\ }\\{\ }}\diamond(\oalpha\diamond\obeta)}
\newcommand\rightass{(\sctableau{{\ }\\{\ }}\diamond\oalpha)\diamond\obeta}
\newcommand\leftassB{\sctableau{{\ }\\{\ }}\star(\olambda\star\omu)}
\newcommand\rightassB{(\sctableau{{\ }\\{\ }}\star\olambda)\star\omu}

\newcommand\leftassbox{\tableau{{\ }}\diamond\Pi(\olambda\star\omu)}

\newcommand\leftassboxB{\tableau{{\ }}\star(\olambda\star\omu)}
\newcommand\rightassboxB{(\tableau{{\ }}\star\olambda)\star\omu}

\newcommand\leftassgenB{g^{\uparrow}\star(\olambda\star\omu)}
\newcommand\rightassgenB{(g^{\uparrow}\star\olambda)\star\omu}

\newcommand{\scellsize}{5}
\newlength{\scellsz} 
\newsavebox{\scell}
\sbox{\scell}{\begin{picture}(\scellsize,\scellsize)
\put(0,3){\line(1,0){\scellsize}}
\put(0,3){\line(0,1){\scellsize}}
\put(\scellsize,3){\line(0,1){\scellsize}}
\put(0,8){\line(1,0){\scellsize}}
\end{picture}}
\newcommand\scellify[1]{\def\thearg{#1}\def\nothing{}%
\ifx\thearg\nothing
\vrule width0pt height\scellsz depth0pt\else
\hbox to 0pt{\usebox{\scell} \hss}\fi%
\vbox to \scellsz{
\vss
\hbox to \scellsz{\hss$#1$\hss}
\vss}}
\newcommand\sctableau[1]{\vtop{\let\\\cr
\baselineskip -16000pt \lineskiplimit 16000pt \lineskip 0pt
\ialign{&\scellify{##}\cr#1\crcr}}}
\newcommand{\cellsize}{9}
\newlength{\cellsz} 
\newsavebox{\cell}
\sbox{\cell}{\begin{picture}(\cellsize,\cellsize)
\put(0,0){\line(1,0){\cellsize}}
\put(0,0){\line(0,1){\cellsize}}
\put(\cellsize,0){\line(0,1){\cellsize}}
\put(0,\cellsize){\line(1,0){\cellsize}}
\end{picture}}
\newcommand\cellify[1]{\def\thearg{#1}\def\nothing{}%
\ifx\thearg\nothing
\vrule width0pt height\cellsz depth0pt\else
\hbox to 0pt{\usebox{\cell} \hss}\fi%
\vbox to \cellsz{
\vss
\hbox to \cellsz{\hss$#1$\hss}
\vss}}
\newcommand\tableau[1]{\vtop{\let\\\cr
\baselineskip -16000pt \lineskiplimit 16000pt \lineskip 0pt
\ialign{&\cellify{##}\cr#1\crcr}}}
\newcommand{\kellsize}{9}
\newlength{\kellsz} 
\newsavebox{\kell}
\sbox{\kell}{\begin{picture}(\kellsize,\kellsize)
\put(0,0){\line(1,0){\kellsize}}
\put(0,0){\line(0,1){\kellsize}}
\put(\kellsize,0){\line(0,1){\kellsize}}
\put(0,\kellsize){\line(1,0){\kellsize}}
\end{picture}}
\newcommand\kellify[1]{\def\thearg{#1}\def\nothing{}%
\ifx\thearg\nothing
\vrule width0pt height\kellsz depth0pt\else
\hbox to 0pt{\usebox{\kell} \hss}\fi%
\vbox to \kellsz{
\vss
\hbox to \kellsz{\hss$#1$\hss}
\vss}}
\newcommand\ktableau[1]{\vtop{\let\\\cr
\baselineskip -16000pt \lineskiplimit 16000pt \lineskip 0pt
\ialign{&\kellify{##}\cr#1\crcr}}}

\newcommand{\sellsize}{5}
\newlength{\sellsz} 
\newsavebox{\sell}
\sbox{\sell}{\begin{picture}(\sellsize,7)
\put(0,0){\line(1,0){\sellsize}}
\put(0,0){\line(0,1){\sellsize}}
\put(\sellsize,0){\line(0,1){\sellsize}}
\put(0,\sellsize){\line(1,0){\sellsize}}
\end{picture}}
\newcommand\sellify[1]{\def\thearg{#1}\def\nothing{}%
\ifx\thearg\nothing
\vrule width0pt height\sellsz depth0pt\else
\hbox to 0pt{\usebox{\sell} \hss}\fi%
\vbox to \sellsz{
\vss
\hbox to \sellsz{\hss$#1$\hss}
\vss}}
\newcommand\stableau[1]{\vtop{\let\\\cr
\baselineskip -16000pt \lineskiplimit 16000pt \lineskip 0pt
\ialign{&\sellify{##}\cr#1\crcr}}}

\begin{document}
\pagestyle{plain}
\title{Root-theoretic Young diagrams and Schubert calculus:\\ 
planarity and the adjoint varieties}
\author{Dominic Searles}

\author{Alexander Yong}
\address{Dept. of Mathematics\\
University of Illinois at Urbana-Champaign\\
Urbana, IL 61801}
\email{searles2@uiuc.edu, ayong@uiuc.edu}
\subjclass[2000]{14M15, 14N15}
\keywords{Root-theoretic Young diagrams, Adjoint varieties, Schubert calculus}

\date{August 29, 2013}

\maketitle

\begin{abstract}
We study root-theoretic Young diagrams 
to investigate the existence of a Lie-type uniform and nonnegative 
combinatorial rule for Schubert calculus. 

We provide formulas for (co)adjoint varieties of classical Lie type. This is a simplest case after the (co)minuscule family (where a rule has been proved
by H.~Thomas and the second author using work of R.~Proctor).
Our results build on earlier Pieri-type rules of P.~Pragacz-J.~Ratajski and of A.~Buch-A.~Kresch-H.~Tamvakis.
Specifically, our formula for $OG(2,2n)$ is the first complete rule for a case where diagrams are non-planar. 
Yet the formulas possess both uniform and non-uniform features. 

Using these classical type rules, as well as results of P.-E.~Chaput-N.~Perrin in the exceptional types,
we suggest a connection between polytopality of the set of nonzero Schubert structure constants
and planarity of the diagrams. This is an addition to work of A.~Klyachko and A.~Knutson-T.~Tao on the Grassmannian and of K.~Purbhoo-F.~Sottile on cominuscule varieties, where the diagrams are always planar.
\end{abstract}

\tableofcontents
\section{Introduction}

\subsection{Overview} Consider the following problem:
\begin{quote}
Does there exist a \emph{root-system uniform} and \emph{manifestly nonnegative} combinatorial rule for Schubert calculus?
\end{quote}

Let $G$ be a complex reductive Lie group.
Fix a choice $B$ of a Borel subgroup and maximal torus $T$, and
let $W$ be its Weyl group: $W\cong N_G(T)/T$.
Write $\Phi = \Phi^{+}\cup\Phi^{-}$ to be the partition
of roots into positives and negatives,
and let $\Delta$ be the base of simple roots. 
Let $\Omega_G = (\Phi^+,\prec)$ denote the canonical poset structure on $\Phi^+$. 
Suppose $\Delta_P=\{\beta(P)_1,\ldots, \beta(P)_k\}\subseteq \Delta$ identifies
the parabolic subgroup $P$, and set $W_P:=W_{\Delta_P}$ as the associated parabolic subgroup of $W$. Consider the subposet
\[\Lambda_{G/P}=\{\alpha\in \Phi^{+}: \mbox{$\beta_i(P)\prec \alpha$ for some $i$}\}\subseteq \Omega_G.\]
The
{\bf Schubert varieties} in $G/P$ are $X_{wW_P}={\overline{B_{-}w P/P}}$
where $wW_P\in W/W_P$. Suppose $w$ is the minimal length coset
representative of $wW_P$; $w$'s inversion set
$\olambda$ sits inside $\Lambda_{G/P}$.
Let us write $X_{\olambda}:=X_{wW_P}$.
For short, call $\olambda$ a {\bf root-theoretic Young diagram} (RYD).
Let ${\mathbb Y}_{G/P}$ be the set of RYDs for $G/P$.

The cohomology ring $H^{\star}(G/P)$ has a ${\mathbb Z}$-additive basis of
{\bf Schubert classes} $\sigma_{\olambda}$. Let $C_{\olambda,\omu}^{\onu}(G/P)$ denote the Schubert structure constants for $G/P$, i.e.,
    \[\sigma_{\olambda}\cdot\sigma_{\omu}=\sum_{\onu} C_{\olambda,\omu}^{\onu}(G/P)\sigma_{\onu}.\]
    When $G/P$ is the Grassmannian $Gr_k({\mathbb C}^n)$,  $C_{\olambda,\omu}^{\onu}:=C_{\olambda,\omu}^{\onu}(Gr_k({\mathbb C}^n))$
    is computed by the \emph{Littlewood-Richardson rule}. 

Ideally, there is a generalization to compute 
$C_{\olambda,\omu}^{\onu}(G/P)$ in a cancellation-free fashion, only using the associated root datum 
(cancellative formulas are known, see, e.g.,
\cite{Knutson:algorithm}). 

Actually, often
the main question is phrased presuming the existence of a rule. 
However, in that case, what is the qualitative nature of such a putative rule? 
Is it too much to expect a counting rule like the \emph{Littelmann path model}? 
Perhaps it makes more sense to search for a patchwork of counting rules and 
nonnegative recursions through different $G/P$'s for varying $G$'s. How can one classify special cases? Why are some special cases of the problem seemingly harder than others? Finally, if one believes that such a rule does \emph{not} exist, what are
concrete and/or falsifiable reasons for that belief?

The main thesis of this paper is that 
RYDs provide a simple but uniform combinatorial perspective to discuss such questions mathematically,
make precise comparisons, and to measure progress towards a rule (uniform, counting, patchwork, or otherwise). 

For example, from this perspective, Grassmannians are special because they sit in the family of $G/P$'s 
for which the above root-system setup is especially graphical:
    \begin{itemize}
    \item[(I)] $\Lambda_{G/P}$ is a planar poset;
\medskip
    \item[(II)] the RYDs are
    lower order ideals (and in fact classical Young diagrams, thus explaining the nomenclature);
\medskip
    \item[(III)] Bruhat order is containment of RYDs.
    \end{itemize}
    These properties also hold for all cominuscule $G/P$'s. Using work of R.~Proctor \cite{Proctor}, they help demonstrate a uniform rule for (co)minuscule Schubert calculus \cite{Thomas.Yong:comin}.

At present, using RYDs is the only known way to solve the problem for (co)minuscule $G/P$. 
Conversely, it is only for (co)minuscule $G/P$'s that there is a uniform rule. It is therefore sensible to 
use RYDs to study other families. 

It seems to us that the key next case is the family of (co)adjoint $G/P$'s. One reason is that 
this family extends the (co)minuscule $G/P$'s, see, e.g., \cite{Lak} and
Section~2.1. However, our main point is that in terms of RYDs, none of the properties (I), (II) or (III) hold in general
for (co)adjoint varieties. Equally important to us is that the
failures of these properties are quantifiably mild (see Fact~\ref{PropAdjointfacts} below). 

Note that use of RYDs, even for (co)adjoint varieties, is not mandatory: there is a different way to index their Schubert varieties, see
\cite{Chaput.Perrin}. For isotropic Grassmannians of classical type, \cite{Prag, PragD} uses another way, and \cite{BKT:Inventiones} yet another.

We obtain positive Schubert calculus rules in the classical (co)adjoint types; this is the principal new evidence we have for the thesis beyond what already fits from the literature. These rules 
have significant, but far from complete, uniformity. 
Additional complexity of $OG(2,2n)$ comes from the nonplanarity of $\Lambda_{OG(2,2n)}$.
To our best knowledge, we give the first complete formula for any $G/P$ with 
nonplanarity --- and what we find is
that it has patchwork features for which we have no broad explanation. Indeed, it separates out the cases covered by the Pieri rule of \cite{BKT:Inventiones}. Perhaps
surprisingly, it is these ``Pieri cases'' that bring unappetizing complications to our rule. Also, our rule depends on the parity of $n$. This is  
traceable to the fact that $\Lambda_{{\mathbb Q}^{2n-4}}$ is a 
subposet of $\Lambda_{OG(2,2n)}$
and that the even-dimensional quadric 
${\mathbb Q}^{2n-4}$ has this dependency as well \cite{Thomas.Yong:comin}.

We think it is plausible that the patchwork features of our rule for $OG(2,2n)$ are essentially unavoidable if maintaining uniformity 
with the other (co)adjoint and (co)minuscule
varieties. That is, we would infer our results present a specific challenge  to
the existence of a root-system uniform counting rule.

Now, there are a number of reasons to doubt this interpretation.
First, in \cite{Chaput.Perrin:Kac}, RYDs are used to generalize \cite{Thomas.Yong:comin}. Their extension uniformly covers a subset of the Schubert problems in each of the (co)adjoint varieties --- but precisely those that are ``cominuscule-like''. Second, the ``flattening trick'' used for the $OG(2,2n)$
problem is non-uniform. However, this step is what allows us to make comparisons with the other (co)adjoint formulas. Third,
there are alternative and powerful models such as \emph{Chains in Bruhat order}, see, e.g., \cite{Bergeron.Sottile:duke}, 
\emph{Puzzles} \cite{Knutson.Tao} and \emph{Mondrian tableaux} \cite{Coskun} among others. However, we reiterate that other approaches are not known yet to (uniformly) resolve the (co)minuscule case, which we think is the simpler problem. 

Some additional support for the main thesis comes from analysis of 
another Schubert calculus problem:

\begin{quote}
What is the set $S^{\tt nonzero}(G/P)$ of $(\olambda,\omu,\onu)\in ({\mathbb Y}_{G/P})^3$ where $C_{\olambda,\omu}^{\onu}(G/P)\neq 0$?
\end{quote}

There is a celebrated \emph{polytopal} answer for $Gr_k({\mathbb C}^n)$. More specifically,
identifying a Young diagram $\olambda$ with its partition $(\olambda_1\geq \olambda_2\geq \ldots \geq\olambda_k)\in {\mathbb Z}^k$, $S^{\tt nonzero}({Gr}_k({\mathbb C}^n))$ can be viewed as the lattice points of the \emph{Horn polytope}
in ${\mathbb Z}^{3k}$. This result was first established
by the combined work of A.~Klyachko \cite{Klyachko} and of
A.~Knutson-T.~Tao \cite{Knutson.Tao:sat}; see also W.~Fulton's survey 
\cite{Fultona} for a historical discussion. 
More recently, K.~Purbhoo-F.~Sottile \cite{Purbhoo.Sottile} established similar descriptions for cominuscule $G/P$ using RYDs. Also, RYDs are used to study the nonzeroness problem in K.~Purbhoo's \cite{Purbhoo}. Both of the latter two papers also provide pre-existing evidence for the thesis.

It is natural to ask when one can expect a polytopal answer to the above question. Indeed, in the Introduction of \cite{Purbhoo.Sottile} the authors 
write ``\emph{We use that $G/P$ is cominuscule in many essential ways in our arguments, which suggests that
cominuscule flag varieties are the natural largest class of flag varieties for which these tangent space methods can be
used to study the non-vanishing of [generic Schubert intersections].}'' For comparison, 
we offer a partial answer. Using specific drawings of $\Lambda_{G/P}$ we associate, in
a type by type manner, a vector of row lengths to each RYD in ${\mathbb Y}_{G/P}$. We will call these descriptions \emph{partition-like}
since they mimic the partition description of Young diagrams used to formulate the Horn polytope. 
Our most general finding is:

\begin{Theorem}
\label{thm:iff}
For adjoint $G/P$, there is a polytopal description of $S^{\tt nonzero}(G/P)$ using the partition-like description of RYDs if and only if $\Lambda_{G/P}$ is planar. 

For coadjoint $G/P$, let $G/Q$ be the adjoint partner. Then 
\[C_{\olambda,\omu}^{\onu}(G/P)=m(G)^{{\rm sh}(\olambda)+{\rm sh}(\omu)-{\rm sh}(\onu)}
C_{\olambda,\omu}^{\onu}(G/Q).\]
\end{Theorem}

Here ${\rm sh}(\olambda)$ is the number of short roots of $\olambda$
and $m(G)$ is the maximum multiplicity of an edge (hence, e.g., $m(G_2)=3$ and $m(F_4)=2$) of the Dynkin diagram of $G$.
For $G$ simply-laced, $G/P=G/Q$ and $m(G)=1$. This uniformly extends the shortroots
correspondence from \cite{Thomas.Yong:comin} in terms of the Cartan classification. (One can index
the Schubert varieties of coadjoint $G/P$ using the RYDs for the adjoint partner $G/Q$.)

We emphasize that Theorem~\ref{thm:iff} does 
not rule out possible polytopal solutions that begin with a different, natural vectorial description of RYDs. Does such a description exist?
Example~\ref{exa:D5badness} for $OG(2,10)$ encapsulates our doubts. In any case, our point is that Theorem~\ref{thm:iff} again indicates a planar/non-planar 
dichotomy in combinatorial Schubert calculus. Theorem~\ref{thm:iff} complements the results of
\cite{Knutson.Tao:sat, Purbhoo.Sottile} on the 
nonzeroness problem within the family of $G/P$'s for quasi-(co)minuscule $P$ (see Section~2.1).

For the classical types of Theorem~\ref{thm:iff}, we obtain a description of $S^{\tt nonzero}(G/P)$ directly from our formulas.  For the non-planar cases we find a ``zero triple'' $(\olambda,\omu,\onu)$ that is a convex
combination of some ``nonzero triples''.
In the exceptional types, we use explicit calculation using \cite{Chaput.Perrin, Chaput.Perrin:computations}. The computations are done by computer, but the ones needed for the proof are actually small enough to be (onerously) checked by hand.

Another piece of evidence for the value of RYDs comes from a new rule for the $GL_n$ Belkale-Kumar coefficients \cite{Searles:13+} (after A.~Knutson-K.~Purbhoo \cite{Knutson.Purbhoo}). A comparison of RYDs to the indexing system of \cite{BKT:Inventiones} is also given in \emph{loc.~cit.} (We also mention that RYDs can also be applied to the study of Kazhdan-Lusztig polynomials \cite{Woo.Yong:KL}.)

\subsection{Definition of (co)adjoint varieties}
The following is standard. Fix a representation $\rho:G\to GL(V)$ for some finite dimensional
complex vector space $V$. The group $G$ acts on ${\mathbb P}(V)$
through the projection $\pi:V\setminus\{0\}\to {\mathbb P}(V)$.
If $\vec v$ is a highest weight vector of $\rho$, then
\[\pi(G\cdot v)\subseteq {\mathbb P}(V)\] is a homogeneous projective variety, see, e.g., \cite[Section~23.3]{Fulton.Harris}. This variety is
{\bf adjoint} if $\rho$ is the adjoint representation of $G$. Adjoint varieties have a root-system theoretic classification, see, e.g., \cite{Chaput.Perrin} and the references therein as well as Table~1 in Section~2. A variety is {\bf coadjoint} if it is adjoint for the dual root system.

       The highest root of $\Lambda_{G/P}$ is the {\bf adjoint root}. If ${\overline\lambda}$ uses it we say $\olambda$ is {\bf on} and we write $\olambda=\langle\lambda|\bullet\rangle$; otherwise we say $\olambda$ is {\bf off} and we write $\olambda=\langle\lambda|\circ\rangle$, where $\lambda$ comprises the roots of
$\Lambda_{G/P}\setminus\{\mbox{adjoint root}\}$ used by $\olambda$. Let $\prec_{\rm Bruhat}$ be the
order on ${\mathbb Y}_{G/P}$ defined by the closure order on Schubert cells.
We recall some facts; cf., \cite[Section~2]{Chaput.Perrin} and the references therein.

\begin{Fact}\label{PropAdjointfacts}
If $G/P$ is adjoint then:
\begin{itemize}
\item[(i)] $|\Lambda_{G/P}|$ is odd
\item[(ii)] If $\olambda=\langle\lambda|\circ\rangle$ then $|\olambda|\leq\frac{|\Lambda_{G/P}|-1}{2}$
\item[(iii)] If $\olambda=\langle\lambda|\bullet\rangle$ then $|\olambda|\geq \frac{|\Lambda_{G/P}|+1}{2}$
\item[(iv)] $\lambda$ is a lower order ideal in the poset $\Lambda_{G/P}\setminus\{\mbox{adjoint root}\}$
\item[(v)] $\langle\lambda|\circ\rangle\prec_{\rm  Bruhat} \langle \mu|\circ\rangle$ and 
$\langle\lambda|\bullet\rangle\prec_{\rm Bruhat} \langle \mu|\bullet\rangle$
if and only if $\lambda\subseteq\mu$
\item[(vi)] $\langle\lambda|\circ\rangle\prec_{\rm Bruhat} \langle \mu|\bullet\rangle$ if and only if $|\lambda
\setminus \mu| \le 1$
\end{itemize}
\end{Fact}
Points (iv)--(vi) explain in what sense the failures of (II) and (III) 
are  ``mild''. 

In the  classical Lie types, there are three (co)adjoint varieties of main interest.
Those are what we discuss in the technical core of our results below.

\subsection{Definition of ${\mathbb A}_{\olambda,\omu}$; main theorem for odd orthogonal Grassmannians $OG(2,2n+1)$ and Lagrangian Grassmannians $LG(2,2n)$}
For the Lie type $B_n$, the adjoint variety $G/P=OG(2,2n+1)$ is the space of isotropic $2$-planes
with respect to a non-degenerate symmetric bilinear form on ${\mathbb C}^{2n+1}$. It has
dimension $|\Lambda_{G/P}|=4n-5$. 

\excise{
\begin{figure}[h]
\begin{center}
\epsfig{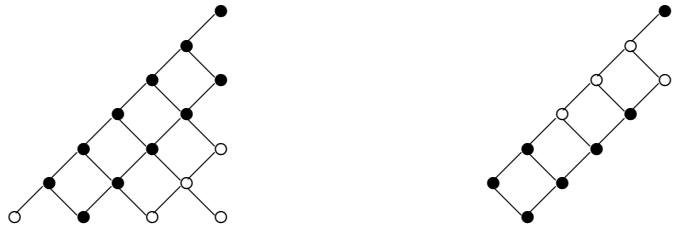}
\end{center}
\caption{
$\Lambda_{OG(2,2n+1)}$,
$\Omega_{SO_{2n+1}}$ and a shape (for $n=5$)}
\end{figure}
}

\[\begin{picture}(300,105)
\multiput(24.5,29.5)(13,13){6}{\line(1,1){10}}
\multiput(50.5,29.5)(13,13){4}{\line(1,1){10}}
\multiput(76.5,29.5)(13,13){2}{\line(1,1){10}}

\put(21,25){$\circ$}
\multiput(34,38)(13,13){6}{$\bullet$}
\multiput(47,25)(13,13){5}{$\bullet$}
\multiput(73,25)(13,13){3}{$\circ$}
\put(99,25){$\circ$}

\multiput(89.5,91)(13,-13){1}{\line(1,-1){10}}
\multiput(76.5,78)(13,-13){2}{\line(1,-1){10}}
\multiput(63.5,65)(13,-13){3}{\line(1,-1){10}}
\multiput(50.5,52)(13,-13){2}{\line(1,-1){10}}
\multiput(37.5,39)(13,-13){1}{\line(1,-1){10}}

\multiput(202,38)(13,13){2}{$\bullet$}
\multiput(228,64)(13,13){3}{$\circ$}
\multiput(267,103)(13,13){1}{$\bullet$}

\multiput(215,25)(13,13){4}{$\bullet$}
\put(267,77){$\circ$}

\multiput(205.5,42.5)(13,13){5}{\line(1,1){10}}

\multiput(218.5,29.5)(13,13){4}{\line(1,1){10}}

\multiput(205.5,39)(13,-13){1}{\line(1,-1){10}}
\multiput(218.5,52)(13,-13){1}{\line(1,-1){10}}
\multiput(231.5,65)(13,-13){1}{\line(1,-1){10}}
\multiput(244.5,78)(13,-13){1}{\line(1,-1){10}}
\multiput(257.5,91)(13,-13){1}{\line(1,-1){10}}

\put(40,0){$\Lambda_{OG(2,2n+1)}$, $\Omega_{SO_{2n+1}}$ and a shape (for $n=4$)}
\end{picture}
\]

The coadjoint partner to $OG(2,2n+1)$ in the $C_n$ root system is the variety $LG(2,2n)$ of isotropic $2$-planes
with respect to a non-degenerate skew-symmetric bilinear form on ${\mathbb C}^{2n}$.  As with all cases, it makes sense to
index the Schubert varieties for the coadjoint variety with RYDs for its adjoint partner. This is analogous to \cite{Thomas.Yong:comin}. We denote the shapes $\olambda$ by $\langle \lambda|{\bullet/}{\circ\rangle}$, where $\lambda$ is a partition in
$2\times \left(\frac{|\Lambda_{G/P}|-1}{2}\right)$.

We will need a reusable definition. For any $\nu=(\nu_1,\nu_2)\in {\mathbb Z}^2$ let $\nu^{\star}=(\nu_1-1,\nu_2)$ and $\nu_{\star}=(\nu_1,\nu_2-1)$.
Fix $\olambda$ and $\omu$ and define   
\[{\mathbb{A}}_{\olambda,\omu}(\nu)=
\begin{cases}
0 & \mbox{if $\olambda$ and $\omu$ are on}\\
\sigma_{\langle \nu|\bullet\rangle} & \mbox{if exactly one of $\olambda$ or $\omu$ is on}\\
\sigma_{\langle \nu|\circ\rangle} & \mbox{if $|\olambda|+|\omu| \leq \frac{|\Lambda_{G/P}|-1}{2}$} \\
\sigma_{\langle \nu^{\star}|\bullet\rangle}+\sigma_{\langle \nu_{\star}|\bullet\rangle} & \mbox{otherwise}.
\end{cases}\]

In the ``otherwise'' case of the definition of ${\mathbb{A}}_{\olambda,\omu}(\nu)$ a nonadjoint root from $\nu$ 
has ``jumped'' to become the adjoint root. Understanding how this occurs in each
type is key in the (co)adjoint cases. This reflects the
additional complexity coming from the failure of (II). 

Say $\sigma_{\langle \nu|\bullet/\circ\rangle}$, $\sigma_{\langle \nu^{\star}|\bullet\rangle}$ or $\sigma_{\langle\nu_{\star}|\bullet\rangle}$ is zero if $\nu$, $\nu^{\star}$ or
$\nu_{\star}$ is not a partition in $2\times \left(\frac{|\Lambda_{G/P}|-1}{2}\right)$. Define ${\rm sh}(\onu)$ to be the number of short roots used by $\onu$. The short roots of $\Lambda_{OG(2,2n+1)}$ consist of the middle pair of the nonadjoint roots. 

\begin{Theorem}\label{Thm:BCintro}
\[\sigma_{\olambda}\cdot \sigma_{\omu}=\sum_{\nu\subseteq \left(\frac{|\Lambda_{G/P}|+1}{2},\frac{|\Lambda_{G/P}|-1}{2}\right)}C_{\lambda,\mu}^{\nu} {\mathbb{A}}_{\olambda,\omu}(\nu)\in H^{\star}(LG(2,2n)).\]
In $H^{\star}(OG(2,2n+1))$, multiply each coefficient by $2^{{\rm sh}(\onu)-{\rm sh}(\olambda)-{\rm sh}(\omu)}$; the result
is provably integral.
\end{Theorem}
While our rule for $OG(2,2n+1)$ is manifestly positive, it is not manifestly
integral because $2^{{\rm sh}(\onu)-{\rm sh}(\olambda)-{\rm sh}(\omu)}=\frac{1}{2}$
does occur. However, integrality is not difficult and is handled by
Proposition~\ref{prop:OGoddintegral}.

\begin{Example}
In $H^\star(LG(2,8))$, $\sigma_{\langle 3,1|\circ\rangle}\cdot\sigma_{\langle 3,2|\circ\rangle}=2\sigma_{\langle 5,3|\bullet\rangle}+
\sigma_{\langle 4,4|\bullet\rangle}$:
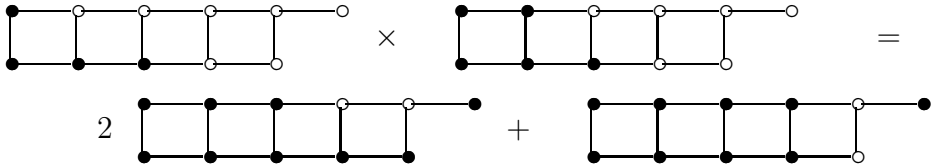
\[\begin{picture}(390,55)
\multiput(0,35)(25,0){5}{$\circ$}
\multiput(0,55)(25,0){6}{$\circ$}
\multiput(0,55)(25,0){1}{$\bullet$}
\multiput(0,35)(25,0){3}{$\bullet$}
\multiput(4,38)(25,0){4}{\line(1,0){21}}
\multiput(4,58)(25,0){5}{\line(1,0){21}}
\multiput(2,40)(25,0){5}{\line(0,1){16.5}}
\put(140,45){$\times$}

\multiput(170,35)(25,0){5}{$\circ$}
\multiput(170,55)(25,0){6}{$\circ$}
\multiput(170,55)(25,0){2}{$\bullet$}
\multiput(170,35)(25,0){3}{$\bullet$}
\multiput(174,38)(25,0){4}{\line(1,0){21}}
\multiput(174,58)(25,0){5}{\line(1,0){21}}
\multiput(172,40)(25,0){5}{\line(0,1){16.5}}

\put(330,45){$=$}

\put(35,10){$2$}
\multiput(50,0)(25,0){5}{$\circ$}
\multiput(50,20)(25,0){6}{$\circ$}
\multiput(50,20)(25,0){3}{$\bullet$}
\multiput(50,0)(25,0){5}{$\bullet$}
\multiput(54,3)(25,0){4}{\line(1,0){21}}
\multiput(54,23)(25,0){5}{\line(1,0){21}}
\multiput(52,05)(25,0){5}{\line(0,1){16.5}}
\put(175,20){$\bullet$}

\put(190,10){$+$}
\multiput(220,0)(25,0){5}{$\circ$}
\multiput(220,20)(25,0){6}{$\circ$}
\multiput(220,20)(25,0){4}{$\bullet$}
\multiput(220,0)(25,0){4}{$\bullet$}
\multiput(224,3)(25,0){4}{\line(1,0){21}}
\multiput(224,23)(25,0){5}{\line(1,0){21}}
\multiput(222,5)(25,0){5}{\line(0,1){16.5}}
\put(345,20){$\bullet$}
\end{picture}
\]
Similarly, in $H^\star(OG(2,9))$, we compute $\sigma_{\langle 2,1|\circ\rangle}\cdot\sigma_{\langle 3,2|\circ\rangle}=\sigma_{\langle 5,2|\bullet\rangle}+
4\sigma_{\langle 4,3|\bullet\rangle}$.
\qed
\end{Example}

\begin{Corollary}\label{cor:BCcoefficients}
$\left\{C_{\olambda,\omu}^{\onu}(LG(2,2n))\right\}= \{0,1,2\}$ and $\left\{C_{\olambda,\omu}^{\onu}(OG(2,2n+1))\right\}= \{0,1,2,4,8\}$.
\end{Corollary}

Declare the partition-like description of RYDs in this case to identify
\begin{equation}
\label{eqn:identify}
\mbox{$\olambda=\langle\lambda_1,\lambda_2|\circ\rangle$ with $(\lambda_1,\lambda_2,0)\in {\mathbb Z}^3$ and
$\olambda=\langle\lambda_1,\lambda_2|\bullet \rangle$ with $(\lambda_1,\lambda_2,1)\in {\mathbb Z}^3$.}
\end{equation}

Thus we arrive at our first case of the nonzeroness question:
\begin{Corollary}\label{cor:BCnonzero}
Assume $\lambda=(\lambda_1,\lambda_2),\mu=(\mu_1,\mu_2),\nu=(\nu_1,\nu_2)\subset 2\times  \left(\frac{|\Lambda_{G/P}|-1}{2}\right)$ are partitions and 
$\lambda_3,\mu_3,\nu_3\in \{0,1\}$.  Then $C_{\olambda,\omu}^{\onu}(LG(2,2n))\neq 0$ and $C_{\olambda,\omu}^{\onu}(OG(2,2n+1))\neq 0$ if and only if:
\begin{align*}
\label{eqn:horntype}
|\onu| & =  |\olambda|+|\omu|\\ 
\nu_1 & \leq  \lambda_1+\mu_1\\  
\nu_2 & \leq  \lambda_1+\mu_2\addtag \\
\nu_2 & \leq  \lambda_2+\mu_1\\ 
\lambda_3+\mu_3& \leq  \nu_3
\end{align*}
\end{Corollary}

The inequalities (except the last) come from those for the Horn polytope for $k=2$.

Corollary~\ref{cor:BCnonzero} shows that neither the failure of (II) nor (III) bar a polytopal answer to the nonzeroness question.

\subsection{Main theorem for even orthogonal Grassmannians $OG(2,2n)$}
 The adjoint variety $G/P=OG(2,2n)$ is the space of isotropic $2$-planes
with respect to a non-degenerate symmetric bilinear form on ${\mathbb C}^{2n}$. It has dimension $|\Lambda_{G/P}|=4n-7$.

\[\begin{picture}(300,105)

\multiput(24.5,29.5)(13,13){3}{\line(1,1){10}}
\multiput(50.5,29.5)(13,13){2}{\line(1,1){10}}
\multiput(76.5,29.5)(13,13){1}{\line(1,1){10}}

\put(21,25){$\circ$}
\multiput(34,38)(13,13){3}{$\bullet$}
\multiput(47,25)(13,13){3}{$\bullet$}
\multiput(73,25)(13,13){2}{$\circ$}
\put(99,25){$\circ$}

\multiput(85,25)(-13,13){2}{$\circ$}
\multiput(59,51)(13,13){3}{$\bullet$}
\multiput(46,64)(13,13){4}{$\bullet$}
\put(85,51){$\circ$}

\multiput(50.5,68.5)(13,13){3}{\line(1,1){10}}
\multiput(63.5,55.5)(13,13){2}{\line(1,1){10}}
\multiput(76.5,42.5)(13,13){1}{\line(1,1){10}}

\multiput(75.5,39)(-13,13){3}{\line(1,-1){10}}
\multiput(63.5,78)(13,-13){2}{\line(1,-1){10}}
\multiput(75.5,91)(13,-13){1}{\line(1,-1){10}}
\multiput(63.5,65)(13,-13){3}{\line(1,-1){10}}
\multiput(50.5,52)(13,-13){2}{\line(1,-1){10}}
\multiput(37.5,39)(13,-13){1}{\line(1,-1){10}}

\put(49.5,55.5){\line(0,1){10}}
\put(75.5,55.5){\line(0,1){10}}
\put(62.5,42.5){\line(0,1){10}}
\put(62.5,68.5){\line(0,1){10}}
\put(75.5,29.5){\line(0,1){10}}
\put(88.5,42.5){\line(0,1){10}}

\multiput(202,38)(13,13){2}{$\bullet$}
\multiput(228,64)(13,13){1}{$\circ$}

\multiput(215,64)(13,13){1}{$\bullet$}
\multiput(228,77)(13,13){2}{$\circ$}
\put(254,103){$\bullet$}

\multiput(228,51)(13,13){2}{$\bullet$}
\put(254,77){$\circ$}

\multiput(215,25)(13,13){3}{$\bullet$}

\multiput(205.5,42.5)(13,13){2}{\line(1,1){10}}

\multiput(218.5,29.5)(13,13){2}{\line(1,1){10}}

\multiput(205.5,39)(13,-13){1}{\line(1,-1){10}}
\multiput(218.5,52)(13,-13){1}{\line(1,-1){10}}
\multiput(231.5,65)(13,-13){1}{\line(1,-1){10}}

\put(217.5,55.5){\line(0,1){10}} 
\put(243.5,55.5){\line(0,1){10}}
\put(230.5,42.5){\line(0,1){10}}
\put(230.5,68.5){\line(0,1){10}}

\multiput(218.5,68.5)(13,13){3}{\line(1,1){10}}
\multiput(231.5,55.5)(13,13){2}{\line(1,1){10}}

\put(219.5,66){\line(1,-1){10}}
\put(232.5,79){\line(1,-1){10}}
\put(245.5,92){\line(1,-1){10}}

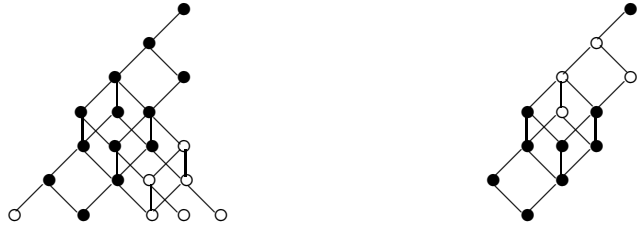
\put(40,0){$\Lambda_{OG(2,2n)}$, $\Omega_{SO_{2n}}$ and a shape (for $n=5$)}

\end{picture}
\]

Here $\Lambda_{G/P}$ is not planar:
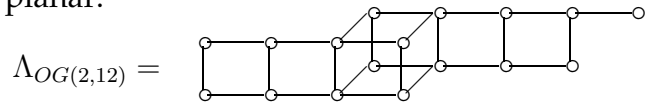
\[\begin{picture}(140,30)
\put(-60,10){$\Lambda_{OG(2,12)}=$}
\multiput(10,0)(25,0){4}{$\circ$}
\multiput(10,20)(25,0){4}{$\circ$}
\multiput(14,3)(25,0){3}{\line(1,0){21}}
\multiput(14,23)(25,0){3}{\line(1,0){21}}
\multiput(12,5)(25,0){4}{\line(0,1){16.5}}
\put(64,4){\line(1,1){10}}
\put(64,24){\line(1,1){10}}
\put(89,4){\line(1,1){10}}
\put(89,24){\line(1,1){10}}
\multiput(74,31)(25,0){5}{$\circ$}
\multiput(74,11)(25,0){4}{$\circ$}
\multiput(78,34)(25,0){4}{\line(1,0){21}}
\multiput(78,14)(25,0){3}{\line(1,0){21}}
\multiput(76,16)(25,0){4}{\line(0,1){16.3}}
\end{picture}\]

   A shape  $\olambda = \langle \lambda |{\bullet/}{\circ}\rangle$ in $\Lambda_{G/P}$ is a triple $\langle\lambda^{(1)},\lambda^{(2)}|{\bullet}{/\circ}\rangle$, where $\lambda^{(1)}$ (respectively, $\lambda^{(2)}$) is the Young diagram, in French notation, for the ``bottom'' (respectively, ``top'') $2\times \left(\frac{|\Lambda_{G/P}|-1}{4}\right)$ rectangle, and ${\bullet}{/\circ}$ indicates if $\olambda$ is on or off.
    For example,
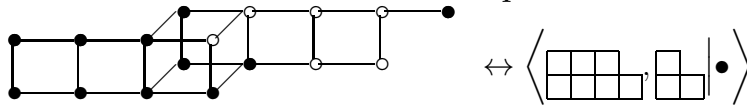
\[\begin{picture}(300,30)
\put(190,10){$\leftrightarrow\left\langle\tableau{{\ }&{\ }&{\ }\\{\ }&{\ }&{\ }&{\ }},
\tableau{{\ }\\{\ }&{\ }}\Big{|} \bullet\right\rangle$}
\multiput(10,0)(25,0){4}{$\bullet$}
\multiput(10,20)(25,0){4}{$\circ$}
\multiput(10,20)(25,0){3}{$\bullet$}
\multiput(14,3)(25,0){3}{\line(1,0){21}}
\multiput(14,23)(25,0){3}{\line(1,0){21}}
\multiput(12,5)(25,0){4}{\line(0,1){16.5}}
\put(64,4){\line(1,1){10}}
\put(64,24){\line(1,1){10}}
\put(89,4){\line(1,1){10}}
\put(89,24){\line(1,1){10}}
\multiput(74,31)(25,0){5}{$\circ$}
\multiput(74,31)(25,0){1}{$\bullet$}
\put(174,31){$\bullet$}
\multiput(74,11)(25,0){4}{$\circ$}
\multiput(74,11)(25,0){2}{$\bullet$}
\multiput(78,34)(25,0){4}{\line(1,0){21}}
\multiput(78,14)(25,0){3}{\line(1,0){21}}
\multiput(76,16)(25,0){4}{\line(0,1){16.3}}
\end{picture}\]
We mainly use a different description of $\olambda$
that is more convenient for comparisons with Section~1.3. Define $\pi(\lambda)=\lambda^{(1)}+\lambda^{(2)}:=(\lambda_1,\lambda_2)$, a partition inside the $2\times \left(\frac{|\Lambda_{G/P}|-1}{2}\right)$
rectangle. Consider an auxiliary poset $\Lambda_{OG(2,2n)}'$, a ``planarization'' of $\Lambda_{OG(2,2n)}$:
\[\begin{picture}(390,48)
\multiput(10,15)(25,0){4}{$\circ$}
\multiput(10,35)(25,0){4}{$\circ$}
\multiput(14,18)(25,0){3}{\line(1,0){21}}
\multiput(14,38)(25,0){3}{\line(1,0){21}}
\multiput(12,20)(25,0){4}{\line(0,1){16.5}}
\put(64,19){\line(1,1){10}}
\put(64,39){\line(1,1){10}}
\put(89,19){\line(1,1){10}}
\put(89,39){\line(1,1){10}}
\multiput(74,46)(25,0){5}{$\bullet$}
\multiput(74,26)(25,0){4}{$\bullet$}
\multiput(78,49)(25,0){4}{\line(1,0){21}}
\multiput(78,29)(25,0){3}{\line(1,0){21}}
\multiput(76,31)(25,0){4}{\line(0,1){16.3}}

\put(175,25){$\mapsto$}
\multiput(200,15)(25,0){8}{$\circ$}
\multiput(200,35)(25,0){9}{$\circ$}
\multiput(300,35)(25,0){5}{$\bullet$}
\multiput(300,15)(25,0){4}{$\bullet$}
\multiput(204,18)(25,0){7}{\line(1,0){21}}
\multiput(204,38)(25,0){8}{\line(1,0){21}}
\multiput(202,20)(25,0){8}{\line(0,1){16.5}}
\put(50,0){$\Lambda_{OG(2,2n)}$}
\put(280,0){$\Lambda_{OG(2,2n)}'$}
\end{picture}
\]
In the above figure, we have marked the roots of the ``top layer'' for emphasis.

Shapes of $\Lambda_{OG(2,2n)}'$ are ${\overline\kappa}=\langle\kappa|{\bullet}{/\circ}\rangle$ where $\kappa$ is a partition contained in a $2\times \left(\frac{|\Lambda_{G/P}|-1}{2}\right)$ rectangle and ${\bullet}{/\circ}$ indicates use of the adjoint root in $\Lambda_{OG(2,2n)}'$. Let ${\mathbb Y}_{OG(2,2n)}'$ be the set of shapes of $\Lambda_{OG(2,2n)}'$. Extend $\pi$ to a map 
\[\Pi:{\mathbb Y}_{OG(2,2n)}\to {\mathbb Y}_{OG(2,2n)}'\]
by defining $\Pi(\olambda)=\langle\pi(\lambda)|\bullet\rangle$ if $\olambda$ is on, and $\Pi(\olambda)=\langle\pi(\lambda)|\circ\rangle$ otherwise. 

The map $\Pi$ is either $1:1$ or $2:1$. In the former case, we identify $\okappa$ and $\Pi^{-1}(\okappa)$. 
In the latter case,
$\Pi^{-1}(\okappa)=\{\okappa^{\uparrow},\okappa^{\downarrow}\}$ and we call $\okappa$ {\bf ambiguous}.
Call $\okappa^{\uparrow}$ and $\okappa^{\downarrow}$ {\bf charged}. If $\okappa$ is on (respectively, off),
let $\okappa^{\downarrow}$ be the shape such that the second part (respectively, first part) 
of the Young diagram $(\pi^{-1}(\kappa))^{(2)}$ is zero; let $\okappa^{\uparrow}$ be the other one. Thus in Example~\ref{exa:updown}
below, $\olambda$ is up and $\omu$ is down.


We need three more notions to state our theorem.
First, for $\okappa \in {\mathbb Y}_{OG(2,2n)}'$, let ${\rm fsh}(\okappa)$ be the number of {\bf fake short roots} used by $\okappa$, i.e., the number of roots in the $(n-2)$-th column used by $\okappa$. The one exception is that we need \[{\rm fsh}(\langle n-2,n-2 |\circ\rangle)=1.\] For $\onu \in {\mathbb Y}_{OG(2,2n)}$, let ${\rm fsh}(\onu)$ denote ${\rm fsh}(\Pi(\onu))$. 
Second, two charged shapes $\olambda$ and $\omu$ {\bf match} if their arrows match and are {\bf opposite} otherwise. Third, let
\begin{equation}
\label{eqn:eta}
\eta_{\olambda,\omu}=
\begin{cases}
2 & \text{if $\olambda$, $\omu$ are charged and match and $n$ is even;}\\
2 & \text{if $\olambda$, $\omu$ are charged and opposite and $n$ is odd;}\\
1 & \text{if $\olambda$ or $\omu$ is neutral;}\\
0 & \text{otherwise}
\end{cases}
\end{equation}

Say $\sigma_{\langle \nu|\bullet/\circ\rangle}$, $\sigma_{\langle \nu^{\star}|\bullet\rangle}$ or $\sigma_{\langle\nu_{\star}|\bullet\rangle}$ is zero if $\nu$, $\nu^{\star}$ or
$\nu_{\star}$ is not a partition in $2\times \left(\frac{|\Lambda_{G/P}|-1}{2}\right)$.

\begin{Theorem}
\label{Thm:Dintro}
If either $\pi(\lambda)$ or $\pi(\mu)$ equals $(j,0)$ (for some $0\leq j\leq \frac{|\Lambda_{G/P}|-1}{2}$) then 
the Schubert expansion of 
$\sigma_{\olambda}\cdot\sigma_{\omu}\in H^{\star}(OG(2,2n))$ is obtained by the Pieri rule of 
\cite{BKT:Inventiones} (reformulated in Section~5.1). 

Otherwise, compute
\begin{equation}
\label{eqn:exp'}
\sigma_{\Pi(\olambda)}\cdot \sigma_{\Pi(\omu)}=\sum_{\nu\subseteq \left(\frac{|\Lambda_{G/P}|+1}{2},\frac{|\Lambda_{G/P}|-1}{2}\right)}
C_{\pi(\lambda),\pi(\mu)}^{\nu}{\mathbb A}_{\olambda,\omu}(\nu).
\end{equation}
\begin{itemize}
\item[(i)] Replace any term $\sigma_{\okappa}$ that has $\kappa_1=\frac{|\Lambda_{G/P}|-1}{2}$ by
$\eta_{\olambda,\omu}\sigma_{\okappa}$
\item[(ii)] Next, replace each $\sigma_{\okappa}$ by $2^{{\rm fsh}(\okappa)-{\rm fsh}(\olambda)-{\rm fsh}(\omu)}\sigma_{\okappa}$
\item[(iii)] Finally, for any ambiguous $\okappa$ replace $\sigma_{\okappa}$ by $\frac{1}{2}(\sigma_{\okappa^{\uparrow}}+\sigma_{\okappa^{\downarrow}})$
\end{itemize}
The result is a provably integral, and manifestly nonnegative, Schubert basis expansion, which equals
$\sigma_{\olambda}\cdot\sigma_{\omu} \in H^{\star}(OG(2,2n))$. 
\end{Theorem}

Integrality is not manifest due to (ii) and (iii); it is proved in Section~5.7. 
Rule (i) extends a parity dependency for even-dimensional quadrics, described in \cite{Thomas.Yong:comin}. 
The point is that the ``double tailed diamond'' which is $\Lambda_{{\mathbb Q}^{2n-4}}$ sits as
a ``side'' of $\Lambda_{OG(2,2n)}$. 
Rule (ii) is analogous to our rule for $OG(2,2n+1)$. Rule (iii) describes how to ``disambiguate''. 

In Section~5.1, we give a version of Theorem~\ref{Thm:Dintro} that includes multiplication 
in the ``Pieri'' case. That statement is more complicated but is in fact the one we prove. 

\begin{Example}
\label{exa:updown}
We wish to compute $\sigma_{\olambda}\cdot\sigma_{\omu}\in H^{\star}(OG(2,12))$ where:
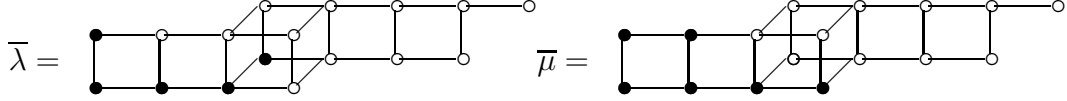
\[
\begin{picture}(400,31)
\put(-10,10){$\olambda=$}
\multiput(20,0)(25,0){4}{$\circ$}
\multiput(20,0)(25,0){3}{$\bullet$}
\multiput(20,20)(25,0){4}{$\circ$}
\multiput(20,20)(25,0){1}{$\bullet$}
\multiput(24,3)(25,0){3}{\line(1,0){21}}
\multiput(24,23)(25,0){3}{\line(1,0){21}}
\multiput(22,5)(25,0){4}{\line(0,1){16.5}}
\put(74,4){\line(1,1){10}}
\put(74,24){\line(1,1){10}}
\put(99,4){\line(1,1){10}}
\put(99,24){\line(1,1){10}}
\multiput(84,31)(25,0){5}{$\circ$}
\multiput(84,11)(25,0){4}{$\circ$}
\multiput(84,11)(25,0){1}{$\bullet$}
\multiput(88,34)(25,0){4}{\line(1,0){21}}
\multiput(88,14)(25,0){3}{\line(1,0){21}}
\multiput(86,16)(25,0){4}{\line(0,1){16.3}}

\put(190,10){$\omu=$}
\multiput(220,0)(25,0){4}{$\circ$}
\multiput(220,0)(25,0){4}{$\bullet$}
\multiput(220,20)(25,0){4}{$\circ$}
\multiput(220,20)(25,0){2}{$\bullet$}
\multiput(224,3)(25,0){3}{\line(1,0){21}}
\multiput(224,23)(25,0){3}{\line(1,0){21}}
\multiput(222,5)(25,0){4}{\line(0,1){16.5}}
\put(274,4){\line(1,1){10}}
\put(274,24){\line(1,1){10}}
\put(299,4){\line(1,1){10}}
\put(299,24){\line(1,1){10}}
\multiput(284,31)(25,0){5}{$\circ$}
\multiput(284,11)(25,0){4}{$\circ$}
\multiput(284,11)(25,0){1}{$\circ$}
\multiput(288,34)(25,0){4}{\line(1,0){21}}
\multiput(288,14)(25,0){3}{\line(1,0){21}}
\multiput(286,16)(25,0){4}{\line(0,1){16.3}}
\end{picture}\]
Both of these shapes are charged. Here $\pi(\lambda)=(4,1)$ and
$\pi(\mu)=(4,2)$. 

The $\nu\subseteq \left(\frac{|\Lambda_{G/P}|+1}{2},\frac{|\Lambda_{G/P}|-1}{2}\right)=(9,8)$ such that $C_{\pi(\lambda),\pi(\mu)}^{\nu}=1$ are
$(8,3), (7,4)$ and $(6,5)$. All other $\nu$ have $C_{\pi(\lambda),\pi(\mu)}^{\nu}=0$.
Hence,
\begin{eqnarray}\nonumber
\sigma_{\Pi(\olambda)}\cdot\sigma_{\Pi(\omu)} & = &  {\mathbb A}_{\olambda,\omu}(8,3)  +  {\mathbb A}_{\olambda,\omu}(7,4)  + 
{\mathbb A}_{\olambda,\omu}(6,5)\\ \nonumber
& = & (\langle 7,3|\bullet\rangle+\langle 8,2|\bullet\rangle)  +  (\langle 6,4|\bullet\rangle+\langle 7,3|\bullet\rangle)  + 
(\langle 5,5|\bullet\rangle+\langle 6,4|\bullet\rangle)\\ \nonumber
& = & \langle 8,2|\bullet\rangle + 2\langle 7,3|\bullet\rangle + 2\langle 6,4|\bullet\rangle + \langle 5,5|\bullet\rangle \\ \nonumber
& \mapsto & \!\!\! 0\langle 8,2|\bullet\rangle + 2\langle 7,3|\bullet\rangle + 2\langle 6,4|\bullet\rangle + \langle 5,5|\bullet\rangle \mbox{\ \ \ \  (by (i) and $\eta_{\olambda,\omu}=0$)}\\ \nonumber
& \mapsto & \ \ \ \ \ \ \ \ \ \ \ \ \ \ \ \ \ \ \   \langle 7,3|\bullet\rangle + 2\langle 6,4|\bullet\rangle + \langle 5,5|\bullet\rangle \mbox{ \ \ \   (by (ii) and ${\rm fsh}(\olambda)={\rm fsh}(\omu)=1$)}
\end{eqnarray}
Finally, (iii) applies to the ambiguous shape $\langle 6,4|\bullet\rangle$, so:
\[\sigma_{\olambda}\cdot\sigma_{\omu}=\langle 7,3|\bullet\rangle + (\langle 6,4|\bullet\rangle^{\uparrow}
+\langle 6,4|\bullet\rangle^{\downarrow})+ \langle 5,5|\bullet\rangle.\]
Each step is nonnegative and integral, in agreement with our theorem. \qed
\end{Example}

\begin{Corollary}\label{cor:Dintegral}
$\left\{C_{\olambda,\omu}^{\onu}(OG(2,2n))\right\}= \{0,1,2,4,8\}$.
\end{Corollary}

We make the following identifications; cf. (\ref{eqn:identify}):
\begin{equation}\nonumber
\label{eqn:identify2}
\mbox{$\Pi(\olambda)=\langle\lambda_1,\lambda_2|\circ\rangle$ with $(\lambda_1,\lambda_2,0)\in {\mathbb Z}^3$ and
$\Pi(\olambda)=\langle\lambda_1,\lambda_2|\bullet \rangle$ with $(\lambda_1,\lambda_2,1)\in {\mathbb Z}^3$.}
\end{equation}
We can give an explicit criterion for nonzeroness:

\begin{Corollary}\label{cor:Dnonzero}
If either $\pi(\lambda)$ or $\pi(\mu)$ equals $(j,0)$ (for some $0\leq j\leq \frac{|\Lambda_{G/P}|-1}{2}$) then nonzeroness of $C_{\olambda,\omu}^{\onu}(OG(2,2n))$ is determined by the Pieri rule of 
\cite{BKT:Inventiones} (restated in Section~5.1).  

If $\nu_1=\frac{|\Lambda_{G/P}|-1}{2}$ then $C_{\olambda,\omu}^{\onu}(OG(2,2n))\neq 0$ if and only if
$\eta_{\olambda,\omu}\neq 0$ and the inequalities (\ref{eqn:horntype}) hold.

Otherwise, assume $(\lambda_1,\lambda_2),(\mu_1,\mu_2),(\nu_1,\nu_2)\subset 2\times \left(\frac{|\Lambda_{G/P}|-1}{2}\right)$ are partitions and $\lambda_3,\mu_3,\nu_3\in \{0,1\}$. 
Then $C_{\olambda,\omu}^{\onu}(OG(2,2n))\neq 0$ if and only if the inequalities
(\ref{eqn:horntype}) hold.
\end{Corollary}

\subsection{Organization}	
The main strategy employed is to 
prove that the rules of Theorems~\ref{Thm:BCintro} and~\ref{Thm:Dintro} define an associative ring. In the companion paper
\cite{Searles:13+}, it is shown that our rules agree with known Pieri rules \cite{BKT:Inventiones} (cf. \cite{Prag, PragD}). 
Together this implies the correctness of the rules; see \cite{Knutson.Tao.Woodward} for a similar
argument in a different case of Schubert calculus.

In Section~2, we discuss the remaining (co)adjoint varieties not
addressed by Theorems~\ref{Thm:BCintro} and~\ref{Thm:Dintro}. 
This includes the cases of exceptional Lie type as well as the remaining and straightforward cases of classical type. One of the latter
cases is $\ourflag$. While rules for this case are essentially
well-known (and easy to derive), we revisit it through the lens of
${\mathbb A}_{\olambda,\omu}$. The exceptional cases are
computationally studied using the results of
\cite{Chaput.Perrin, Chaput.Perrin:computations}. We end with the conclusion of the
proof of Theorem~\ref{thm:iff}.

In Sections 3, 4 and 5 we give similar associativity arguments in order
of increasing difficulty: $\ourflag, OG(2,2n+1)/LG(2,2n)$ and 
$OG(2,2n)$, respectively (the first case being mostly a light warmup for the other two). We also prove the stated corollaries. 

\section{Remaining (co)adjoint cases and proof of Theorem~\ref{thm:iff}}

\subsection{(Co)minuscule, (co)adjoint and quasi-(co)minuscule}
We now recall in what sense (co)minuscule $G/P$'s extend the (co)adjoint $G/P$'s. 
A dominant weight $\omega$ 
(associated to $P$) is {\bf minuscule} if for every $\alpha\in \Phi^+$
        we have $\langle\alpha^{\vee},\omega\rangle\leq 1$. Such a weight is {\bf quasi-minuscule}
        if for every $\alpha\in \Phi^{+}$ we have $\langle\alpha^{\vee},\omega\rangle\leq 2$,
        with equality only if $\alpha=\omega$. The quasi-minuscule weights that are not
        minuscule are precisely the coadjoint ones. A weight is cominuscule 
(respectively, adjoint and quasi-cominuscule) if it is minuscule (respectively, 
coadjoint and quasi-minuscule) for the dual root system.

Now, $G/P$ is an {\bf adjoint variety} if $P$ is the parabolic subgroup for an adjoint weight $\omega$.
Similarly, one defines minuscule, quasi-miniscule, cominuscule, co-adjoint and quasi-cominuscule varieties. 
The classification of adjoint $G/P$'s is given in Table 1; nodes associated to $P$ are marked, cf. \cite{Chaput.Perrin}.

Our analysis is made possible by rapid computation of all structure constants
in these cases using the presentation of the cohomology ring in \cite{Chaput.Perrin}, their
Giambelli-type formulas \cite{Chaput.Perrin:computations}, and
standard Gr\"obner basis techniques. Here is a table of $|{\mathbb Y}_{G/P}|$:

\begin{table}[h]
\begin{center}
\begin{tabular}{|l|l|l|l|l|l|l|l|l|}
\hline
\multicolumn{9}{|c|}{Adjoint}\\
\hline
$A_{n-1}$ & $B_n$ & $C_n$ & $D_n$ & $G_2$ & $F_4$ & $E_6$ & $E_7$ & $E_8$ \\ \hline
$2{n\choose 2}$ &  $4{n\choose 2}$  & $2n$  & $4{n\choose 2}$   & $6$   & $24$   & $72$  & $126$ & $240$ \\
\hline
\end{tabular} \ \ \ \
\begin{tabular}{|l|l|l|l|}
\hline
\multicolumn{4}{|c|}{Coadjoint}\\
\hline
$B_n$ & $C_n$ & $G_2$ & $F_4$ \\ \hline
$2n$ &  $4{n\choose 2}$  & $6$  & $24$  \\
\hline
\end{tabular}
\end{center}
\end{table}

\subsection{The exceptional types}
We begin with:

\noindent
{\bf Type $F_4$:} This adjoint node is node $4$ of the Dynkin diagram
while the coadjoint node is node $1$. First, we consider the adjoint case:
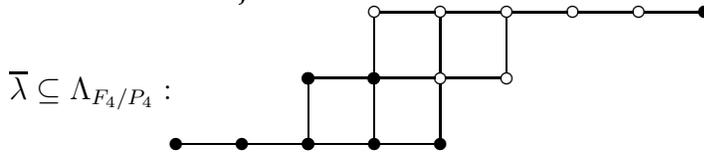
\[\begin{picture}(100,50)
\put(-100,20){$\olambda\subseteq \Lambda_{F_4/P_4}:$}
\multiput(10,0)(0,25){2}{$\bullet$}
\multiput(13,5)(0,25){1}{\line(0,1){21}}
\multiput(-40,0)(25,0){5}{$\bullet$}
\multiput(-35,3)(25,0){4}{\line(1,0){21}}
\multiput(10,25)(25,0){4}{$\circ$}
\multiput(10,25)(25,0){2}{$\bullet$}
\multiput(15,28)(25,0){3}{\line(1,0){21}}
\multiput(35,50)(25,0){6}{$\circ$}
\multiput(40,53)(25,0){5}{\line(1,0){21}}
\multiput(38,5)(0,25){2}{\line(0,1){21}}
\multiput(63,5)(0,25){2}{\line(0,1){21}}
\multiput(88,30)(0,25){1}{\line(0,1){21}}
\put(160,50){$\bullet$}
\end{picture}
\]
The short roots consist of the third root (from the left) in the bottom row,
all roots in the middle row, and the third root (from the left) in the top row.

Define the partition-like description of a shape in $\Lambda_{F_4/P_4}$ by associating $\langle\lambda|\circ\rangle$ with the vector
$(\lambda_1,\lambda_2,\lambda_3,0)\in {\mathbb Z}^4$ and $\langle\lambda|\bullet\rangle$ with $(\lambda_1,\lambda_2,\lambda_3,1)$.
Here $\lambda_1,\lambda_2$ and $\lambda_3$ are the number of roots used in the bottom, middle and top rows of $\Lambda_{F_4/P_4}$
respectively. Let $\lambda_4$ be the fourth coordinate. So for example, the displayed shape has associated vector $(\lambda_1,\lambda_2,\lambda_3,\lambda_4)=(5,2,0,1)$ and has three short roots.

\begin{Fact}
$C_{\olambda,\omu}^{\onu}(F_4/P_4)\neq 0$ if and only if 
\begin{eqnarray}\nonumber
|\onu| & = & |\olambda|+|\omu|  \\ \nonumber
\lambda_1+\lambda_4 & \leq & 6-\mu_3-\mu_4  \\ \nonumber
\lambda_2+\lambda_4 & \leq & 5-\mu_2-\mu_4  \\ \nonumber
\lambda_3+\lambda_4 & \leq & 6-\mu_1-\mu_4 \nonumber
\end{eqnarray}
\begin{eqnarray}\nonumber
\lambda_i+\lambda_4 & \leq & \nu_i+\nu_4  \ \ (\mbox{for $1\leq i\leq 3$})\\ \nonumber
\mu_i+\mu_4 & \leq & \nu_i+\nu_4  \ \ (\mbox{for $1\leq i\leq 3$})\\ \nonumber
\lambda_1+\mu_1-\nu_3  & \leq  & 9. \nonumber
\end{eqnarray}
\end{Fact}
The first five inequalities partly encode $\olambda\prec_{\rm Bruhat} \omu^{\vee}$ and $\olambda,\omu\prec_{\rm Bruhat} \onu$.

We do \emph{not} have an isomorphism between $\Lambda_{F_4/P_4}$ and $\Lambda_{F_4/P_1}$. In fact
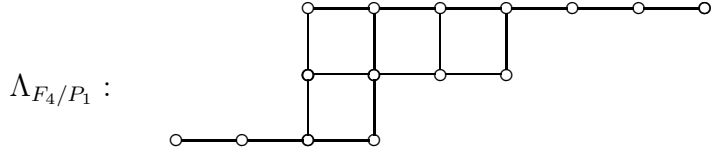
\[\begin{picture}(100,58)
\put(-100,20){$\Lambda_{F_4/P_1}:$}
\multiput(10,0)(0,25){2}{$\circ$}
\multiput(13,5)(0,25){2}{\line(0,1){21}}
\multiput(-40,0)(25,0){4}{$\circ$}
\multiput(-35,3)(25,0){3}{\line(1,0){21}}
\multiput(10,25)(25,0){4}{$\circ$}
\multiput(10,25)(25,0){2}{$\circ$}
\multiput(15,28)(25,0){3}{\line(1,0){21}}
\multiput(10,50)(25,0){7}{$\circ$}
\multiput(15,53)(25,0){6}{\line(1,0){21}}
\multiput(38,5)(0,25){2}{\line(0,1){21}}
\multiput(63,30)(0,25){1}{\line(0,1){21}}
\multiput(88,30)(0,25){1}{\line(0,1){21}}
\put(160,50){$\circ$}
\end{picture}
\]

However, there is still a natural correspondence of ${\mathbb Y}_{F_4/P_4}$ with ${\mathbb Y}_{F_4/P_1}$: given a reduced word $s_{i_1}\cdots s_{i_\ell}$ for a minimal coset representative of $F_4^{P_4}$, then 
$s_{5-i_1}\cdots s_{5-i_{\ell}}$ is a reduced word of a minimal coset representative of $F_4^{P_1}$. If $\olambda\in {\mathbb Y}_{F_4/P_4}$ is the RYD associated
to the first reduced word, we may declare it to be the RYD indexing the
Schubert class of $H^{\star}(F_4/P_1)$ associated to the second reduced word.
Thus when we write $C_{\olambda,\omu}^{\onu}(F_4/P_1)$ we refer to the proxy shapes
from ${\mathbb Y}_{F_4/P_4}$. (Similar reasoning works in the other types as well, allowing us to
index coadjoint Schubert varieties with proxy adjoint shapes.)

\begin{Fact}
$C_{\olambda,\omu}^{\onu}(F_4/P_1)=2^{{\rm sh}(\olambda)+{\rm sh}(\omu)-{\rm sh}(\onu)}C_{\olambda,\omu}^{\onu}(F_4/P_4)$.

All numbers below $8$ except $7$ appear as Schubert structure constants for (adjoint) $F_4/P_4$.
For (coadjoint) $F_4/P_1$ it is all numbers below $6$.
\end{Fact}

\smallskip
\begin{center}
\begin{table}[b]
\begin{tabular}{||c|c|c||}\hline
\emph{Root system} & \emph{Dynkin Diagram} & \emph{Nomenclature (if any)}\\\hline \hline
$A_{n-1}$ &
\setlength{\unitlength}{3mm}

\begin{picture}(11,3)
\multiput(0,1.5)(2,0){6}{$\circ$}
\multiput(0.55,1.85)(2,0){5}{\line(1,0){1.55}}
\put(0,1.5){$\bullet$}
\put(10,1.5){$\bullet$}
\put(0,0){$1$}
\put(2,0){$2$}
\put(3.5,0){$\cdots$}
\put(5.9,0){$k$}
\put(7.1,0){$\cdots$}
\put(9.1,0){$n\!-\!1$}
\end{picture}
& Point-hyperplane incidence in ${\mathbb P}^{n-1}$; $\ourflag$\\ \hline
$B_n$ &
\setlength{\unitlength}{3mm} \begin{picture}(11,3)
\multiput(0,1.5)(2,0){6}{$\circ$}
\multiput(0.55,1.85)(2,0){4}{\line(1,0){1.55}}
\multiput(8.55,1.75)(0,.2){2}{\line(1,0){1.55}} \put(8.75,1.52){$>$}
\put(2,1.5){$\bullet$}
\put(0,0){$1$}
\put(2,0){$2$}
\put(4,0){$\cdots$}
\put(7,0){$\cdots$}
\put(10,0){$n$}\end{picture}
& Odd orthogonal Grassmannian; $OG(2,2n+1)$\\ \hline
$C_n, n\geq 3$ &
\setlength{\unitlength}{3mm} \begin{picture}(11,3)
\multiput(0,1.5)(2,0){6}{$\circ$}
\multiput(0.55,1.85)(2,0){4}{\line(1,0){1.55}}
\multiput(8.55,1.75)(0,.2){2}{\line(1,0){1.55}} \put(8.85,1.53){$<$}
\put(0,1.5){$\bullet$}
\put(0,0){$1$}
\put(2,0){$2$}
\put(4,0){$\cdots$}
\put(7,0){$\cdots$}
\put(10,0){$n$}
\end{picture}
& Odd Projective space; ${\mathbb P}^{2n-1}$\\ \hline
$D_n, n\geq 4$ &
\setlength{\unitlength}{2.9mm} \begin{picture}(11,3.5)
\multiput(0,1.6)(2,0){5}{$\circ$}
\multiput(0.55,2)(2,0){4}{\line(1,0){1.55}}
\put(8.5,1.95){\line(2,-1){1.55}}
\put(8.5,1.95){\line(2,1){1.55}}
\put(10,2.5){$\circ$}
\put(10,0.7){$\circ$}
\put(2,1.6){$\bullet$}
\put(0,0){$1$}
\put(2,0){$2$}
\put(4,0){$\cdots$}
\put(7,0){$\cdots$}
\put(9.1,0){$n\!-\!1$}
\put(11, 2.3){$n$}\end{picture}

& Even orthogonal Grassmannian; $OG(2,2n)$
\\\hline
$E_6$ &
\setlength{\unitlength}{3mm}
\begin{picture}(9,3.6)
\multiput(0,0.5)(2,0){5}{$\circ$}
\multiput(0.55,0.95)(2,0){4}{\line(1,0){1.6}}
\put(4,2.6){$\bullet$}
\put(4,2.6){$\circ$}
\put(4.35,1.2){\line(0,1){1.5}}
\put(0,-.6){$1$}
\put(2,-0.6){$3$}
\put(4,-.6){$4$}
\put(6,-.6){$5$}
\put(5,2.5){$2$}
\put(8,-.6){$6$}
\end{picture}
& $E_6/P_2$
\\ \hline
$E_7$ &
\setlength{\unitlength}{3mm}
\begin{picture}(11,4)
\put(10,0.9){$\bullet$}
\multiput(2,0.9)(2,0){5}{$\circ$}
\multiput(0.55,1.35)(2,0){5}{\line(1,0){1.6}}
\put(10,0.9){$\circ$}
\put(6,3){$\circ$}
\put(0,0.9){$\circ$}
\put(6.35,1.6){\line(0,1){1.5}}
\put(0,-.2){$1$}
\put(2,-0.2){$3$}
\put(4,-.2){$4$}
\put(6,-.2){$5$}
\put(7,2.9){$2$}
\put(8,-.2){$6$}
\put(10,-.2){$7$}
\end{picture}
& $E_7/P_7$
\\ \hline
$E_8$ &
\setlength{\unitlength}{3mm}
\begin{picture}(15,4)
\put(12,0.9){$\bullet$}
\multiput(2,0.9)(2,0){6}{$\circ$}
\multiput(0.55,1.35)(2,0){6}{\line(1,0){1.6}}
\put(10,0.9){$\circ$}
\put(4,3){$\circ$}
\put(0,0.9){$\circ$}
\put(4.35,1.6){\line(0,1){1.5}}
\put(0,-.2){$1$}
\put(2,-0.2){$3$}
\put(4,-.2){$4$}
\put(6,-.2){$5$}
\put(5,2.9){$2$}
\put(8,-.2){$6$}
\put(10,-.2){$7$}
\put(12,-.2){$8$}
\end{picture}
& $E_8/P_8$
\\ \hline
$F_4$ &
\setlength{\unitlength}{3mm}
\begin{picture}(11,4)
\put(6,0.9){$\bullet$}
\put(0,0.9){$\circ$}
\multiput(2,0.9)(2,0){3}{$\circ$}
\multiput(0.55,1.35)(2,0){3}{\line(1,0){1.6}}
\put(0,-.2){$1$}
\put(2,-0.2){$2$}
\put(4,-.2){$3$}
\put(6,-.2){$4$}
\put(2.3,1.1){\line(1,0){1.8}}
\put(2,1){\line(-1,1){1.1}}
\put(2.9,.8){$<$}
\end{picture}
& $F_4/P_4$
\\ \hline
$G_2$ &
\setlength{\unitlength}{3mm}
\begin{picture}(6,4)
\put(0,0.9){$\circ$}
\put(2,0.9){$\bullet$}
\put(0,-.2){$1$}
\put(2,-0.2){$2$}
\put(.5,1.3){\line(1,0){1.5}}
\put(.5,1.1){\line(1,0){1.5}}
\put(.5,1.5){\line(1,0){1.5}}
\put(.6,.9){$<$}
\end{picture}
& $G_2/P_2$ \\ \hline
\end{tabular}
\caption{
Classification of adjoint $G/P$'s}
\end{table}
\end{center}
%
\vspace{-.4in} \noindent
{\bf Type $G_2$:} Both the adjoint $\Lambda_{G_2/P_2}$ and  coadjoint $\Lambda_{G_2/P_1}$ are a chain of five elements, with the maximal element being the adjoint root. Both $\mathbb{Y}_{G_2/P_2}$ and $\mathbb{Y}_{G_2/P_1}$ have six elements, one each of size $k$ for $0\le k \le 5$.
We identify each element of $\mathbb{Y}_{G_2/P_1}$ with the element of $\mathbb{Y}_{G_2/P_2}$ having the same size, and compute using the elements of $\mathbb{Y}_{G_2/P_2}$.
The short roots of $\Lambda_{G_2/P_2}$ are the middle two nonadjoint roots.

\begin{Fact}
If $(\olambda,\omu,\onu)\neq (\langle\lambda|\circ\rangle,\langle\mu|\circ\rangle,\langle\nu|\bullet\rangle)$, then $C_{\olambda,\omu}^{\onu}(G_2/P_1) = C_{\lambda,\mu}^{\nu}(Gr_1(\mathbb{C}^5))$. Otherwise $C_{\langle\lambda|\circ\rangle, \langle \mu| \circ\rangle}^{\langle\nu|\bullet\rangle}(G_2/P_1) = 2\cdot C_{\lambda,\mu}^{\nu^+}(Gr_1(\mathbb{C}^6))$, where $\nu^+$ is $\nu$ with one additional root. Also,
\[C_{\olambda,\omu}^{\onu}(G_2/P_2)= 3^{{\rm sh}(\onu)-{\rm sh}(\olambda)-{\rm sh}(\omu)}C_{\olambda,\omu}^{\onu}(G_2/P_1).\] 
Moreover, $C_{\olambda,\omu}^{\onu}(G_2/P_2)\neq 0$ if and only if $|\onu|=|\olambda|+|\omu|$. Finally, $C_{\olambda,\omu}^{\onu}(G_2/P_2)\in \{0,1,2,3\}$ and $C_{\olambda,\omu}^{\onu}(G_2/P_1)\in\{0,1,2\}$.
\end{Fact}

\noindent
{\bf Type $E_{n}$ series:}  Below is an example of 
$\olambda\subseteq \Lambda_{E_6/P_2}$ and $\olambda\subseteq \Lambda_{E_7/P_7}$.
In both cases the adjoint root is the rightmost root.

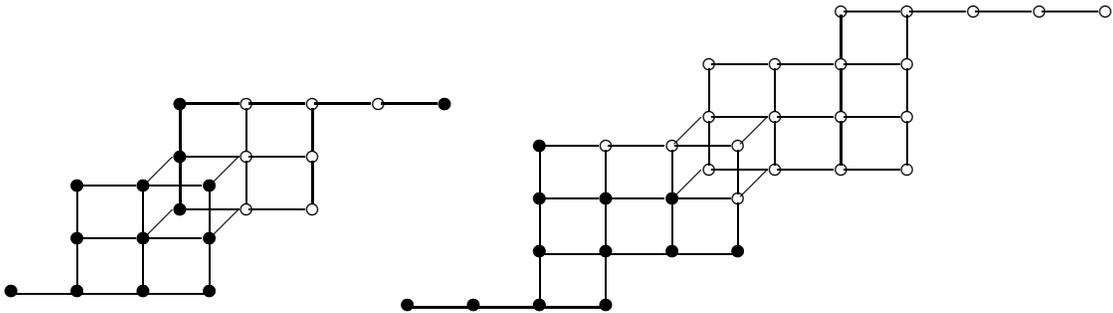
\[\begin{picture}(380,100)
\multiput(-15,0)(25,0){4}{$\circ$}
\multiput(-15,0)(25,0){4}{$\bullet$}
\multiput(-10,2)(25,0){3}{\line(1,0){21}}
\multiput(13,5)(25,0){3}{\line(0,1){16.5}}

\multiput(10,20)(25,0){3}{$\circ$}
\multiput(10,20)(25,0){3}{$\bullet$}

\multiput(10,40)(25,0){3}{$\circ$}
\multiput(10,40)(25,0){3}{$\bullet$}
\multiput(49,71)(25,0){5}{$\circ$}
\multiput(149,71)(25,0){1}{$\bullet$}
\multiput(49,71)(25,0){1}{$\bullet$}

\multiput(49,51)(25,0){1}{$\bullet$}
\multiput(49,31)(25,0){1}{$\bullet$}
\multiput(53,74)(25,0){4}{\line(1,0){21}}
\multiput(52,56)(25,0){3}{\line(0,1){16.3}}

\multiput(14,23)(25,0){2}{\line(1,0){21}}
\multiput(14,43)(25,0){2}{\line(1,0){21}}
\multiput(13,25)(25,0){3}{\line(0,1){16.5}}
\put(39,24){\line(1,1){10}}
\put(39,44){\line(1,1){10}}
\put(64,24){\line(1,1){10}}
\put(64,44){\line(1,1){10}}
\multiput(49,51)(25,0){3}{$\circ$}
\multiput(49,31)(25,0){3}{$\circ$}
\multiput(53,54)(25,0){2}{\line(1,0){21}}
\multiput(53,34)(25,0){2}{\line(1,0){21}}
\multiput(52,36)(25,0){3}{\line(0,1){16.3}}


\multiput(135,-5)(25,0){4}{$\circ$}
\multiput(140,-3)(25,0){3}{\line(1,0){21}}
\multiput(188,0)(25,0){2}{\line(0,1){16.5}}
\multiput(135,-5)(25,0){4}{$\bullet$}

\multiput(185,15)(25,0){4}{$\circ$}
\multiput(190,17)(25,0){3}{\line(1,0){21}}
\multiput(188,20)(25,0){4}{\line(0,1){16.5}}
\multiput(185,15)(25,0){4}{$\bullet$}

\multiput(185,35)(25,0){4}{$\circ$}
\multiput(185,55)(25,0){4}{$\circ$}

\multiput(185,35)(25,0){3}{$\bullet$}
\multiput(185,55)(25,0){1}{$\bullet$}

\multiput(249,86)(25,0){4}{$\circ$}
\multiput(253,89)(25,0){3}{\line(1,0){21}}
\multiput(252,71)(25,0){4}{\line(0,1){16.3}}

\multiput(299,106)(25,0){5}{$\circ$}
\multiput(303,109)(25,0){4}{\line(1,0){21}}
\multiput(302,91)(25,0){2}{\line(0,1){16.3}}

\multiput(189,38)(25,0){3}{\line(1,0){21}}
\multiput(189,58)(25,0){3}{\line(1,0){21}}
\multiput(188,40)(25,0){4}{\line(0,1){16.3}}

\put(239,39){\line(1,1){10}}
\put(239,59){\line(1,1){10}}
\put(264,39){\line(1,1){10}}
\put(264,59){\line(1,1){10}}

\multiput(249,66)(25,0){4}{$\circ$}
\multiput(249,46)(25,0){4}{$\circ$}

\multiput(253,69)(25,0){3}{\line(1,0){21}}
\multiput(253,49)(25,0){3}{\line(1,0){21}}
\multiput(252,51)(25,0){4}{\line(0,1){16.3}}
\end{picture}\]


\begin{Fact}
All integers in $[0,7]$ appear as structure constants for $E_6/P_2$.
\end{Fact}

\begin{Fact} The only integer in $[0,33]$ not appearing as a structure constant for $E_7/P_7$
is $25$.
\end{Fact}
For
 $\Lambda_{E_8/P_8}$, the adjoint root is the rightmost one. 
An example of $\olambda \subseteq \Lambda_{E_8/P_8}$ is:
 
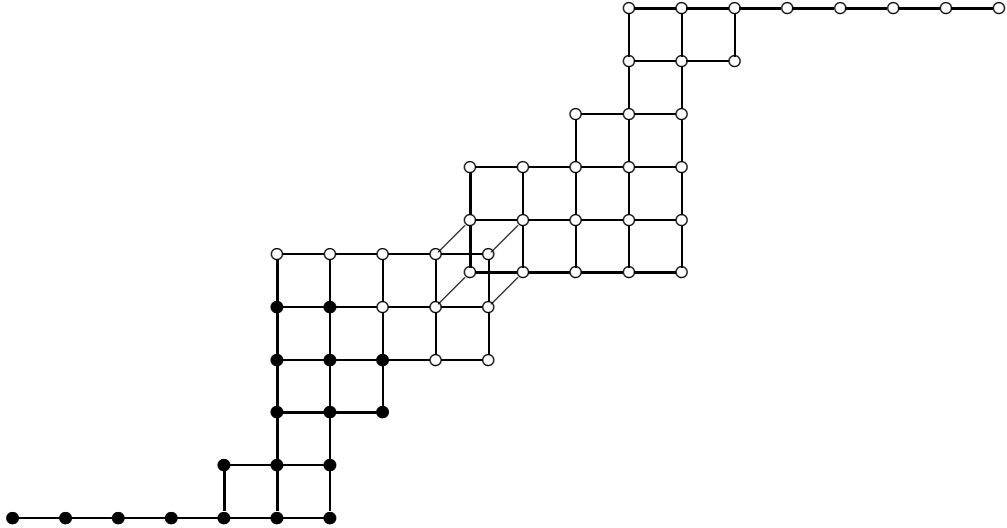
\[\begin{picture}(350,200)
\multiput(-10,5)(20,0){7}{$\circ$}
\multiput(-5,8)(20,0){6}{\line(1,0){15.5}}
\multiput(73,11)(20,0){3}{\line(0,1){15}}
\multiput(70,25)(20,0){3}{$\circ$}
\multiput(75,28)(20,0){2}{\line(1,0){15.5}}
\multiput(90,45)(20,0){3}{$\circ$}
\multiput(95,48)(20,0){2}{\line(1,0){15.5}}
\multiput(93,30)(20,0){2}{\line(0,1){16}}

\multiput(-10,5)(20,0){7}{$\bullet$}
\multiput(70,25)(20,0){3}{$\bullet$}
\multiput(90,45)(20,0){3}{$\bullet$}
\multiput(90,65)(20,0){3}{$\bullet$}
\multiput(90,85)(20,0){2}{$\bullet$}

\multiput(90,65)(20,0){5}{$\circ$}
\multiput(95,68)(20,0){4}{\line(1,0){15.5}}
\multiput(93,50)(20,0){3}{\line(0,1){16}}

\multiput(90,85)(20,0){5}{$\circ$}
\multiput(95,88)(20,0){4}{\line(1,0){15.5}}
\multiput(93,70)(20,0){5}{\line(0,1){16}}

\multiput(90,105)(20,0){5}{$\circ$}
\multiput(95,108)(20,0){4}{\line(1,0){15.5}}
\multiput(93,90)(20,0){5}{\line(0,1){16}}

\put(154,89){\line(1,1){10}}
\put(174,89){\line(1,1){10}}

\multiput(163,98)(20,0){5}{$\circ$}
\multiput(168,101)(20,0){4}{\line(1,0){15.5}}
\multiput(166,103)(20,0){5}{\line(0,1){16}}

\put(154,109){\line(1,1){10}}
\put(174,109){\line(1,1){10}}

\multiput(163,118)(20,0){5}{$\circ$}
\multiput(168,121)(20,0){4}{\line(1,0){15.5}}
\multiput(166,123)(20,0){5}{\line(0,1){16}}

\multiput(163,138)(20,0){5}{$\circ$}
\multiput(168,141)(20,0){4}{\line(1,0){15.5}}

\multiput(203,158)(20,0){3}{$\circ$}
\multiput(208,161)(20,0){2}{\line(1,0){15.5}}
\multiput(206,143)(20,0){3}{\line(0,1){16}}

\multiput(223,178)(20,0){3}{$\circ$}
\multiput(228,181)(20,0){2}{\line(1,0){15.5}}
\multiput(226,163)(20,0){2}{\line(0,1){16}}

\multiput(223,198)(20,0){8}{$\circ$}
\multiput(228,201)(20,0){7}{\line(1,0){15.5}}
\multiput(226,183)(20,0){3}{\line(0,1){16}}
\end{picture}\]

\begin{Fact}
In the range $[0,975]$, $469$ of the integers appear as structure constants for $E_8/P_8$. The smallest missing value is $221$.
\end{Fact}

Our partition-like description for $G/P=E_6/P_2$ identifies
a shape $\olambda$ with a vector
in ${\mathbb Z}^7$. The first three coordinates describe the number of roots
used in each row on the ``bottom layer'' of $\Lambda_{E_6/P_2}$, 
the second three similarly describe the second layer, and the last coordinate indicates use of the adjoint root. For example, the displayed shape for $E_6/P_2$ above is encoded as $(4,3,3,1,1,1,1)$.

\begin{Fact}
Let $\olambda = (4,0,0,0,0,0,0)$ and $\omu = (3,2,0,1,0,0,0)$. Then \[C_{\olambda,\omu}^{\onu}(E_6/P_2)\neq 0 \mbox{ \ for $\onu \in \{(4,3,0,3,0,0,0), (4,3,2,1,0,0,0)\}$ and}\] 
\[C_{\olambda,\omu}^{\onu}(E_6/P_2)= 0 \mbox{\  for $\onu \in \{(4,3,1,2,0,0,0), (4,3,3,0,0,0,0)\}$.}\]
\end{Fact}

 This yields four collinear triples, alternating between $S^{\tt nonzero}(E_6/P_2)$ and $S^{\tt zero}(E_6/P_2)$. This implies 
these embeddings of $S^{\tt nonzero}(E_6/P_2)$ and $S^{\tt zero}(E_6/P_2)$ are not polytopal.

For $G/P=E_7/P_1$ our partition-like description identifies shapes $\olambda$ with vectors in ${\mathbb Z}^9$. The first four coordinates describe the number of roots
used in each row 
 on the ``bottom layer'' of $\Lambda_{E_7/P_1}$, 
the second four similarly describe the second layer, and the last coordinate indicates use of the adjoint root. Thus, for example the $E_7/P_1$ shape above is $(4,4,3,1,0,0,0,0,0)$.

\begin{Fact}
Let $\olambda=\omu=(4,4,0,0,0,0,0,0,0)$. Then 
\[C_{\olambda,\omu}^{\onu}(E_7/P_1)\neq 0 \mbox{ \ for $\onu \in \{(4,4,4,0,4,0,0,0,0,0), (4,4,4,2,2,0,0,0,0)\}$ and}\] 
\[C_{\olambda,\omu}^{\onu}(E_7/P_1)= 0 \mbox{\ for $\onu \in \{(4,4,4,1,3,0,0,0,0), (4,4,4,3,1,0,0,0,0)\}$.}\] 
\end{Fact}
Thus the embeddings of $S^{\tt nonzero}(E_7/P_1)$ and $S^{\tt zero}(E_7/P_1)$ are not polytopal.

For $G/P=E_8/P_8$, we identify shapes $\olambda$ with vectors in ${\mathbb Z}^{13}$. The first six coordinates describe the number of roots
used in each row 
 on the ``bottom layer'' of $\Lambda_{E_8/P_8}$, 
the second six describe the second layer, and the last coordinate indicates use of the adjoint root.

\begin{Fact}
Let $\onu=( 7,3,3,5,5,0,5,0,0,0,0,0, 0)$. Then
\[C_{(1,0,0,0,0,0,0,0,0,0,0,0,0),(7,3,3,5,5,0,4,0,0,0,0,0,0)}^{\onu}(E_8/P_8)\neq 0,\] 
\[C_{(5,0,0,0,0,0,0,0,0,0,0,0,0),(7,3,3,5,5,0,0,0,0,0,0,0,0)}^{\onu}(E_8/P_8)\neq 0,\]
and
\[C_{(7,2,0,0,0,0,0,0,0,0,0,0,0),(7,3,3,3,3,0,0,0,0,0,0,0,0)}^{\onu}(E_8/P_8)\neq 0.\]
However,
\[C_{(5,1,0,0,0,0,0,0,0,0,0,0,0), (7,3,3,4,4,0,1,0,0,0,0,0,0)}^{\onu}(E_8/P_8)= 0.\]
\end{Fact}
Note that the $\olambda$ vector for the last coefficient is a convex combination of the corresponding vectors of the first three coefficients. That is:
\begin{multline}\nonumber 
(5,1,0,0,0,0,0,0,0,0,0,0,0)=\frac{1}{4}(1,0,0,0,0,0,0,0,0,0,0,0,0)\\ \nonumber
+\frac{1}{4}(5,0,0,0,0,0,0,0,0,0,0,0,0)
+\frac{1}{2}(7,2,0,0,0,0,0,0,0,0,0,0,0).
\end{multline}
Similarly, the $\omu$ and (obviously) $\onu$ vector of the last coefficient
is a convex combination of the corresponding vectors of the other coefficients,
with the same parameters $\frac{1}{4},\frac{1}{4},\frac{1}{2}$. Therefore
the convex hull of the points $S^{\tt nonzero}(E_8/P_8)$ contains a point of
$S^{\tt zero}(E_8/P_8)$ and hence no polytopal description of $S^{\tt nonzero}(E_8/P_8)$ is possible with this partition-like description. 

Also, let $\olambda = (1,0,0,0,0,0,0,0,0,0,0,0,0)$ and $\omu =(7,3,3,5,5,1,3,0,0,0,0,0,0)$. Then 
\[C_{\olambda,\omu}^\onu(E_8/P_8)=0 \mbox{\ for $\onu\in\{(7,3,3,5,5,3,2,0,0,0,0,0,0),(7,3,3,5,5,0,5,0,0,0,0,0,0)\}$ and}\]
\[C_{\olambda,\omu}^\onu(E_8/P_8)\neq 0 \mbox{\ for $\onu=(7,3,3,5,5,1,4,0,0,0,0,0,0)$}\]
thus the embedding of $S^{\tt zero}(E_8/P_8)$ is also not polytopal.

Even in $E_8$ one can compute
all vectors in ${\mathbb Z}^{39}$ that correspond to both feasible
and infeasible Schubert triple intersections. With this data one can use a solver on a linear program defined by a relatively large matrix 
to find the vectors of the fact above. This helps automate demonstrating non-convexity for other descriptions of $S^{\tt nonzero}(E_8/P_8)$. 

Notice, all of the counterexamples
to convexity we have given occur when $|\olambda|+|\omu|=\frac{|\Lambda_{G/P}|-1}{2}$.

\subsection{The ``(line,hyperplane)'' flag variety $\ourflag$}
We revisit a simple case of the adjoint varieties, $G/P=\ourflag$. This is the
two step partial flag variety $\{\langle 0\rangle \subset F_1\subset F_{n-1}\subset {\mathbb C}^n\}$ where $F_1$ and $F_{n-1}$ have dimensions $1$ and $n-1$ 
respectively. It has dimension $|\Lambda_{G/P}|=2n-3$. All two-step flag varieties have been solved, in a different way, by I.~Coskun \cite{Coskun}. 
However, our approach is in line with our study of other (co)adjoint cases.

\[\begin{picture}(300,90)
\multiput(4.5,29.5)(13,13){5}{\line(1,1){10}}
\multiput(30.5,29.5)(13,13){4}{\line(1,1){10}}
\multiput(56.5,29.5)(13,13){3}{\line(1,1){10}}
\multiput(82.5,29.5)(13,13){2}{\line(1,1){10}}
\multiput(108.5,29.5)(13,13){1}{\line(1,1){10}}

\multiput(1,25)(13,13){6}{$\bullet$}
\multiput(27,25)(13,13){4}{$\circ$}
\multiput(53,25)(13,13){3}{$\circ$}
\multiput(79,25)(13,13){2}{$\circ$}
\multiput(105,25)(13,13){1}{$\circ$}

\multiput(69.5,91)(13,-13){5}{\line(1,-1){10}}
\multiput(56.5,78)(13,-13){4}{\line(1,-1){10}}
\multiput(43.5,65)(13,-13){3}{\line(1,-1){10}}
\multiput(30.5,52)(13,-13){2}{\line(1,-1){10}}
\multiput(17.5,39)(13,-13){1}{\line(1,-1){10}}

\multiput(66,90)(13,-13){6}{$\bullet$}

\multiput(169,25)(13,13){3}{$\bullet$}
\multiput(208,64)(13,13){2}{$\circ$}
\multiput(234,90)(13,13){1}{$\bullet$}

\multiput(234,90)(13,-13){2}{$\circ$}
\multiput(260,64)(13,-13){4}{$\bullet$}

\multiput(172.5,29.5)(13,13){5}{\line(1,1){10}}

\multiput(237.5,91)(13,-13){5}{\line(1,-1){10}}

\put(50,0){$\Lambda_{\ourflag}$, $\Omega_{GL_{n}}$ and a shape (for $n=7$)}
\end{picture}
\]

We denote the shapes ${\overline \lambda}$ by
$\langle\lambda_1,\lambda_2|\circ\rangle$ and $\langle \lambda_1,\lambda_2|\bullet\rangle$ where $0\leq \lambda_1,\lambda_2 \leq \frac{|\Lambda_{G/P}|-1}{2}$.

Set $\sigma_{\langle \nu|\bullet/\circ\rangle}$, $\sigma_{\langle \nu^{\star}|\bullet\rangle}$ or $\sigma_{\langle\nu_{\star}|\bullet\rangle}$ to be zero if $\nu$, $\nu^{\star}$ or $\nu_{\star}$ are not in $\left[0,\frac{|\Lambda_{G/P}|-1}{2}\right]\times \left[0,\frac{|\Lambda_{G/P}|-1}{2}\right]$.

\begin{Proposition}
\label{prop:flagcase}
$\sigma_{\olambda}\cdot \sigma_{\omu}={\mathbb{A}}_{\olambda,\omu}(\lambda+\mu) \in H^{\star}(\ourflag)$.
\end{Proposition}

\begin{Example}
For $n=5$, the rule gives $\sigma_{\langle 2,0|\circ\rangle}\cdot \sigma_{\langle1,2|\circ\rangle}=
{\mathbb A}_{\langle 2,0|\circ\rangle, \langle1,2|\circ\rangle}(3,2)=
\sigma_{\langle2,2|\bullet\rangle}
+\sigma_{\langle3,1|\bullet\rangle}$.
Pictorially:
\[\begin{picture}(400,40)
\multiput(4.5,4.5)(13,13){3}{\line(1,1){10}}
\multiput(1,0)(13,13){4}{$\circ$}
\multiput(43.5,40)(13,-13){3}{\line(1,-1){10}}
\multiput(40,39)(13,-13){4}{$\circ$}
\multiput(1,0)(13,13){2}{$\bullet$}

\put(85,25){$\times$}

\multiput(104.5,4.5)(13,13){3}{\line(1,1){10}}
\multiput(101,0)(13,13){4}{$\circ$}
\multiput(143.5,40)(13,-13){3}{\line(1,-1){10}}
\multiput(140,39)(13,-13){4}{$\circ$}
\multiput(101,0)(13,13){1}{$\bullet$}
\multiput(166,13)(13,-13){2}{$\bullet$}

\put(185,25){$=$}

\multiput(204.5,4.5)(13,13){3}{\line(1,1){10}}
\multiput(201,0)(13,13){4}{$\circ$}
\multiput(243.5,40)(13,-13){3}{\line(1,-1){10}}
\multiput(240,39)(13,-13){4}{$\circ$}
\multiput(201,0)(13,13){2}{$\bullet$}
\multiput(266,13)(13,-13){2}{$\bullet$}
\put(240,39){$\bullet$}

\put(285,25){$+$}

\multiput(304.5,4.5)(13,13){3}{\line(1,1){10}}
\multiput(301,0)(13,13){4}{$\circ$}
\multiput(343.5,40)(13,-13){3}{\line(1,-1){10}}
\multiput(340,39)(13,-13){4}{$\circ$}
\multiput(301,0)(13,13){3}{$\bullet$}
\multiput(379,0)(13,-13){1}{$\bullet$}
\put(340,39){$\bullet$}
\put(402,0){\qed}
\end{picture}
\]
\end{Example}

\begin{Corollary}
$\left\{C_{\olambda,\omu}^{\onu}(\ourflag)\right\}= \{0,1\}$.
\end{Corollary}

We describe $S^{\tt nonzero}(\ourflag)$ using 
the identification (\ref{eqn:identify}). The following is clear:

\begin{Corollary}\label{cor:Anonzero}
Assume $\lambda=(\lambda_1,\lambda_2),\mu=(\mu_1,\mu_2),\nu=(\nu_1,\nu_2) \in {\mathbb Z}^2\cap \left[0,\frac{|\Lambda_{G/P}|-1}{2}\right]\times \left[0,\frac{|\Lambda_{G/P}|-1}{2}\right]$ and $\lambda_3,\mu_3,\nu_3\in \{0,1\}$. 
Then $C_{\olambda,\omu}^{\onu}(\ourflag)\neq 0$ if and only if:
\begin{align*}
|\onu| & =  |\olambda|+|\omu|\\ 
\nu_1 & \leq  \lambda_1+\mu_1\\  
\nu_2 & \leq  \lambda_2+\mu_2\\ 
\lambda_3+\mu_3& \leq  \nu_3
\end{align*}
\end{Corollary}

\subsection{Remaining classical (co)adjoint varieties} 
The coadjoint variety $G/P=B_n/P_1$ has $\Lambda_{B_n/P_1}$ equal to a chain of length
    $2n-1$ where the maximal element is the adjoint root. We think of this as sitting in the
    $C_n$ root system for the purposes of computing ${\mathbb Y}_{B_n/P_1}$. 
The next facts clearly follow Monk's rule:
\begin{Proposition}
If $(\olambda,\omu,\onu)\neq ( \langle \lambda|\circ\rangle,\langle \mu|\circ\rangle,\langle\nu|\bullet\rangle)$
then $C_{\olambda, \omu}^{\onu}(B_n/P_1) = C_{\lambda, \mu}^{\nu}(Gr_1(\mathbb{C}^{2n-1}))$. Otherwise,
$C_{\langle\lambda|\circ\rangle, \langle \mu| \circ\rangle}^{\langle\nu|\bullet\rangle}(B_n/P_1)= 2
\cdot C_{\lambda, \mu}^{\nu^{\star}}((Gr_1(\mathbb{C}^{2n}))$, where $\nu^{\star}$ is $\nu$ with one additional root.
\end{Proposition}

The adjoint $G/P$ in type $C_n$ is the odd projective space $C_n/P_1\iso {\mathbb P}^{2n-1}$. Here $\Lambda_{{\mathbb P}^{2n-1}}\iso \Lambda_{B_n/P_1}$ is a
chain of length $2n-1$, with the maximal element being the adjoint root, and all roots are short roots except the adjoint root. Clearly:

\begin{Fact}
$C_{\olambda,\omu}^{\onu}(B_n/P_1)\neq 0$ and $C_{\olambda,\omu}^{\onu}(C_n/P_1)\neq 0$ if and only if $|\onu|=|\olambda|+|\omu|$. Also
\[C_{\olambda, \omu}^{\onu}(C_n/P_1) = 2^{{\rm sh}(\onu)-{\rm sh}(\olambda)-{\rm sh}(\omu)} \cdot C_{\olambda, \omu}^{\onu}(B_n/P_1).\]  
\end{Fact}

The short roots factor can be fractional. 
In fact, if $C_{\olambda, \omu}^{\onu}(C_n/P_1)\neq 0$ then 
the factor equals $\frac{1}{2}$ if and only if $(\olambda,\omu,\onu)= (\langle\lambda|\circ\rangle,\langle \mu|\circ\rangle,\langle\nu|\bullet\rangle)$ (the factor equals $1$ otherwise). Thus all $C_{\olambda,\omu}^{\onu}(C_n/P_1) \in \{0,1\}$.

\subsection{Conclusion of the proof of Theorem~\ref{thm:iff}}
The results of Section~1 and the discussion of
this section prove
Theorem~\ref{thm:iff}, except for $G/P=OG(2,2n)$.

In this final case we identify $\olambda = \langle \lambda^{(1)},\lambda^{(2)}|{\bullet}/{\circ}\rangle$ with the vector 
$(\lambda^{(1)}_1,\lambda^{(1)}_2,\lambda^{(2)}_1,\lambda^{(2)}_2,1/0)\in\mathbb{Z}^5$. Then the triple $(\olambda,\omu,\onu)$ is a vector in $\mathbb{Z}^{15}$. 

With this identification, let $\olambda = \omu = (n-2,0,0,0,0)$. First suppose $n\ge 5$ and consider $\onu \in \{(n-2,0,n-2,0,0), (n-2,1,n-3,0,0), (n-2,2,n-4,0,0), (n-2,3,n-5,0,0)\}$. This defines four collinear triples. By Theorem~\ref{Thm:Dmult} (which reformulates Theorem~\ref{Thm:Dintro}), one verifies these
points alternate between being in $S^{\tt nonzero}(OG(2,2n))$ and $S^{\tt zero}(OG(2,2n))$ (which two are in $S^{\tt zero}(OG(2,2n))$ depends on the parity of $n$). Thus, neither $S^{\tt nonzero}(OG(2,2n))$ nor $S^{\tt zero}(OG(2,2n))$ are polytopal.

If $n=4$ then $C_{\olambda,\omu}^\onu(OG(2,8)) \neq 0$ for $\onu \in \{(2,0,2,0,0), (2,2,0,0,0)\}$ while we have $C_{\olambda,\omu}^\onu(OG(2,8)) = 0$ for $\onu=(2,1,1,0,0)$. Thus $S^{\tt nonzero}(OG(2,8))$ is not polytopal. If instead $\olambda = (2,0,0,0,0)$ and $\omu = (1,0,1,0,0)$, we have $C_{\olambda,\omu}^\onu(OG(2,8)) = 0$ for $\onu \in \{(2,0,2,0,0), (2,2,0,0,0)\}$ while $C_{\olambda,\omu}^\onu(OG(2,8))\neq 0$ for $\onu=(2,1,1,0,0)$. Thus we see $S^{\tt zero}(OG(2,8))$ is not polytopal.\qed

\begin{Example}
\label{exa:D5badness}
We now give some alternative vector descriptions of shapes and show that polytopality is not achieved in 
$OG(2,10)$. 
 
One could choose to identify $\olambda$ with the vector in $\mathbb{Z}^{2n-3}$ whose first $n-2$ coordinates are the columns of the bottom layer of $\olambda$, second $n-2$ coordinates are the columns of the top layer, and whose last coordinate is $1$ if $\olambda=\langle\lambda|\bullet\rangle$ and $0$ otherwise. Consider $OG(2,10)$ and let $\olambda = (2,0,0,0,0,0,0)$, $\omu=(2,2,0,0,0,0,0)$. Then $C_{\olambda,\omu}^{\onu}(OG(2,10))=0$ for $\onu =(2,2,1,1,0,0,0)$, while $C_{\olambda,\omu}^{\onu}(OG(2,10))\neq 0$ for $\onu\in \{(2,2,2,0,0,0,0), (2,2,0,2,0,0,0)\}$.  

Suppose instead we use the flattening process to identify $\olambda$ with the vector in $\mathbb{Z}^{4}$ whose first coordinate is $\lambda_1$, second coordinate is $\lambda_2$, third coordinate is $1$ if $\olambda=\langle\lambda|\bullet\rangle$ and $0$ otherwise, and whose fourth coordinate is $1$ if $\olambda$ is up, $-1$ if $\olambda$ is down, and $0$ if $\olambda$ is neutral. Consider $OG(2,10)$ and let $\olambda=(3,0,0,1)$, $\omu= (3,0,0,-1)$. Then $C_{\olambda,\omu}^{\onu}(OG(2,10))\neq 0$ for $\onu\in \{(6,0,0,0), (4,2,0,0)\}$, while $C_{\olambda,\omu}^{\onu}(OG(2,10))=0$ for $\onu =(5,1,0,0)$. \qed
\end{Example}

\section{Proof in the $\ourflag$ case}
Proposition~\ref{prop:flagcase} is easily seen to be equivalent to:
\begin{Proposition}
\label{thm:flagscasereformulated}
\begin{itemize}
\item[(A)] If $|\langle\lambda|\circ\rangle| + |\langle \mu|\circ\rangle| \le n-2$, then
$\sigma_{\langle\lambda|\circ\rangle} \cdot \sigma_{\langle \mu|\circ\rangle} = \sigma_{\langle\lambda + \mu|\circ\rangle}$
\item[(B)] If $|\langle\lambda|\circ\rangle| + |\langle \mu|\circ\rangle| > n-2$, then
$\sigma_{\langle\lambda|\circ\rangle} \cdot \sigma_{\langle \mu|\circ\rangle} = \sigma_{\langle(\lambda + \mu)^{\star}|\bullet\rangle} + \sigma_{\langle(\lambda + \mu)_{\star}|\bullet\rangle}$
\item[(C)] $\sigma_{\langle\lambda|\bullet\rangle}\cdot \sigma_{\langle \mu|\circ\rangle} = \sigma_{\langle\lambda|\circ\rangle}\cdot \sigma_{\langle \mu|\bullet\rangle} = \sigma_{\langle\lambda + \mu|\bullet\rangle}$
\item[(D)] $\sigma_{\langle\lambda|\bullet\rangle} \cdot \sigma_{\langle \mu|\bullet\rangle} = 0$.
\end{itemize}
Declare any $\sigma_{\overline{\alpha}}$ in the above expressions to be zero if
either $\alpha_1$ or $\alpha_2 \notin [0,n-2]$. Such $\overline{\alpha}$ will be called {\bf illegal}.
\end{Proposition}

\begin{proof}
Let $R={\mathbb Z}[{\mathbb Y}_{\ourflag}]$ denote the ${\mathbb Z}$-module linearly spanned
by shapes. Define an (obviously) commutative product $\star$ on $R$ by setting $\olambda\star\omu$
to be the linear combination of shapes indicated by (A)-(D). Set
\[\ktableau{{\ }}_{\ 1}:=\langle 1,0|\circ\rangle \mbox{\ \ \  and \ \ \  } \ktableau{{\ }}_{\ 2}:=\langle 0,1|\circ\rangle.\]
We first prove the following two associativity relations:
\begin{equation}
\label{leftassoc}
\ktableau{{\ }}_{\ 1}\star \Big(\olambda \star \omu \Big) = \Big(\ktableau{{\ }}_{\ 1}\star \olambda \Big) \star \omu
\end{equation}
\begin{equation}
\label{rightassoc}
\ktableau{{\ }}_{\ 2} \star \Big(\olambda \star \omu \Big) = \Big(\ktableau{{\ }}_{\ 2}\star \olambda \Big) \star \omu
\end{equation}

We may assume that $|\olambda|, |\omu| > 0$. For a given $\olambda$, $\omu$, define $s=1-\delta_{\lambda_1+\mu_1,0}$ (Kronecker symbol) and $t=1-\delta_{\lambda_2,0}$. 

\noindent {\bf Case 1:} ($\olambda$ or $\omu$ is on, or $|\olambda| + |\omu| < n-2$):
If $1+|\olambda|+|\omu|>2n-3=|\Lambda_{\ourflag}|$, then both sides of (\ref{leftassoc}) are zero. Otherwise, both sides of (\ref{leftassoc}) are computed using only (A) and (C), whence we are done since (A) and (C) amounts to ``addition of coordinates''.

\noindent {\bf Case 2:} ($|\langle\lambda|\circ\rangle|+|\langle\mu|\circ\rangle|>n-2$, $|\langle\lambda|\circ\rangle|<n-2$):
By (B),
\[\olambda \star \omu = s\cdot\langle(\lambda + \mu)^{\star}|\bullet\rangle + \langle(\lambda + \mu)_{\star}|\bullet\rangle.\]
By (C),
\[\tableau{{\ }}_{\ 1}\star(\olambda\star\omu)=s\cdot \langle\lambda_1+\mu_1,\lambda_2+\mu_2|\bullet\rangle + \langle\lambda_1+\mu_1+1,\lambda_2+\mu_2-1|\bullet\rangle.\]
By (A),
\[\tableau{{\ }}_{\ 1}\star \olambda = \langle\lambda_1+1,\lambda_2|\circ\rangle,\]
and by (B),
\[(\tableau{{\ }}_{\ 1}\star \olambda)\star \omu=\langle\lambda_1+\mu_1,\lambda_2+\mu_2|\bullet\rangle + \langle\lambda_1+\mu_1+1,\lambda_2+\mu_2-1|\bullet\rangle.\]
Note that if $s=0$ then $\lambda_2+\mu_2 > n-2$, so $\langle\lambda_1+\mu_1,\lambda_2+\mu_2|\bullet\rangle$ is illegal.

\noindent {\bf Case 3:} ($|\langle\lambda|\circ\rangle|+|\langle\mu|\circ\rangle|>n-2$, $|\langle\lambda|\circ\rangle|=n-2$):
The left hand side of (\ref{leftassoc}) is as in Case 2. For the right hand side of (\ref{leftassoc}), by (B)
\[\tableau{{\ }}_{\ 1}\star \olambda = \langle\lambda_1,\lambda_2|\bullet\rangle+t\cdot\langle\lambda_1+1,\lambda_2-1|\bullet\rangle,\]
and by (C),
\[(\tableau{{\ }}_{\ 1}\cdot \olambda)\star\omu=\langle\lambda_1+\mu_1,\lambda_2+\mu_2|\bullet\rangle + t\cdot\langle\lambda_1+\mu_1+1,\lambda_2+\mu_2-1|\bullet\rangle.\]
If $t=0$, then $\lambda_1+\mu_1+1>n-2$ and $\langle\lambda_1+\mu_1+1,\lambda_2+\mu_2-1|\bullet\rangle$ is illegal, while (as in Case 2) $s=0$ implies $\langle\lambda_1+\mu_1,\lambda_2+\mu_2|\bullet\rangle$ is illegal.

\noindent {\bf Case 4:} ($|\langle\lambda|\circ\rangle|+|\langle\mu|\circ\rangle|=n-2$, $|\langle\lambda|\circ\rangle|<n-2$):
The right hand side of (\ref{leftassoc}) is as in Case 2. For the left hand side of (\ref{leftassoc}), by (A) 
\[\olambda \star \omu = \langle\lambda + \mu|\circ\rangle,\]
and by (B),
\[(\olambda \star \omu)\star\tableau{{\ }}_{\ 1}=\langle\lambda_1+\mu_1,\lambda_2+\mu_2|\bullet\rangle + \langle\lambda_1+\mu_1+1,\lambda_2+\mu_2-1|\bullet\rangle.\]

The proof of (\ref{rightassoc}) is analogous.

Monk's formula states: $\sigma_{s_r}\cdot \sigma_w = \sum_{\nu}\sigma_{v}$,
where the sum is over all $v \in S_n$ satisfying $v=w(p,q)\in S_n$
where $1 \le p \le r$ and $r <q \le n$, such that $w(p)<w(q)$ and for all $i \in (p,q)$ we have $w(i) \notin (w(p),w(q))$.
Let us write $w_{\olambda}\in S_n$ to be the permutation whose inversion set is
$\olambda$. Since the surjection $G/B\twoheadrightarrow G/P$ induces an injection
$H^{\star}(G/P)\hookrightarrow H^{\star}(G/B)$, Monk's formula
allows one to compute $\sigma_{\stableau{{\ }}_{\ 1}}\cdot \sigma_{\olambda}$
and $\sigma_{\stableau{{\ }}_{\ 2}}\cdot \sigma_{\olambda}$. It is then straightforward
to check that the resulting expansion agrees with Proposition~\ref{thm:flagscasereformulated}.

Since the classes $\sigma_{\stableau{{\ }}_{\ 1}}$ and $\sigma_{\stableau{{\ }}_{\ 2}}$
are well known to algebraically generate $H^{\star}(\ourflag)$, every class $\sigma_{\olambda}$ can be expressed as a polynomial in these classes. Consequently, in $R$,
the element $\olambda$ can be expressed using the same polynomial in terms of $\tableau{{\ }}_{\ 1}$ and $\tableau{{\ }}_{\ 2}$. This polynomial is well-defined by (\ref{leftassoc}) and (\ref{rightassoc}). From this,
one concludes by an easy induction (cf. Lemma~\ref{lemma:Dassocfinal}) that $(R,\star)$ is an associative ring. Hence,
$R\cong H^{\star}(\ourflag)$ and $\olambda\star\omu$ agrees with $\sigma_{\olambda}\cdot\sigma_{\omu}$, as desired. \end{proof}

\section{Proofs in the Lagrangian and odd orthogonal Grassmannian cases}

We begin with a straightforward reformulation of the $LG(2,2n)$ rule.
Let $M=\min\{\lambda_1-\lambda_2, \mu_1-\mu_2\}$. 
\begin{Theorem}
\label{thm:LGreformulated}
\begin{enumerate}
\item[(A)] If $|\langle\lambda|\circ\rangle| + |\langle \mu|\circ\rangle| \le 2n-3$, then
\[\sigma_{\langle\lambda|\circ\rangle} \cdot \sigma_{\langle \mu|\circ\rangle} = \sum_{0 \le k \le M} \sigma_{\langle\lambda_1+\mu_1-k, \lambda_2+\mu_2+k| \circ\rangle}\]
\item[(B)] If $|\langle\lambda|\circ\rangle| + |\langle \mu|\circ\rangle| > 2n-3$, then
\[\sigma_{\langle\lambda|\circ\rangle}\cdot \sigma_{\langle \mu|\circ\rangle} =
\sum_{0 \le k \le M} [\sigma_{\langle\lambda_1+\mu_1-k, \lambda_2+\mu_2+k-1| \bullet\rangle}+ \sigma_{\langle\lambda_1+\mu_1-k-1, \lambda_2+\mu_2+k| \bullet\rangle}]\]
\item[(C)] \[\sigma_{\langle\lambda|\bullet\rangle}\cdot \sigma_{\langle \mu|\circ\rangle} = \sigma_{\langle\lambda|\circ\rangle}\cdot \sigma_{\langle \mu|\bullet\rangle} =
\sum_{0 \le k \le M} \sigma_{\langle\lambda_1+\mu_1-k, \lambda_2+\mu_2+k| \bullet\rangle}\]
\item[(D)] $\sigma_{\langle\lambda|\bullet\rangle} \cdot \sigma_{\langle \mu|\bullet\rangle} = 0$.
\end{enumerate}
Declare any $\overline{\alpha}$ in the above expressions to be zero if
$(\alpha_1,\alpha_2)$ is not a partition in $2\times (2n-3)$. Such $\overline{\alpha}$ will be called {\bf illegal}.
\end{Theorem}

To prove the $LG(2,2n)$ case of Theorem~\ref{thm:LGreformulated}, define an (obviously commutative) product $\star$ on $R={\mathbb Z}[{\mathbb Y}_{LG(2,2n)}]$ according to (A)--(D).
Then we set
\[\tableau{{\ }}:=\langle 1,0|\circ\rangle \mbox{\ \ \  and \ \ \  } \sctableau{{\ }\\{\ }}:=\langle 1,1|\circ\rangle.\]
We will directly prove the following two associativity relations:
\begin{equation}
\label{eqn:dominoassoc}
\sctableau{{\ }\\{\ }}\star \Big(\olambda \star \omu \Big) = \Big(\sctableau{{\ }\\{\ }}\star \olambda \Big) \star \omu
\end{equation}
\begin{equation}
\label{eqn:boxassoc}
\tableau{{\ }} \star \Big(\olambda \star \omu \Big) = \Big(\tableau{{\ }}\star \olambda \Big) \star \omu
\end{equation}

\excise{
In Appendix~A we show that Theorem \ref{thm:LGreformulated} agrees
with the Pieri Rule of \cite{Prag} for $\tableau{{\ }}$ and $\sctableau{{\ }\\{\ }}$. Then since any $\sigma_{\olambda} \in H^{\star}(LG(2,2n))$ can be expressed as a polynomial $\sigma_{\olambda}= f(\sigma_{\tableau{{\ }}},\sigma_{\sctableau{{\ }\\{\ }}}$), we have $\olambda=f(\tableau{{\ }},\sctableau{{\ }\\{\ }})$ in $R$ (this is well-defined because of (\ref{eqn:dominoassoc}) and (\ref{eqn:boxassoc})). This fact combined with (\ref{eqn:dominoassoc}) and (\ref{eqn:boxassoc}) implies $\star$ is associative, via an easy induction; cf.~Lemma~\ref{lemma:Dassocfinal}. Thus $(R,\star)$ is an associative algebra, algebraically generated by $\tableau{{\ }}$ and $\sctableau{{\ }\\{\ }}$. Then since $\star$ agrees with an established Pieri rule for $\sigma_{\tableau{{\ }}}$ and $\sigma_{\sctableau{{\ }\\{\ }}}$, we have that $\olambda \leftrightarrow \sigma_{\olambda}$ defines an isomorphism of the rings $R$ and $H^{\star}(LG(2,2n))$. 
}
Then we will prove that every $\olambda\in\mathbb{Y}_{LG(2,2n)}$ can be expressed as a polynomial in $\tableau{{\ }}$ and $\sctableau{{\ }\\{\ }}$ (this is well-defined because of (\ref{eqn:dominoassoc}) and (\ref{eqn:boxassoc})). It then follows that $(R,\star)$ is an associative ring, via an easy induction; cf. Lemma~\ref{lemma:Dassocfinal}.
In \cite{Searles:13+} it is shown that $\star$ agrees with a Pieri rule of \cite{BKT:Inventiones} for certain
generators of $H^*(LG(2,2n))$ indexed by shapes $\olambda$ that therefore also generate $R$. Hence, it follows $\olambda \leftrightarrow \sigma_{\olambda}$ defines an isomorphism of the rings $R$ and $H^{\star}(LG(2,2n))$.
 
Once the $LG(2,2n)$ case of Theorem \ref{thm:LGreformulated} is proved, we deduce the $OG(2,2n+1)$ case as follows. Let $\mathcal{B}_w$ denote the Schubert class indexed by a signed permutation $w$ in $H^*(SO_{2m+1}\mathbb{C}/B)$ and $\mathcal{C}_w$ that in $H^*(Sp_{2m}\mathbb{C}/B)$. Let $s(w)$ count the sign changes in $w$. It is well-known to experts, see, e.g. \cite{Berg} that the map $\mathcal{C}_w \mapsto 2^{s(w)}\mathcal{B}_w$ embeds $H^*(Sp_{2m}\mathbb{C}/B)$ into $H^*(SO_{2m+1}\mathbb{C}/B)$. The deduction is easy after observing that $s(w)$ is exactly the number of short roots (in $\Lambda_{OG(2,2n+1)}$)
that are in the inversion set of $w$.

\excise{In particular, suppose:
\[ \sigma_{\olambda}\cdot \sigma_{\omu} = \sum C_{{\olambda},{\omu}}^{\overline{\nu}} \sigma_{\overline{\nu}} \]
then for $H^{\star}(OG(2,2n+1))$ we have
\[ 2^{s(\olambda)}\sigma_{\olambda}\cdot 2^{s(\omu)}\sigma_{\omu} = \sum C_{{\olambda},{\omu}}^{\overline{\nu}} 2^{s(\overline{\nu})}\sigma_{\overline{\nu}} \]
rearranging this yields
\[ \sigma_{\olambda}\cdot \sigma_{\omu} = \sum C_{{\olambda},{\omu}}^{\overline{\nu}} 2^{s(\overline{\nu})-(s(\olambda)+s(\omu))} \sigma_{\overline{\nu}} \]
and since the number of bars in a signed permutation is exactly the number of short roots it inverts in $\Lambda_{OG(2,2n+1)}$, we have
\[ \sigma_{\olambda}\cdot \sigma_{\omu} = \sum C_{{\olambda},{\omu}}^{\overline{\nu}} 2^{\text{short}(\overline{\nu})-(\text{short}(\olambda)+\text{short}(\omu))} \sigma_{\overline{\nu}}, \]
as desired.
}

One aspect of the short roots factor (that does not arise in the cominuscule setting of
\cite{Thomas.Yong:comin}) is it may be fractional; it may equal (at worst) $\frac{1}{2}$.
Therefore, we argue:

\begin{Proposition}
\label{prop:OGoddintegral}
The rule for the $Y=OG(2,2n+1)$ case is integral.
\end{Proposition}
\begin{proof}
Integrality is obvious if ${\rm{sh}}(\olambda)={\rm{sh}}(\omu)=0$. If ${\rm{sh}}(\olambda) + {\rm{sh}}(\omu)>2$, then $\lambda_1+\mu_1-M>2n-3$ and then by (C), $\sigma_{\olambda} \cdot \sigma_{\omu}=0$. If ${\rm sh}(\olambda)=2$ 
and ${\rm sh}(\omu)=0$ (or vice versa) then $\olambda=\langle\lambda|\bullet\rangle$ (respectively $\omu = \langle \mu|\bullet\rangle$), and $\onu$ contains $\olambda$ (respectively $\omu$) and thus ${\rm sh}(\nu)=2$. If ${\rm sh}(\olambda)=1$ and ${\rm sh}(\omu)=0$ (or vice versa), then if $\onu = \langle\nu|\bullet\rangle$ it has at least $1$ short root, while if $\onu = \langle \nu|\circ\rangle$ it contains $\olambda$ (respectively $\omu$), and so has $1$ short root. 

Suppose $\olambda$ and $\omu$ both have $1$ short root. If either $\olambda = \langle\lambda|\bullet\rangle$ or $\omu = \langle \mu|\bullet\rangle$, then for dimension reasons $\onu$ has $2$ short roots. If $\olambda = \langle\lambda|\circ\rangle$ and $\omu = \langle \mu|\circ\rangle$, then $\onu = \langle\nu|\bullet\rangle$ for dimension reasons and thus ${\rm sh}(\nu)\geq 1$. Rewriting (B), we have
\[\sigma_{\olambda} \cdot \sigma_{\omu} = \sigma_{\langle\lambda_1+\mu_1,\lambda_2+\mu_2-1|\bullet\rangle} + 2\sum_{1\le k \le M}\sigma_{\langle\lambda_1+\mu_1-k,\lambda_2+\mu_2+k-1|\bullet\rangle} + \sigma_{\langle\lambda_1+\mu_1-M-1,\lambda_2+\mu_2+M|\bullet\rangle}.\]
As the first term is illegal, and the last term is illegal or has $2$ short roots, we are done.
\end{proof}
\excise{
Proposition~\ref{prop:OGoddintegral} implies any term of a $\star$-product of shapes that
appears with a fractional coefficient must in fact be illegal. This is used implicitly throughout the proof below.
}

\noindent
\emph{Proof of (\ref{eqn:dominoassoc}):} We may assume in proving (\ref{eqn:dominoassoc}) and later (\ref{eqn:boxassoc}) that $|\olambda|, |\omu| > 0$. The following lemma is clear from the definition of $\star$.

\begin{Lemma}\label{lemma:LagrangianLR}
Suppose $c\cdot\onu$ is a term of $\olambda\star\omu$. Then if $\olambda\star\omu$ is computed by (A) or (C), $c=C_{\lambda,\mu}^{\nu}(Gr_2({\mathbb C}^{2n}))$.
\end{Lemma}

\noindent {\bf Case 1:} ($\olambda$ or $\omu$ is on, or $|\olambda| + |\omu| \le 2n-5$): if $2+|\olambda|+|\omu|>4n-5=|\Lambda_{G/P}|$, then both sides of (\ref{eqn:dominoassoc}) are $0$.
Otherwise, both sides of (\ref{eqn:dominoassoc}) are computed using only (A) and (C), so we are done by Lemma~\ref{lemma:LagrangianLR} and the associativity of the Littlewood-Richardson rule. This completes Case 1.

We need some preparation for the remaining cases. Define
\[M(i)=\min\{\lambda_1-\lambda_2+i, \mu_1-\mu_2\}.\]

Define $r=1-\delta_{\lambda_2+\mu_2,0}$, $s=1-\delta_{\lambda_1-\lambda_2,\mu_1-\mu_2}$ and $t=1-\delta_{\lambda_1,\lambda_2}$.
Also declare
\excise{
\begin{eqnarray}\nonumber
f^*(a,b) &=& \langle\lambda_1+\mu_1+a,\lambda_2+\mu_2+b|\bullet\rangle \\ \nonumber
f_*(a,b) &=& \langle\lambda_1+\mu_1-M+a,\lambda_2+\mu_2+M+b|\bullet\rangle \\ \nonumber
F(a,b) &=& \sum_{1 \le k \le M}\langle\lambda_1+\mu_1-k+a,\lambda_2+\mu_2+k+b|\bullet\rangle \\ \nonumber
F^{(1)}(a,b) &=& \sum_{1 \le k \le M(1)}\langle\lambda_1+\mu_1-k+a,\lambda_2+\mu_2+k+b|\bullet\rangle \\ \nonumber
F_{(1)}(a,b) &=& \sum_{1 \le k \le M(-1)}\langle\lambda_1+\mu_1-k+a,\lambda_2+\mu_2+k+b|\bullet\rangle \\ \nonumber
f^{(1)}_*(0,0) &=& \langle\lambda_1+\mu_1-M(1),\lambda_2+\mu_2+M(1)|\bullet\rangle\nonumber
\end{eqnarray}
\begin{eqnarray}\nonumber
f_{(1)*}(-1,1) &=& \langle\lambda_1+\mu_1-M(-1)-1,\lambda_2+\mu_2+M(-1)+1|\bullet\rangle \\ \nonumber
F'^{(i)}(a,b) &=& \sum_{0 \le k \le M(i)}\langle\lambda_1+\mu_1-k+a,\lambda_2+\mu_2+k+b|\bullet\rangle\\ \nonumber
F'_{(i)}(a,b) &=& \sum_{0 \le k \le M(-i)}\langle\lambda_1+\mu_1-k+a,\lambda_2+\mu_2+k+b|\bullet\rangle \nonumber
\end{eqnarray}
\begin{eqnarray}\nonumber
F'(a,b) &=& \sum_{0 \le k \le M}\langle\lambda_1+\mu_1-k+a,\lambda_2+\mu_2+k+b|\bullet\rangle\nonumber
\end{eqnarray}
}
\begin{eqnarray}\nonumber
f(a,b) &=& \langle\lambda_1+\mu_1+a,\lambda_2+\mu_2+b|\bullet\rangle \\ \nonumber
F^{(i)}(a,b) &=& \sum_{1 \le k \le M(i)}\langle\lambda_1+\mu_1-k+a,\lambda_2+\mu_2+k+b|\bullet\rangle \nonumber
\end{eqnarray}
\begin{eqnarray}\nonumber
F'^{(i)}(a,b) &=& \sum_{0 \le k \le M(i)}\langle\lambda_1+\mu_1-k+a,\lambda_2+\mu_2+k+b|\bullet\rangle\\ \nonumber
\end{eqnarray}

\noindent {\bf Case 2:} ($|\olambda|+ |\omu| > 2n-3$ and $|\olambda| \le 2n-5$):
Re-expressing (B),
\[\olambda \star \omu=r\cdot f(0,-1)+2F^{(0)}(0,-1)+s\cdot f(-M(0)-1,M(0)).\]
By (C),
\[\sctableau{{\ }\\{\ }}\star(\olambda\star\omu)=r\cdot f(1,0)+2F^{(0)}(1,0)+s\cdot f(-M(0),M(0)+1).\]
By (A),
\[\sctableau{{\ }\\{\ }} \star \olambda = \langle\lambda_1+1,\lambda_2+1|\circ\rangle.\]
By (B),
\[\left(\sctableau{{\ }\\{\ }} \star \olambda\right)\star \omu=f(1,0)+2F^{(0)}(1,0)+s\cdot f(-M(0),M(0)+1).\]
Note if $r=0$, $f(1,0)$ is illegal since $\lambda_1+\mu_1=|\lambda|+|\mu|>2n-3$.

\noindent {\bf Case 3:} ($|\olambda|+|\omu| > 2n-3$ and $2n-4 \le |\olambda| \le 2n-3$):
$\leftassB$ is as in Case 2. By (B),
\[\sctableau{{\ }\\{\ }} \star \olambda = \langle\lambda_1+1,\lambda_2|\bullet\rangle + t\cdot \langle\lambda_1,\lambda_2+1|\bullet\rangle.\] Thus by (C)
\[(\sctableau{{\ }\\{\ }} \star \olambda)\star\omu
=F'^{(1)}(1,0) + t\cdot F'^{(-1)}(0,1).\]
If $t=0$, $F'^{(-1)}(0,1)=0$ hence
\[F'^{(1)}(1,0) + t\cdot F'^{(-1)}(0,1)=F'^{(1)}(1,0) + F'^{(-1)}(0,1).\]
Notice if $r=0$ then $f(1,0)$ is illegal.

If $\mu_1-\mu_2<\lambda_1-\lambda_2$ then $M(1)=M(-1)=M(0)$ and
\[F'^{(1)}(1,0) + F'^{(-1)}(0,1)=f(1,0)+2F^{(0)}(1,0)+f(-M(0),M(0)+1); \mbox{ \ we are done since $s=1$.} \]
If $\mu_1-\mu_2=\lambda_1-\lambda_2$ then $M(1)=M(0)$, $M(-1)=M(0)-1$, and
\[F'^{(1)}(1,0) + F'^{(-1)}(0,1)=f(1,0)+2F^{(0)}(1,0); \mbox{ \ we are done since $s=0$.}\]
 If $\mu_1-\mu_2>\lambda_1-\lambda_2$, then $M(1)=M(0)+1$, $M(-1)=M(0)-1$, and
\[F'^{(1)}(1,0) + F'^{(-1)}(0,1)=f(1,0)+2F^{(0)}(1,0)+f(-M(0),M(0)+1); \mbox{ \ we are done since $s=1$.}\]

\noindent {\bf Case 4:} ($2n-4\leq |\olambda|+ |\omu| \leq 2n-3$ and $|\olambda| \le 2n-5$):
By (A) and then (B), \[\leftassB=f(1,0)+2F^{(0)}(1,0)+s\cdot f(-M(0),M(0)+1).\]
As in Case 2,
\[\rightassB=f(1,0)+2F^{(0)}(1,0)+s\cdot f(-M(0),M(0)+1).\]

\noindent {\bf Case 5:} ($2n-4\leq |\olambda|+|\omu| \leq 2n-3$ and $2n-4\leq |\olambda| \leq 2n-3$): Since $|\omu| \neq 0$, then $|\olambda| = 2n-4$ and $\omu = \tableau{{\ }}$. We then use (\ref{eqn:boxassoc}), which we prove below independently of (\ref{eqn:dominoassoc}).

\noindent
\emph{Proof of (\ref{eqn:boxassoc}):}

\noindent {\bf Case 1:} ($\olambda$ or $\omu$ is on, or $|\olambda| + |\omu| < 2n-3$):
same argument as in Case 1 of the proof of (\ref{eqn:dominoassoc}).

\noindent {\bf Case 2:} ($|\olambda|+|\omu| > 2n-3$ and $|\olambda| < 2n-3$):
Let
\begin{multline}\nonumber
F_L:=(r\cdot f(1,-1)+ 2F^{(0)}(1,-1)+ s\cdot f(-M(0),M(0))) \\ \nonumber 
+(r\cdot f(0,0) +2F^{(0)}(0,0) +s\cdot f(-M(0)-1,M(0)+1)),\nonumber
\end{multline}
\begin{multline}\nonumber
F_R:=r\cdot f(1,-1) + 2F^{(1)}(1,-1) + f(-M(1),M(1)) \\ \nonumber 
+ t\cdot (f(0,0) + 2F^{(-1)}(0,0) + f(-M(-1)-1,M(-1)+1)). \nonumber
\end{multline}
Here $\olambda \star \omu$ is as in Case 2 of the proof of (\ref{eqn:dominoassoc}). By (C),
$\tableau{{\ }}\star (\olambda\star\omu)=F_L$.
By (A) and the assumption $|\olambda|<2n-3$, $\tableau{{\ }} \star \olambda = \langle\lambda_1+1,\lambda_2|\circ\rangle + t\cdot\langle\lambda_1,\lambda_2+1|\circ\rangle$. Then by (B),
$(\tableau{{\ }} \star \olambda)\star \omu=F_R$.

\begin{Claim}
$F_L=F_R$.
\end{Claim}
\begin{proof} If $r=0$, the $t\cdot f(0,0)$ term of $F_R$ is illegal. If $\lambda_1-\lambda_2=0$, then $t=0$, $M(0)=0$, and it is easy to check $F_L=F_R$. Suppose $\lambda_1-\lambda_2 \neq 0$. Then if $\mu_1-\mu_2 <\lambda_1-\lambda_2$, we have $M(1)=M(-1)=M(0)$, $s=1$, and so $F_L=F_R$. If $\mu_1-\mu_2 =\lambda_1-\lambda_2$, we have $M(1)=M(0)$, $M(-1)=M-1$, $s=0$, and so $F_L=F_R$. If $\mu_1-\mu_2 >\lambda_1-\lambda_2$, we have $M(1)=M+1$, $M(-1)=M-1$, $s=1$, and so $F_L=F_R$.
\end{proof}

\noindent {\bf Case 3:} ($|\olambda|+|\omu| > 2n-3$ and $|\olambda| = 2n-3$):
As in Case~2, $\tableau{{\ }}\star (\olambda\star\omu)=F_L$. By (B),
\[\tableau{{\ }} \star \olambda = r'\cdot \langle\lambda_1+1,\lambda_2-1|\bullet\rangle + 2\langle\lambda_1,\lambda_2|\bullet\rangle + s'\cdot \langle\lambda_1-1,\lambda_2+1|\bullet\rangle,\]
where
$r'=1-\delta_{\lambda_2,0}$ and $s'=1-\delta_{\lambda_1-\lambda_2,1}$.
(Note $\lambda_1-\lambda_2 \neq 0$, since $|\lambda|$ is odd). So by (C),
\[\left(\tableau{{\ }} \star \olambda\right)\star\omu=r'\cdot F'^{(2)}(1,-1) + 2F'^{(0)}(0,0) + s'\cdot F'^{(-2)}(-1,1):=F'_R.\]
\begin{Claim}
$F_L=F'_R$.
\end{Claim}
\begin{proof}
If $r'=0$, then $\lambda_1=2n-3$ and $M(2) = \mu_1-\mu_2$. Thus the term of $F'^{(2)}(1,-1)$
with shortest longer row is $\langle 2n-3+\mu_2+1,\lambda_2+\mu_1-1|\bullet\rangle$. However, this longer row is strictly larger than $2n-3$, so all terms of $F'^{(2)}(1,-1)$ are illegal. Thus we can ``ignore'' $r'$ in the expression of $F'_R$ in the precise sense that the above expression for $F'_R$ is valid even if we replace the coefficient $r'$ by $1$.
If $s'=0$ then $F'^{(-2)}(-1,1)=0$, so we may similarly ignore $s'$ in the expression of $F'_R$. If $r=0$ then $f(1,-1)$ and $f(0,0)$ are both illegal, so we may ignore $r$ in the expression of $F_L$. We break the proof into five cases:

(The ``main case'' $\lambda_1-\lambda_2-(\mu_1-\mu_2)\ge 2$):
Here $M(2)=M(-2)=M(0)$. Now, $F'^{(2)}(1,-1)+F'^{(-2)}(-1,1) = f(1,-1)+F^{(0)}(1,-1)+ F^{(0)}(0,0)+f(-M(0)-1,M(0)+1)$ and $F'^{(0)}(0,0) + F'^{(0)}(0,0)= (f(0,0)+F^{(0)}(0,0))+(F^{(0)}(1,-1)+f(-M(0),M(0)))$. Thus
\[F'_R = f(1,-1)+f(0,0) + 2F^{(0)}(1,-1)+ 2F^{(0)}(0,0) + f(-M(0),M(0))+f(-M(0)-1,M(0)+1).\]
Since $s=1$, $F_L=F'_R$.

($\lambda_1-\lambda_2-(\mu_1-\mu_2)=1$): Here $M(2)=M(0)$ and $M(-2)=M(0)-1$.
The difference from the general case is that $F'_R$ no longer possesses the $f(-M(0)-1,M(0)+1)$ term, but this term is now illegal in $F_L$ and so $F_L=F'_R$.

($\lambda_1-\lambda_2=\mu_1-\mu_2$): Now $M(2)=M(0)$ and $M(-2)=M(0)-2$. The difference from the
general case is that $F'_R$ no longer possesses the $f(-M(0)-1,M(0)+1)$ term and one of the $f(-M(0),M(0))$ terms, but now $s=0$ so $F_L=F'_R$.

($\mu_1-\mu_2-(\lambda_1-\lambda_2)=1$): Thus $M(2)=M(0)+1$ and $M(-2)=M(0)-2$. The change from the general case is exactly the same as that for Subcase 3.2.

($\mu_1-\mu_2-(\lambda_1-\lambda_2)\ge 2$): Here
$M(2)=M(0)+2$, $M(-2)=M(0)-2$. $F'_R$, $F_L$ are the same as in the general case, so again $F_L=F'_R$.
\end{proof}

\noindent {\bf Case 4:} ($|\olambda|+|\omu| = 2n-3$ and $|\olambda| < 2n-3$):
Here $\left(\tableau{{\ }} \star \olambda\right)\star\omu=F_R$. Using (A) then (B),
\[\tableau{{\ }}\star(\olambda\star\omu)=F'^{(0)}(1,-1) + F'^{(0)}(0,0) + F'^{(0)}(0,0) + F'^{(0)}(-1,1).\]
This rearranges to
\[F_L':=f(1,-1)+f(0,0) + 2F^{(0)}(1,-1)+ 2F^{(0)}(0,0) +f(-M(0),M(0))+f(-M(0)-1,M(0)+1).\]
\begin{Claim}
$F'_L=F_R$.
\end{Claim}
\begin{proof}
If $r=0$ then $f(1,-1)$ is illegal, so we may ignore $r$ in $F_R$. If $t=0$ then $M=0$ and it is easy to check $F'_L=F_R$. Suppose $t \neq 0$. Then if $\mu_1-\mu_2<\lambda_1-\lambda_2$, we have $M(1)=M(-1)=M(0)$ and $F'_L=F_R$. Since $|\lambda|+|\mu|$ is odd, we cannot have $\mu_1-\mu_2=\lambda_1-\lambda_2$. If $\mu_1-\mu_2>\lambda_1-\lambda_2$, we have $M(1)=M(0)+1$, $M(-1)=M(0)-1$ and $F'_L=F_R$.
\end{proof}

\begin{Proposition}
\label{prop:Cgenerate}
Assuming (\ref{eqn:dominoassoc}) and (\ref{eqn:boxassoc}), every $\olambda\in \mathbb{Y}_{LG(2,2n)}$ has a well-defined expression as a $\star$-polynomial in $\tableau{{\ }}$ and $\sctableau{{\ }\\{\ }}$ with rational coefficients.
\end{Proposition}
\begin{proof}
For any $2\le k\le 2n-3$, we have $\tableau{{\ }}\star \langle k-1,0|\circ\rangle - \sctableau{{\ }\\{\ }}\star \langle k-2,0|\circ\rangle = \langle k,0|\circ\rangle$. Then since $\tableau{{\ }}=\langle 1,0|\circ\rangle$, we generate all $\langle k,0|\circ\rangle$ for $1\le k \le 2n-3$. Since $\langle \lambda|\circ\rangle = \langle \lambda_1-\lambda_2,0|\circ\rangle\star(\sctableau{{\ }\\{\ }})^{\lambda_2}$, we generate all $\langle \lambda|\circ\rangle$.

We have $\langle 2n-3,0|\bullet\rangle= \tableau{{\ }} \star \langle 2n-3,0|\circ\rangle - \sctableau{{\ }\\{\ }} \star \langle 2n-4,0|\circ\rangle$. Then repeatedly multiplying $\langle 2n-3,0|\bullet\rangle$ by $\tableau{{\ }}$ gives us all $\langle 2n-3,k|\bullet\rangle$. Since we have $\langle 2n-3,0|\bullet\rangle$, we also have $\langle 2n-4,1|\bullet\rangle= \sctableau{{\ }\\{\ }} \star \langle 2n-4,0|\circ\rangle-\langle 2n-3,0|\bullet\rangle$. Then since we have $\langle 2n-3,k|\bullet\rangle$, we can (by repeatedly multiplying $\langle 2n-4,1|\bullet\rangle$ by $\tableau{{\ }}$) obtain all $\langle 2n-4,k|\bullet\rangle$. Now, $\langle 2n-5,2|\bullet\rangle=  \sctableau{{\ }\\{\ }} \star \langle 2n-5,1|\circ\rangle -\langle 2n-4,1|\bullet\rangle$, so we obtain all $\langle 2n-5,k|\bullet\rangle$. Continuing in this way we obtain all $\langle \lambda| \bullet\rangle$.
\end{proof}

Finally:

\noindent
\emph{Proof of Corollary~\ref{cor:BCcoefficients}:}
It is clear from the definitions of $\star$ and ${\rm sh}$ that nonzero $C_{\olambda,\omu}^{\onu}(OG(2,2n+1))$ are powers of $2$. Furthermore, $2^{{\rm sh}(\onu)-{\rm sh}(\olambda)-{\rm sh}(\omu)}$ can be at most $2^2=4$ and thus (using (B)) $C_{\olambda,\omu}^{\onu}(OG(2,2n+1))\leq 8$. For $n\ge 5$, we have $C_{\langle n-2,n-3|\circ\rangle,\langle n-2,n-3|\circ\rangle}^{\langle 2n-5,2n-6|\bullet\rangle}(OG(2,2n+1))=8$. Given $k\in \{1,2,4\}$ it is easy to find triples $(\olambda,\omu,\onu)$ such that $C_{\olambda,\omu}^{\onu}(OG(2,2n+1))=k$. \qed

\noindent
\emph{Proof of Corollary~\ref{cor:BCnonzero}:}
We have $C_{\olambda,\omu}^{\onu}(LG(2,2n))\neq 0$ if and only if $C_{\olambda,\mu}^{\onu}(OG(2,2n+1))\neq 0$.
The condition $|\onu|= |\olambda|+|\omu|$ is clearly necessary for nonzeroness, and the necessity of $\lambda_3+\mu_3 \leq  \nu_3$ follows from the definition of $\mathbb{A}_{\olambda,\omu}$. Therefore assume both of these conditions hold. Let (a) denote the other three inequalities. First assume $\lambda_3+\mu_3=\nu_3$. Then by Theorem~\ref{Thm:BCintro} and the definition of $\mathbb{A}_{\olambda,\omu}$, $C_{\olambda,\omu}^{\onu}(LG(2,2n))\neq 0$ if and only if $C_{\lambda,\mu}^{\nu}\neq 0$. We have $|\lambda|+|\mu|=|\nu|$, so by the Horn inequalities for a $2$-row Grassmannian, $C_{\lambda,\mu}^{\nu}\neq 0$ if and only if the inequalities (a) hold.

Therefore assume $\lambda_3+\mu_3 < \nu_3$, i.e., $\lambda_3=\mu_3=0$ and $\nu_3=1$. Then by Theorem~\ref{Thm:BCintro} and the definition of $\mathbb{A}_{\olambda,\omu}$, $C_{\olambda,\omu}^{\onu}(LG(2,2n))\neq 0$ if and only if either $C_{\lambda,\mu}^{(\nu_1+1,\nu_2)}\neq 0$ or $C_{\lambda,\mu}^{(\nu_1,\nu_2+1)}\neq 0$. First suppose $C_{\olambda,\omu}^{\onu}(LG(2,2n))\neq 0$. Then since $\nu_1+\nu_2+1=|\lambda|+|\mu|$, by the Horn inequalities either the set $(b)=\{\nu_1 +1 \leq  \lambda_1+\mu_1, \nu_2 \leq  \lambda_1+\mu_2, \nu_2 \leq  \lambda_2+\mu_1\}$ or $(c)=\{\nu_1 \leq  \lambda_1+\mu_1, \nu_2+1 \leq  \lambda_1+\mu_2, \nu_2+1 \leq  \lambda_2+\mu_1\}$ holds. But if either (b) or (c) hold, then (a) holds. 

Now suppose $C_{\olambda,\omu}^{\onu}(LG(2,2n))= 0$, and also assume $\nu_1>\nu_2$ (so $(\nu_1,\nu_2+1)$ is a partition). Then one of the inequalities from (b) and one from (c) must be false. If one of the latter two inequalities of (b) or the first of (c) is false, then (a) does not hold. Thus assume $\nu_1 +1 >  \lambda_1+\mu_1$, and either $\nu_2+1 >  \lambda_1+\mu_2$ or $\nu_2+1 >  \lambda_2+\mu_1$. Then for (a) to hold, we must have $\nu_1 = \lambda_1+\mu_1$ and either $\nu_2= \lambda_1+\mu_2$ or $\nu_2= \lambda_2+\mu_1$. But this contradicts $|\nu|+1=|\lambda|+|\mu|$.

Finally suppose $C_{\olambda,\omu}^{\onu}(LG(2,2n))= 0$, and also $\nu_1=\nu_2$. Then one of the inequalities from (b) must be false. If either of the latter two inequalities of (b) is false, then (a) does not hold. Thus assume $\nu_1 +1 >  \lambda_1+\mu_1$. If (a) holds then $\nu_1 = \lambda_1+\mu_1$, and then since $\nu_1=\nu_2$ all inequalities in (a) are equalities, again contradicting $|\nu|+1=|\lambda|+|\mu|$. \qed

\section{Proofs in the even orthogonal Grassmannian case}
\subsection{Reformulation of the rule} 

For given shapes $\oalpha, \obeta \in {\mathbb Y}_{OG(2,2n)}'$ let
\begin{equation}\label{eqn:M}
M=\min\{\alpha_1-\alpha_2, \beta_1-\beta_2\}.
\end{equation}

\begin{Definition}
\label{def:fakeproduct}
Let $\oalpha,\obeta\in {\mathbb Y}_{OG(2,2n)}'$. Define a product $\diamond$ by
\begin{enumerate}
\item[(A)] If $|\langle\alpha| \circ\rangle|+|\langle\beta| \circ\rangle| \le 2n-4$, then
\[\langle\alpha|\circ\rangle \diamond \langle\beta|\circ\rangle = \sum_{0 \le k \le M}
\langle\alpha_1+\beta_1-k, \alpha_2+\beta_2+k| \circ\rangle\]
\item[(B)] If $|\langle\alpha|\circ\rangle| + |\langle\beta|\circ\rangle| > 2n-4$, then
\begin{multline}\nonumber
\langle\alpha|\circ\rangle\diamond \langle\beta|\circ\rangle = \langle\alpha_1+\beta_1, \alpha_2+\beta_2-1| \bullet\rangle +
2\sum_{1 \le k \le M}  \langle\alpha_1+\beta_1-k, \alpha_2+\beta_2+k-1| \bullet\rangle\\ + \langle\alpha_1+\beta_1-M-1, \alpha_2+\beta_2+M|\bullet\rangle
\end{multline}
\item[(C)] $\langle\alpha|\bullet\rangle\diamond \langle\beta|\circ\rangle = \langle\alpha|\circ\rangle\diamond \langle\beta|\bullet\rangle =
\sum_{0 \le k \le M}  \langle\alpha_1+\beta_1-k, \alpha_2+\beta_2+k| \bullet\rangle$
\item[(D)] $\langle\alpha|\bullet\rangle \diamond \langle\beta|\bullet\rangle = 0$.
\end{enumerate}
and extend $\diamond$ to be distributive by linearity.

\excise{
Now, let $\gamma_{\oalpha,\obeta}^{\okappa}=2^{{\rm fsh}(\okappa)-({\rm fsh}(\oalpha)+{\rm fsh}(\obeta))}$.
For brevity, we write $\gamma_k$ to mean $\gamma_{\oalpha,\obeta}^{\okappa}$ where $\okappa$ is the term for which $\gamma_k$ is the coefficient.}

Declare any $\overline{\delta}$ in the above expressions to be zero if
$(\delta_1,\delta_2)$ is not a partition in $2\times(2n-4)$. Such $\overline{\delta}$ will be called {\bf illegal}.
\end{Definition}

Recall the definition (\ref{eqn:eta}) of $\eta_{\olambda,\omu}$ from the
introduction, as well as the definitions in the three paragraphs preceding it. Also, 
a shape $\olambda \in {\mathbb Y}_{OG(2,2n)}$ 
is {\bf Pieri} if $\Pi(\olambda)=\langle j,0|{\bullet}/{\circ}\rangle$, and {\bf non-Pieri} otherwise.

\excise{We recall the following definitions from the introduction. Let ${\rm fsh}(\oalpha)$ be the number of boxes in the $(n-2)$th column of $2\times (2n-4)$ used by $\oalpha$, with the exception that ${\rm fsh}(\langle n-2,n-2|\circ\rangle)=1$.
If $\olambda = \langle \lambda|\bullet/\circ\rangle \in {\mathbb Y}_{OG(2,2n)}$ is charged, define ${\rm ch}(\olambda)\in \{\uparrow,\downarrow\}$ to be the charge of $\olambda$, and ${\rm op}(\olambda)\in \{\uparrow,\downarrow\}$ to be the opposite charge. Define $\Pi(\olambda)=\langle\pi(\lambda)|\bullet/\circ\rangle\in \mathbb{Y}'_{OG(2,2n)}$, where $\pi(\lambda) = \lambda^{(1)} + \lambda^{(2)} :=(\lambda_1,\lambda_2)$ is a partition in $2\times (2n-4)$. Define ${\rm fsh}(\olambda) = {\rm fsh}(\Pi(\olambda))$. Recall $\okappa\in {\mathbb Y}_{OG(2,2n)}'$ is {\bf ambiguous} if $\Pi^{-1}(\okappa)=\{\okappa^\uparrow, \okappa^\downarrow\}$ and {\bf unambiguous} otherwise. If $\okappa$ is unambiguous, we identify $\okappa$ with its preimage $\Pi^{-1}(\okappa)\in {\mathbb Y}_{OG(2,2n)}$.

For shapes $\olambda, \omu \in {\mathbb Y}_{OG(2,2n)}$, define:

$\eta_{\olambda,\omu}=
\begin{cases}
2 & \text{if $\olambda$, $\omu$ are charged and match and $n$ is even;}\\
2 & \text{if $\olambda$, $\omu$ are charged and opposite and $n$ is odd;}\\
1 & \text{if $\olambda$ or $\omu$ are not charged;}\\
0 & \text{otherwise}
\end{cases}$}

\begin{Definition}\label{def:starproduct}
Let $\olambda,\omu\in {\mathbb Y}_{OG(2,2n)}$. Define a commutative product $\star$ on $R={\mathbb Z}[{\mathbb Y}_{OG(2,2n)}]$:

If $\Pi(\olambda)=\Pi(\omu)=\langle n-2,0|\circ\rangle$, then

$\olambda\star\omu=\begin{cases}
\sum_{0\le k\le \frac{n-2}{2}}\langle 2n-4-2k, 2k|\circ\rangle & \text{if $n$ is even and $\olambda$, $\omu$ match}\\                                       
\sum_{0\le k\le \frac{n-4}{2}}\langle 2n-5-2k, 2k+1|\circ\rangle & \text{if $n$ is even and $\olambda$, $\omu$ are opposite}\\
\sum_{0\le k\le \frac{n-3}{2}}\langle 2n-5-2k, 2k+1|\circ\rangle & \text{if $n$ is odd and $\olambda$, $\omu$ match} \\                                    
\sum_{0\le k\le \frac{n-3}{2}}\langle 2n-4-2k, 2k|\circ\rangle & \text{if $n$ is odd and $\olambda$, $\omu$ are opposite} 
\end{cases}$

\noindent
where in the first and third cases above, the shape $\langle n-2,n-2|\circ\rangle$ is assigned ${\rm ch}(\olambda)={\rm ch}(\omu)$.

Otherwise, compute $\Pi(\olambda)\diamond\Pi(\omu)$ and
\begin{itemize}
\item[(i)] Replace any term $\okappa$ that has $\kappa_1=2n-4$ by
$\eta_{\olambda,\omu}\okappa$.
\item[(ii)] Next, replace each $\okappa$ by $2^{{\rm fsh}(\okappa)-{\rm fsh}(\olambda)-{\rm fsh}(\omu)}\okappa$.
\item[(iii)] Finally, ``disambiguate'' using one in the following complete list of possibilities:
\begin{itemize}
\item[(iii.1)] \noindent{\sf (if $\olambda, \omu$ are both non-Pieri)} 
Replace any ambiguous $\okappa$ by $\frac{1}{2}(\okappa^{\uparrow}+\okappa^{\downarrow})$.

\item[(iii.2)] \noindent{\sf (if one of $\olambda, \omu$ is  neutral and Pieri)} 
Since the product $\diamond$ is commutative, we may assume $\olambda$ is Pieri. Then replace any ambiguous $\okappa$ by $\frac{1}{2}(\okappa^{\uparrow}+\okappa^{\downarrow})$ if $\omu$ is neutral, and by $\okappa^{{\rm ch}(\omu)}$ if $\omu$ is charged.

\item[(iii.3)] \noindent{\sf (if one of $\olambda, \omu$ is  charged and Pieri, and the other is non-Pieri)}. We may assume $\olambda$ is Pieri. In particular, $\Pi(\olambda)=\langle n-2,0|\circ\rangle$. 
\begin{itemize}
\item[(iii.3a)] If $\omu = \langle \mu |\bullet\rangle$ is neutral and $|\mu|=2n-4$,
then replace the ambiguous term $\langle 2n-4, n-2|\bullet\rangle$ by $\langle 2n-4, n-2|\bullet\rangle^{{\rm ch}(\olambda)}$ if $\mu_1$ is even and by $\langle 2n-4, n-2|\bullet\rangle^{{\rm op}(\olambda)}$ if $\mu_1$ is odd.
\item[(iii.3b)] Otherwise, replace any ambiguous $\okappa$ by $\frac{1}{2}(\okappa^{\uparrow}+\okappa^{\downarrow})$ if $\omu$ is 
neutral, and by $\okappa^{{\rm ch}(\omu)}$ if $\omu$ is charged.
\end{itemize}
\end{itemize}
\end{itemize}
\end{Definition}

We are now ready to restate our theorem for $OG(2,2n)$:

\begin{Theorem}
\label{Thm:Dmult}
Given $\olambda,\omu\in {\mathbb Y}_{OG(2,2n)}$, compute $\olambda\star\omu$ and replace every $\okappa$ by $\sigma_{\okappa}$. The result equals $\sigma_{\olambda}\cdot\sigma_{\omu}\in H^\star(OG(2,2n))$.
\end{Theorem}

Clearly, if $\olambda, \omu$ are non-Pieri 
then Theorem \ref{Thm:Dmult} agrees with Theorem \ref{Thm:Dintro}. 

\subsection{Associativity relations}
Let $\tableau{{\ }}:=\langle 1,0|\circ\rangle \mbox{\ \ \  and \ \ \  } \sctableau{{\ }\\{\ }}:=\langle 1,1|\circ\rangle$.
The main part of our proof is to establish these three associativity relations:
\begin{equation}
\label{eqn:Ddominoassoc}
\leftassB=\rightassB
\end{equation}
\begin{equation}
\label{eqn:Dboxassoc}
\tableau{{\ }} \star (\olambda \star \omu) = (\tableau{{\ }}\star \olambda) \star \omu
\end{equation}
\begin{equation}
\label{eqn:Dambigassoc}
\langle n-2,0 |\circ\rangle^\uparrow \star \Big(\olambda \star \omu \Big) = \Big( \langle n-2,0|\circ\rangle^\uparrow \star \olambda \Big) \star \omu.
\end{equation}

An ambiguous term $c\cdot\okappa$ of an expression in ${\mathbb Z}[{\mathbb Y}_{OG(2,2n)}']$ {\bf splits} if (iii) replaces it by $\frac{c}{2}(\okappa^{\uparrow}+\okappa^{\downarrow})$. We say an expression in ${\mathbb Z}[{\mathbb Y}_{OG(2,2n)}]$ is {\bf balanced} if it contains the same number of up charged shapes as down charged shapes. The following observation is used throughout our arguments:

\begin{Lemma}\label{cor:case1DomassocGone}
If $n-2\le d\le 3n-6$, there is exactly one ambiguous shape $\oalpha$ with $|\alpha|=d$. Otherwise, every shape $\oalpha$ with $|\alpha|=d$ is unambiguous.

Thus, at most one ambiguous shape $\okappa$ appears in either
$\Pi(\overline{\tau}\star(\olambda\star\omu))$ or $\Pi((\overline{\tau}\star\olambda)\star\omu)$. 
\end{Lemma}

Clearly, we may assume throughout that $|\olambda|, |\omu| > 0$.

\subsection{Proof of relation (\ref{eqn:Ddominoassoc})}
If $\Pi(\olambda)=\Pi(\omu)=\langle n-2,0|\circ\rangle$ then (\ref{eqn:Ddominoassoc}) is just a calculation:
\[\leftassB=\rightassB=\sum_{1\le k \le n-3} \langle 2n-4-k+1,k|\bullet\rangle + \langle n-1,n-2|
\bullet\rangle^\uparrow  + \langle n-1,n-2|\bullet\rangle^\downarrow.\]

Otherwise, our strategy is to reduce to analyzing the simpler $\diamond$ product.

\begin{Lemma}\label{lemma:Domassocleftbalanced}
$\sctableau{{\ }\\{\ }}\star\onu$ is balanced. 
\end{Lemma}
\begin{proof}
Since $\sctableau{{\ }\\{\ }}$ is non-Pieri and neutral, 
this follows from (iii). \end{proof}
\begin{Corollary}\label{cor:Domassocleftbalanced}
$\leftassB$ is balanced.
\end{Corollary}

\begin{Lemma}\label{lemma:Domassocrightbalanced}
$\rightassB$ is balanced. 
\end{Lemma}
\begin{proof}
Throughout, let $\okappa$ be a shape appearing in $\sctableau{{\ }\\{\ }}\diamond\Pi(\olambda)$.

\noindent
{\bf Case 1:} ($\omu\neq \langle n-2,0|\circ\rangle^{{\rm ch}(\omu)}$): $\omu$ is either neutral or non-Pieri. We have two subcases since if $\okappa$ is ambiguous, by (iii) applied to $\sctableau{{\ }\\{\ }}\diamond\Pi(\olambda)$, $\okappa \to \frac{1}{2}(\okappa^\uparrow + \okappa^\downarrow)$.

\noindent($(\okappa^\uparrow + \okappa^\downarrow)\star\omu$): 
Clearly, if $\okappa$ is ambiguous it is non-Pieri.
Thus, if $\omu$ is non-Pieri then balancedness holds
by (iii.1). If $\omu$ is a neutral Pieri shape, then we are balanced by (iii.2).

\noindent($\okappa\star\omu$ for $\okappa$ neutral): If $\omu$ is neutral then balancedness holds by (iii.1) or (iii.2).
Thus assume $\omu$ is a charged non-Pieri shape. If $\okappa$ is also non-Pieri then use (iii.1). Otherwise, $\okappa=\langle 2n-4,0|\bullet\rangle$. Since $\omu$ is charged and non-Pieri we have $\mu_1\ge n-2$ and $\mu_2>0$. Thus any $\otau$ in $\okappa\star\omu$ has $|\tau|>3n-6$. Hence $\otau$ is neutral (Lemma~\ref{cor:case1DomassocGone}) and $\okappa\star\omu$ is balanced.

\noindent
{\bf Case 2:} ($\omu= \langle n-2,0|\circ\rangle^{{\rm ch}(\omu)}$): 

\noindent($|\lambda|\neq 2n-5$): Any shape $\okappa$ appearing in $\sctableau{{\ }\\{\ }}\diamond\Pi(\olambda)$ is non-Pieri. If $\okappa$ is unambiguous, then any ambiguous $\otau$ in $\okappa\diamond\Pi(\omu)$ splits by (iii.3b). By Lemma~\ref{lemma:Domassocleftbalanced},
$\sctableau{{\ }\\{\ }}\diamond\Pi(\olambda)$ is balanced, thus 
if $\okappa$ is ambiguous then $(\okappa^\uparrow + \okappa^\downarrow)\star\omu$ is balanced
by (iii.3b). 

\noindent($|\lambda|=2n-5$): First suppose $\Pi(\olambda) \neq \langle n-2,n-3|\circ\rangle$. Then $\sctableau{{\ }\\{\ }}\star\olambda = \langle \lambda_1+1,\lambda_2|\bullet\rangle + \langle \lambda_1,\lambda_2+1|\bullet\rangle$ (both neutral). 
The ambiguous term $\langle 2n-4,n-2|\bullet\rangle$ appears in both $\langle \lambda_1+1,\lambda_2|\bullet\rangle\diamond\Pi(\omu)$ and $\langle \lambda_1,\lambda_2+1|\bullet\rangle\diamond\Pi(\omu)$. For either
of these computations, (i) has no effect, and (ii) multiplies $\langle 2n-4,n-2|\bullet\rangle$ by $1$. Since exactly one of $\{\lambda_1+1,\lambda_1\}$ is even, by (iii.3a) $\langle 2n-4,n-2|\bullet\rangle$ is assigned ${\rm ch}(\omu)$ in one of the expressions and ${\rm op}(\omu)$ in the other, giving balancedness.

If instead $\Pi(\olambda) = \langle n-2,n-3|\circ\rangle$, then 
\[\sctableau{{\ }\\{\ }}\star\olambda = \langle n-1,n-3|\bullet\rangle + \langle n-2,n-2|\bullet\rangle^\uparrow + \langle n-2,n-2|\bullet\rangle^\downarrow.\]

In $\langle n-1,n-3|\bullet\rangle\diamond\Pi(\omu)$ there is an ambiguous term $\langle 2n-4,n-2|\bullet\rangle$. Its coefficient remains $1$ after applying (i) and (ii). Then (iii.3a) assigns it ${\rm ch}(\omu)$ if $n-1$ is even and ${\rm op}(\omu)$ if $n-1$ is odd. Next,
\[\Pi(\langle n-2,n-2|\bullet\rangle^\uparrow + \langle n-2,n-2|\bullet\rangle^\downarrow)\diamond \Pi(\omu) = \langle 2n-4,n-2|\bullet\rangle + \langle 2n-4,n-2|\bullet\rangle.\] 
Here (i) rescales one of these shapes by $2$ and the other by $0$, after which (ii) rescales the surviving shape by $\frac{1}{2}$. The surviving shape is the one arising from $\langle n-2,n-2|\bullet\rangle^{{\rm ch}(\omu)}$ if $n$ is even and from $\langle n-2,n-2|\bullet\rangle^{{\rm op}(\omu)}$ if $n$ is odd. Thus (iii.3b) assigns that shape ${\rm ch}(\omu)$ if $n$ is even and ${\rm op}(\omu)$ if $n$ is odd. This balances the charge assigned in $\langle n-1,n-3|\bullet\rangle\diamond\Pi(\omu)$.
\end{proof}

Define $\mathcal{L}(\otau,\olambda,\omu)\in\mathbb{Z}[\mathbb{Y}_{OG(2,2n)}']$ to be the result of applying (i) and (ii) (but not (iii)) to $\Pi(\otau)\diamond\Pi(\olambda\star\omu)$. We need a comparable: define $\Omega(\olambda,\omu)$ to be the expression obtained by computing $\Pi(\olambda)\diamond\Pi(\omu)$ and applying (ii), and define $\Omega'(\otau,\olambda,\omu)$ to be the result of applying (ii) to $\Pi(\otau)\diamond\Omega(\olambda,\omu)$. Thus $\Omega'(\otau,\olambda,\omu)$ ignores (i) while $\mathcal{L}(\otau,\olambda,\omu)$ does not.

Similarly, define $\mathcal{R}(\otau,\olambda,\omu)$ by computing
$\Pi(\otau\star\olambda)\diamond\Pi(\omu)$ and applying (i) and (ii) (but not (iii)).  Define $\Sigma(\otau,\olambda)$ to be the expression obtained by computing $\Pi(\otau)\diamond\Pi(\olambda)$ and applying (ii), and define $\Sigma'(\otau,\olambda,\omu)$ to be the result of applying (ii) to $\Sigma(\otau,\olambda)\diamond\Pi(\omu)$.

\begin{Lemma}\label{lemma:partialproducts}
If $\Pi(\olambda), \Pi(\omu)$ are not both $\langle n-2,0|\circ\rangle$, then $\mathcal{L}(\sctableau{{\ }\\{\ }},\olambda,\omu) = \Omega'(\sctableau{{\ }\\{\ }},\olambda,\omu)$ and $\mathcal{R}(\sctableau{{\ }\\{\ }},\olambda,\omu) = \Sigma'(\sctableau{{\ }\\{\ }},\olambda,\omu)$.
\end{Lemma}
\begin{proof}
First consider $\mathcal{L}(\sctableau{{\ }\\{\ }},\olambda,\omu)$ and $\Omega'(\sctableau{{\ }\\{\ }},\olambda,\omu)$. 
If one of $\olambda$ or  $\omu$ is neutral, then (i) has no effect on $\Pi(\olambda)\diamond\Pi(\omu)$ and so $\Pi(\olambda\star\omu)=\Omega(\olambda, \omu)$. Then since $\sctableau{{\ }\\{\ }}$ is neutral, (i) has no effect on $\sctableau{{\ }\\{\ }} \diamond \Pi(\olambda\star\omu)$, so $\mathcal{L}(\sctableau{{\ }\\{\ }},\olambda,\omu)=\Omega'(\sctableau{{\ }\\{\ }}, \olambda, \omu)$. Thus assume both $\olambda$, $\omu$ are charged. Then in $\Pi(\olambda)\diamond\Pi(\omu)$, (i) only affects the term $\okappa=\langle 2n-4,\kappa_2|\bullet\rangle$. So 
\[\Pi(\olambda\star\omu) - \Omega(\olambda,\omu) = 2^{{\rm fsh}(\okappa)-{\rm fsh}(\olambda)-{\rm fsh}(\omu)}\cdot \eta_{\olambda,\omu}\cdot \okappa - 2^{{\rm fsh}(\okappa)-{\rm fsh}(\olambda)-{\rm fsh}(\omu)}\cdot \okappa.\]
(This computation is the only place where the hypothesis is used.)
Since $\sctableau{{\ }\\{\ }}$ is neutral, (i) has no effect on $\sctableau{{\ }\\{\ }} \diamond \Pi(\olambda\star\omu)$. Then $\mathcal{L}(\sctableau{{\ }\\{\ }},\olambda,\omu)=\Omega'(\sctableau{{\ }\\{\ }},\olambda,\omu)$ since $\sctableau{{\ }\\{\ }}\diamond\okappa=0$.

Now consider $\mathcal{R}(\sctableau{{\ }\\{\ }},\olambda,\omu)$ and $\Sigma'(\sctableau{{\ }\\{\ }},\olambda,\omu)$.
Since $\sctableau{{\ }\\{\ }}$ is neutral, (i) has no effect on $\sctableau{{\ }\\{\ }}\diamond\Pi(\olambda)$, so $\Pi(\sctableau{{\ }\\{\ }}\star\olambda) = \Sigma(\sctableau{{\ }\\{\ }},\olambda)$. If either $\omu$ is neutral or $\sctableau{{\ }\\{\ }}\diamond\Pi(\olambda)$ has no ambiguous term, then (i) has no effect on $\Pi(\sctableau{{\ }\\{\ }}\star\olambda)\diamond\Pi(\omu)$, so $\mathcal{R}(\sctableau{{\ }\\{\ }},\olambda,\omu)=\Sigma'(\sctableau{{\ }\\{\ }}, \olambda, \omu)$. Thus assume $\omu$ is charged and $\okappa$ appearing in $\sctableau{{\ }\\{\ }}\diamond\Pi(\olambda)$ is ambiguous. Then by (iii.1), (iii.2) or (iii.3b), $\okappa \mapsto \frac{1}{2}(\okappa^\uparrow + \okappa^\downarrow)$. For $\Pi(\sctableau{{\ }\\{\ }}\star\olambda)\diamond\Pi(\omu)$, (i) only affects the expressions $\Pi(\okappa^\uparrow)\diamond\Pi(\omu)$ and $\Pi(\okappa^\downarrow)\diamond\Pi(\omu)$, by doubling the term $\otau = \langle 2n-4,\tau_2|\bullet\rangle$ in one expression and multiplying $\otau$ by $0$ in the other. Then $\mathcal{R}(\sctableau{{\ }\\{\ }},\olambda,\omu)=\Sigma'(\sctableau{{\ }\\{\ }},\olambda,\omu)$ follows from the fact $\sctableau{{\ }\\{\ }}\star\olambda$ is balanced.
\end{proof}

\begin{Lemma}\label{lemma:diamondstar}
If $\Pi(\olambda),\Pi(\omu)$ are not both $\langle n-2,0|\circ\rangle$ and
$\sctableau{{\ }\\{\ }}\diamond(\Pi(\olambda)\diamond\Pi(\omu))=(\sctableau{{\ }\\{\ }}\diamond\Pi(\olambda))\diamond\Pi(\omu)$ then $\leftassB=\rightassB$. 
\end{Lemma}
\begin{proof}
By Corollary~\ref{cor:Domassocleftbalanced} and Lemma~\ref{lemma:Domassocrightbalanced} both $\leftassB$ and $\rightassB$ are balanced. Hence it suffices to show $\mathcal{L}(\sctableau{{\ }\\{\ }},\olambda,\omu)=\mathcal{R}(\sctableau{{\ }\\{\ }},\olambda,\omu)$. By Lemma~\ref{lemma:partialproducts} it suffices to show $\Omega'(\sctableau{{\ }\\{\ }},\olambda,\omu)=\Sigma'(\sctableau{{\ }\\{\ }},\olambda,\omu)$. 
Consider any $\okappa$ appearing in one of $\Omega'(\sctableau{{\ }\\{\ }},\olambda,\omu)$ or $\Sigma'(\sctableau{{\ }\\{\ }},\olambda,\omu)$. Since ${\rm fsh}(\sctableau{{\ }\\{\ }})=0$, the total effect of both applications of (ii) in computing $\Omega'(\sctableau{{\ }\\{\ }},\olambda,\omu)$ or $\Sigma'(\sctableau{{\ }\\{\ }},\olambda,\omu)$ is the same as multiplying $\okappa$ in 
$\sctableau{{\ }\\{\ }}\diamond(\Pi(\olambda)\diamond\Pi(\omu))$ or $(\sctableau{{\ }\\{\ }}\diamond\Pi(\olambda))\diamond\Pi(\omu)$ (respectively)  by $2^{{\rm fsh}(\okappa)-{\rm fsh}(\olambda)-{\rm fsh}(\omu)}$. 
However, we assumed $\sctableau{{\ }\\{\ }}\diamond(\Pi(\olambda)\diamond\Pi(\omu))=(\sctableau{{\ }\\{\ }}\diamond\Pi(\olambda))\diamond\Pi(\omu)$ so $\Omega'(\sctableau{{\ }\\{\ }},\olambda,\omu)=\Sigma'(\sctableau{{\ }\\{\ }},\olambda,\omu)$.
\end{proof}

The proof that 
$\sctableau{{\ }\\{\ }}\diamond(\Pi(\olambda)\diamond\Pi(\omu))=(\sctableau{{\ }\\{\ }}\diamond\Pi(\olambda))\diamond\Pi(\omu)$ is identical to the proof of (\ref{eqn:dominoassoc}), \emph{mutatis mutandis} for a $2\times (2n-4)$ rectangle instead of a $2\times (2n-3)$ rectangle.

\subsection{Proof of relation (\ref{eqn:Dboxassoc})}
Since $\tableau{{\ }}$ is a Pieri shape, if $\omu$ is a Pieri shape then every product in $\leftassboxB$ and $\rightassboxB$ is computed by the Pieri rule of \cite{BKT:Inventiones}, so associativity is immediate. Thus we assume throughout the proof of (\ref{eqn:Dboxassoc}) that $\omu$ is a non-Pieri shape. We break our argument into the amenable and non-amenable
cases.

\subsubsection{The amenable case}
Call a pair $\olambda, \omu\in \mathbb{Y}_{OG(2,2n)}$ of shapes {\bf amenable} if 
\begin{itemize}
\item $\olambda$ and $\omu$ are non-Pieri; or
\item $\olambda$ is Pieri and $\omu$ is both neutral and non-Pieri.
\end{itemize}

\begin{Proposition}\label{lemma:boxtimesbalanced}
If $\Theta\in \mathbb{Z}[\mathbb{Y}_{OG(2,2n)}]$ is balanced  then $\tableau{{\ }}\star\Theta$ is balanced. 
\end{Proposition}
\begin{proof}
If $\olambda$ is a neutral shape in $\Theta$ then  $\tableau{{\ }}\star\olambda$ is balanced by (iii.2). If $\olambda$ is charged, then any ambiguous shape in $\tableau{{\ }}\diamond\Pi(\olambda)$ is assigned ${\rm ch}(\olambda)$ by (iii.2). However, $\Theta$ is balanced, so there are the same number of $\olambda$ in $\Theta$ with charge $\uparrow$ as with $\downarrow$. 
\end{proof}
\begin{Corollary}\label{cor:leftboxbalanced}
If $\olambda\star\omu$ is balanced, then $\leftassboxB$ is balanced.
\end{Corollary}

\begin{Lemma}\label{lemma:boxbalanced}
Let $\olambda, \omu$ be an amenable pair. Then $\leftassboxB$ and $\rightassboxB$ are balanced.  
\end{Lemma}
\begin{proof}
For $\leftassboxB$, by Corollary~\ref{cor:leftboxbalanced} it suffices to show $\olambda\star\omu$ is balanced. This is clear by (iii) unless $\olambda=\langle n-2,0|\circ\rangle^{{\rm ch}(\olambda)}$ and $\omu$ is on with $|\mu|=2n-4$ (in which case we use (iii.3a)). In this case, any $\otau$ appearing in $\leftassboxB$ has $|\tau|>3n-6$ and so by Lemma~\ref{cor:case1DomassocGone}, $\leftassboxB$ is neutral. 

For $\rightassboxB$ there are two cases:

\noindent{\bf Case 1:} ($\olambda$ and $\omu$ are non-Pieri): Every term of $\tableau{{\ }}\star\olambda$ is non-Pieri, and thus balancedness of $\rightassboxB$ holds by (iii.1), unless $\olambda=\langle 2n-5,1|\circ\rangle$. In this case, $\tableau{{\ }}\star\olambda$ has a single Pieri term, namely $\onu=\langle 2n-4,0|\bullet\rangle$. It suffices to show $\onu\star\omu$ is balanced. If $\omu$ is neutral
we are done by (iii.2). Otherwise $\omu$ is charged, and then $|\mu|>n-2$ since $\omu$ is non-Pieri. This implies any $\otau$ in $\onu\star\omu$ has $|\tau|>3n-6$, implying $\otau$ is neutral by Lemma~\ref{cor:case1DomassocGone}, so $\onu\star\omu$ is balanced.

\noindent{\bf Case 2:} ($\olambda$ is Pieri and $\omu$ is neutral and non-Pieri): For a given shape $\onu$, 
$\onu\star\omu$ is balanced by (iii)
except if $\omu$ is on, $|\mu|=2n-4$, and $\onu=\langle n-2,0|\circ\rangle^{{\rm ch}(\onu)}$ (when we use (iii.3a)). Since $\olambda$ is Pieri, $\langle n-2,0|\circ\rangle$ is a term of $\tableau{{\ }}\diamond\Pi(\olambda)$ only if $\lambda_1 = n-3$. But then $\tableau{{\ }}\star\olambda = (\langle n-2,0|\circ\rangle^\uparrow + \langle n-2,0|\circ\rangle^\downarrow)+ \langle n-3,1|\circ\rangle$. So if $\omu$ is on and $|\mu|=2n-4$, the charged term in $\langle n-2,0|\circ\rangle^\uparrow \star \omu$ is balanced by the charged term in $\langle n-2,0|\circ\rangle^\downarrow \star \omu$ when using (iii.3a).
\end{proof}

\begin{Lemma}\label{lemma:boxparti}
Let $\olambda, \omu$ be amenable. If $\Omega'(\tableau{{\ }},\olambda,\omu) = \Sigma'(\tableau{{\ }},\olambda,\omu)$, then $\mathcal{L}(\tableau{{\ }},\olambda,\omu) = \mathcal{R}(\tableau{{\ }},\olambda,\omu)$.
\end{Lemma}
\begin{proof}
If $\omu$ is neutral, then since $\tableau{{\ }}$ is neutral (i) has no effect on either $\Pi(\olambda)\diamond\Pi(\omu)$ or $\tableau{{\ }}\diamond\Pi(\olambda)$. Thus $\Omega(\olambda,\omu) = \Pi(\olambda\star\omu)$ and $\Sigma(\tableau{{\ }},\olambda)=\Pi(\tableau{{\ }}\star\olambda)$. Also, (i) has no effect on either $\leftassbox$ or $\Pi(\tableau{{\ }}\star\olambda)\diamond\Pi(\omu)$. So $\Omega'(\tableau{{\ }},\olambda,\omu)=\mathcal{L}(\tableau{{\ }},\olambda,\omu)$ and $\Sigma'(\tableau{{\ }},\olambda,\omu)=\mathcal{R}(\tableau{{\ }},\olambda,\omu)$.

Therefore, assume $\omu$ is charged. There are two cases:

\noindent{\bf Case 1:} ($\olambda$ is neutral): Then (i) has no effect on either $\Pi(\olambda)\diamond\Pi(\omu)$ or $\tableau{{\ }}\diamond\Pi(\olambda)$, so $\Omega(\olambda,\omu) = \Pi(\olambda\star\omu)$ and $\Sigma(\tableau{{\ }},\olambda)=\Pi(\tableau{{\ }}\star\olambda)$. Next, (i) has no effect on $\leftassbox$, so $\Omega'(\tableau{{\ }},\olambda,\omu)=\mathcal{L}(\tableau{{\ }},\olambda,\omu)$. If $\tableau{{\ }}\diamond\Pi(\olambda)$ has no ambiguous term, then (i) has no effect on $\Pi(\tableau{{\ }}\star\olambda)\diamond\Pi(\omu)$ and so $\Sigma'(\tableau{{\ }},\olambda,\omu)=\mathcal{R}(\tableau{{\ }},\olambda,\omu)$. If an ambiguous shape $\oalpha$ appears in $\tableau{{\ }}\diamond\Pi(\olambda)$, then by (iii.2), $\oalpha\mapsto\frac{1}{2}(\oalpha^\uparrow + \oalpha^\downarrow)$. For $\Pi(\tableau{{\ }}\star\olambda)\diamond\Pi(\omu)$, (i) only affects the expressions $\Pi(\oalpha^\uparrow)\diamond\Pi(\omu)$ and $\Pi(\oalpha^\downarrow)\diamond\Pi(\omu)$, by doubling the term $\otau=\langle 2n-4, \tau_2|\bullet\rangle$ in one expression and multiplying $\otau$ by $0$ in the other. Then $\Sigma'(\tableau{{\ }},\olambda,\omu)=\mathcal{R}(\tableau{{\ }},\olambda,\omu)$ follows from the fact $\tableau{{\ }}\star\olambda$ is balanced. 

\noindent{\bf Case 2:} ($\olambda$ is charged): (The first three subcases below don't even use
the hypothesis $\Omega'(\tableau{{\ }},\olambda,\omu) = \Sigma'(\tableau{{\ }},\olambda,\omu)$.)

\noindent\emph{Subcase 2a: (If both $\olambda$, $\omu$ are on)}: Then $\mathcal{L}(\tableau{{\ }},\olambda,\omu)=\mathcal{R}(\tableau{{\ }},\olambda,\omu)=0$ by (D). 

\noindent\emph{Subcase 2b: ($\olambda$ is off and $\omu$ is on)}: We may assume $|\olambda|<2n-4$, otherwise $1+|\olambda|+|\omu|>4n-7=|\Lambda_{OG(2,2n)}|$ and again $\mathcal{L}(\tableau{{\ }},\olambda,\omu)=\mathcal{R}(\tableau{{\ }},\olambda,\omu)=0$. Then
\[\mathcal{L}(\tableau{{\ }},\olambda,\omu) = \mathcal{R}(\tableau{{\ }},\olambda,\omu) = \frac{1}{2}\eta_{\olambda,\omu}\langle 2n-4,\lambda_2+\mu_1|\bullet\rangle\]
\noindent
\emph{Subcase 2c: ($\olambda$ is on and $\omu$ is off)}: As in Subcase 2b, assume $|\omu|<2n-4$. Then:
\[\mathcal{L}(\tableau{{\ }},\olambda,\omu) = \mathcal{R}(\tableau{{\ }},\olambda,\omu) = \frac{1}{2}\eta_{\olambda,\omu}\langle 2n-4,\lambda_1+\mu_2|\bullet\rangle \]

\noindent\emph{Subcase 2d: ($\olambda$, $\omu$ are off)}: Then $\lambda_1=\mu_1=n-2$ and since $\omu$ is non-Pieri
we know $|\omu|>n-2$ and hence $\Pi(\olambda)\diamond\Pi(\omu)$ is computed by (B). Define
\[\okappa=\langle 2n-4,\lambda_2+\mu_2-1|\bullet\rangle \mbox{ \ and \ } \okappa'=\langle 2n-4,\lambda_2+\mu_2|\bullet\rangle, \mbox{ \ where $\okappa,\okappa'\in \mathbb{Y}_{OG(2,2n)}'$.}\]
Then the only term of $\Pi(\olambda)\diamond\Pi(\omu)$ affected by (i) is $\okappa$. Also, (ii) multiplies $\okappa$ by $c=\frac{1}{2}(1+\delta_{{\rm fsh}(\okappa),2})$. Thus:
\[\Pi(\olambda\star\omu)-\Omega(\olambda,\omu) = c\cdot \eta_{\olambda,\omu}\okappa - c\cdot \okappa \mbox{ \ and therefore }\]
\[\leftassbox - \tableau{{\ }}\diamond \Omega(\olambda,\omu) = c\cdot \eta_{\olambda,\omu}\okappa' - c\cdot \okappa'.\]
Since $\tableau{{\ }}$ is neutral (i) has no effect on either
$\leftassbox$ or $\tableau{{\ }}\diamond \Omega(\olambda,\omu)$. After that, 
(ii) multiplies $\okappa'$ by $c'=1+\delta_{\lambda_2+\mu_2,n-2}$, so
\begin{equation}\label{eqn:LL}
\mathcal{L}(\tableau{{\ }},\olambda,\omu) - \Omega'(\tableau{{\ }},\olambda,\omu) = c'\cdot c\cdot \eta_{\olambda,\omu}\okappa' - c'\cdot c\cdot \okappa'.
\end{equation}

First, suppose $\olambda\neq \langle n-2,n-2|\circ\rangle^{{\rm ch}(\olambda)}$. Then by (A), $\tableau{{\ }}\diamond\Pi(\olambda) = \langle n-1,\lambda_2|\circ\rangle + \langle n-2,\lambda_2+1|\circ\rangle$. Since $\tableau{{\ }}$ is neutral, (i) has no effect. (ii) multiplies each term by $1$. (Below we use that (iii.2) assigns the second term ${\rm ch}(\olambda)$.) Hence:
\[\Sigma(\tableau{{\ }},\olambda) = \Pi(\tableau{{\ }}\star\olambda) = \langle n-1,\lambda_2|\circ\rangle + \langle n-2,\lambda_2+1|\circ\rangle.\]
Since $\langle n-1,\lambda_2|\circ\rangle$ is neutral, no term of $\Pi(\langle n-1,\lambda_2|\circ\rangle)\diamond\Pi(\omu)$ is affected by (i). For $\Pi(\langle n-2,\lambda_2+1|\circ\rangle^{{\rm ch}(\olambda)})\diamond\Pi(\omu)$, (i) affects only the term $\okappa'$ (multiplying it by $\eta_{\olambda,\omu}$), while (ii) multiplies $\okappa'$ by $d=\frac{1}{2}(1+\delta_{{\rm fsh}(\okappa'),2})$. Therefore,
\begin{equation}\label{eqn:RR1}
\mathcal{R}(\tableau{{\ }},\olambda,\omu) - \Sigma'(\tableau{{\ }},\olambda,\omu) = d\cdot \eta_{\olambda,\omu}\okappa' - d\cdot \okappa'.
\end{equation}
The statement follows by the hypothesis, comparing (\ref{eqn:LL}) and (\ref{eqn:RR1}), and noting $c'\cdot c = d$.

Finally, assume $\olambda=\langle n-2,n-2|\circ\rangle^{{\rm ch}(\olambda)}$. Then by (B), $\tableau{{\ }}\diamond\Pi(\olambda) = \langle n-1,n-3|\bullet\rangle + \langle n-2,n-2|\bullet\rangle$. Since $\tableau{{\ }}$ is neutral, (i) has no effect. (ii) multiplies the first term by $1$ and the second term by $2$. (Below we use that (iii.2) assigns the second term ${\rm ch}(\olambda)$.) We have:
\[\Sigma(\tableau{{\ }},\olambda) = \Pi(\tableau{{\ }}\star\olambda) = \langle n-1,n-3|\bullet\rangle + 2\langle n-2,n-2|\bullet\rangle.\]
Since $\langle n-1,n-3|\bullet\rangle$ is neutral, no term of $\Pi(\langle n-1,n-3|\bullet\rangle)\diamond\Pi(\omu)$ is affected by (i). For $\Pi(2\langle n-2,n-2|\bullet\rangle^{{\rm ch}(\olambda)})\diamond\Pi(\omu)$, (i) affects only the term $2\okappa'$ (multiplying it by $\eta_{\olambda,\omu}$), while (ii) multiplies $2\okappa'$ by $\frac{d}{2}$. Therefore,
\begin{equation}\nonumber
\mathcal{R}(\tableau{{\ }},\olambda,\omu) - \Sigma'(\tableau{{\ }},\olambda,\omu) = d\cdot \eta_{\olambda,\omu}\okappa' - d\cdot \okappa'
\end{equation}
as in (\ref{eqn:RR1}), and we are done.
\end{proof}

\begin{Lemma}\label{lemma:boxdiamondstar}
If $\olambda, \omu$ is amenable and 
\begin{equation}\label{eqn:boxdiamondassoc}
\tableau{{\ }}\diamond(\Pi(\olambda)\diamond\Pi(\omu))=(\tableau{{\ }}\diamond\Pi(\olambda))\diamond\Pi(\omu)
\end{equation}
then (\ref{eqn:Dboxassoc}) holds.
\end{Lemma}
\begin{proof}
By Lemma~\ref{lemma:boxbalanced} both $\leftassboxB$ and $\rightassboxB$ are balanced. Hence it suffices to show $\mathcal{L}(\tableau{{\ }},\olambda,\omu)=\mathcal{R}(\tableau{{\ }},\olambda,\omu)$. By Lemma~\ref{lemma:boxparti} it suffices to show $\Omega'(\tableau{{\ }},\olambda,\omu)=\Sigma'(\tableau{{\ }},\olambda,\omu)$.

For any term $\okappa$ appearing in either $\Omega'(\tableau{{\ }},\olambda,\omu)$ or $\Sigma'(\tableau{{\ }},\olambda,\omu)$, since ${\rm fsh}(\tableau{{\ }})=0$ we can conclude exactly as in Lemma~\ref{lemma:diamondstar}, except 
where we replace $\sctableau{{\ }\\{\ }}$ by $\tableau{{\ }}$.
\end{proof}

The following lemma is clear from the definition of $\diamond$.

\begin{Lemma}\label{lemma:diamondLR}
Suppose $c\cdot \okappa$ is a term of $\Pi(\olambda)\diamond\Pi(\omu)$. Then if $\Pi(\olambda)\diamond\Pi(\omu)$ is computed by (A) or (C), $c=C_{\pi(\lambda),\pi(\mu)}^{\kappa}(Gr_2(\mathbb{C}^{2n-2}))$.
\end{Lemma}

By Lemma~\ref{lemma:boxdiamondstar}, it suffices to establish (\ref{eqn:boxdiamondassoc}). 

\noindent {\bf Case 1:} ($\olambda$ or $\omu$ is on, or $|\olambda|+|\omu|<2n-4$): If $1 + |\olambda|+|\omu|>4n-7 = |\Lambda_{OG(2,2n)}|$ then both sides of (\ref{eqn:boxdiamondassoc}) are $0$. Otherwise, both sides of (\ref{eqn:boxdiamondassoc}) are computed using only (A) and (C), so we are done by Lemma~\ref{lemma:diamondLR} and the associativity of the Littlewood-Richardson rule. This completes Case 1.

We need some preparation for the remaining cases. Define
\begin{equation}
\label{eqn:aa6}
M(i)=\min\{\lambda_1-\lambda_2+i, \mu_1-\mu_2\}.
\end{equation}
We recall definitions from the proofs of (\ref{eqn:dominoassoc}) and (\ref{eqn:boxassoc}). 

Let $s=1-\delta_{\lambda_1-\lambda_2,\mu_1-\mu_2}$ and $t=1-\delta_{\lambda_1,\lambda_2}$. For a given $\olambda$ and $\omu$, let
\begin{eqnarray}\nonumber
f(a,b) &=& \langle\lambda_1+\mu_1+a,\lambda_2+\mu_2+b|\bullet\rangle \\ \nonumber
F^{(i)}(a,b) &=& \sum_{1 \le k \le M(i)}\langle\lambda_1+\mu_1-k+a,\lambda_2+\mu_2+k+b|\bullet\rangle \\ \nonumber
F'^{(i)}(a,b) &=& \sum_{0 \le k \le M(i)}\langle\lambda_1+\mu_1-k+a,\lambda_2+\mu_2+k+b|\bullet\rangle\\ \nonumber
\end{eqnarray}
Furthermore, using the above expressions we define
\begin{eqnarray}\nonumber
J^+_L &=& f(1,-1)+2F^{(0)}(1,-1)+s\cdot f(-M(0),M(0))\\ \nonumber
J_{L +} &=& f(0,0)+2F^{(0)}(0,0)+s\cdot f(-M(0)-1,M(0)+1)\\ \nonumber
J^{(1)}_R &=& f(1,-1)+  2F^{(1)}(1,-1) + f(-M(1),M(1))\\ \nonumber
J_{R (1)} &=& t\cdot(f(0,0)+ 2F^{(-1)}(0,0) + f(-M(-1)-1,M(-1)+1))\nonumber
\end{eqnarray}

\noindent {\bf Case 2:} ($|\langle\lambda|\circ\rangle| + |\langle \mu|\circ\rangle| > 2n-4$, $|\langle\lambda|\circ\rangle| < 2n-4$): By (B), 
\[\Pi(\olambda)\diamond\Pi(\omu)=f(0,-1)+2F^{(0)}(0,-1)+s\cdot f(-M(0)-1,M(0))\]
and thus by (C), 
\[\tableau{{\ }}\diamond(\Pi(\olambda)\diamond\Pi(\omu))=J^+_L+J_{L +}.\] 
By (A), 
$\tableau{{\ }}\diamond \Pi(\olambda) = \langle \lambda_1+1,\lambda_2|\circ\rangle+ t\cdot\langle \lambda_1,\lambda_2+1|\circ\rangle$. Thus, by (B), 
\[(\tableau{{\ }}\diamond\Pi(\olambda))\diamond\Pi(\omu)=J^{(1)}_R+J_{R (1)}.\]

\begin{Lemma}\label{lemma:case3BoxassocA}
$J^+_L+J_{L +}=J^{(1)}_R+J_{R (1)}$.
\end{Lemma}
\noindent
\emph{Proof.}
If $t=0$, then $\lambda_1=\lambda_2$, so $M(0)=0$ and the statement is easily checked. If $t=1$, the claim holds by noting:
\begin{itemize}
\item If $\mu_1-\mu_2 < \lambda_1-\lambda_2$, then $M(1) = M(-1) = M(0)$ and $s=1$.
\item If $\mu_1-\mu_2 = \lambda_1-\lambda_2$, then $M(1) = M(0), M(-1) = M(0)-1$ and $s=0$.
\item If $\mu_1-\mu_2 > \lambda_1-\lambda_2$, then $M(1) = M(0)+1, M(-1) = M(0)-1$ and $s=1$.\qed
\end{itemize}

\noindent {\bf Case 3:} ($|\langle\lambda|\circ\rangle| + |\langle \mu|\circ\rangle|>2n-4$, $|\langle\lambda|\circ\rangle| = 2n-4$): Then exactly as in Case 2, by (B) and then (C) 
\[\tableau{{\ }}\diamond(\Pi(\olambda)\diamond\Pi(\omu))= J^+_L+J_{L +}.\] 
Let  $r'=1-\delta_{\lambda_2,0}$ and
\[K_R = r'\cdot F'^{(2)}(1,-1) + F'^{(0)}(0,0) + t\cdot (F'^{(0)}(0,0) + F'^{(-2)}(-1,1)).\]
By (B), 
\[\tableau{{\ }} \diamond \Pi(\olambda) = r'\cdot\langle \lambda_1+1,\lambda_2-1|\bullet\rangle + \langle\lambda_1,\lambda_2|\bullet\rangle + t\cdot(\langle\lambda_1,\lambda_2|\bullet\rangle + \langle\lambda_1-1,\lambda_2+1|\bullet\rangle).\] 
Then by (C),
\begin{equation}\label{eqn:boxcase3}
(\tableau{{\ }}\diamond\Pi(\olambda))\diamond\Pi(\omu)=K_R.
\end{equation}

\begin{Lemma}\label{lemma:case4BoxassocA}
$J^+_L+J_{L +}=K_R$.
\end{Lemma}
\begin{proof}
If $r'=0$ then $\lambda_1=2n-4$ and 
$M(2)=\min\{(2n-4)+2,\mu_1-\mu_2\}=\mu_1-\mu_2$.
Thus, each term $\langle 2n-4+\mu_1-k+1,\mu_2+k-1\bullet\rangle$
of $F'^{(2)}(1,-1)$ is illegal as $2n-4+\mu_2-k+1>2n-4$.
Therefore, we may ignore $r'$ in the precise sense that (\ref{eqn:boxcase3}) is still valid if $r'$ is set equal to $1$.

If $t=0$, then $M(0)=0$. Therefore $J^+_L = f(1,-1) + s\cdot f(0,0)$ and $J_{L+}=f(0,0)+s\cdot f(-1,1)$. On the other hand, $K_R = (\sum_{0\le k\le M(2)}\langle \lambda_1+\mu_1+1-k,\lambda_2+\mu_2-1+k) + f(0,0)$. We compare these in three cases:

($\mu_1-\mu_2=0$:) Then $s=0$, $M(2)=0$ and $J^+_L+J_{L +}=K_R = f(1,-1) + f(0,0)$.

($\mu_1-\mu_2=1$:) Then $s=1$ and $M(2)=1$. $J^+_L+J_{L +} = (f(1,-1) +f(0,0)) + (f(0,0) + f(-1,1))$, but $f(-1,1))$ is illegal. $K_R = (f(1,-1) + f(0,0)) + f(0,0)$.

($\mu_1-\mu_2\ge 2$:) Then $s=1$ and $M(2)=2$. $J^+_L+J_{L +} = (f(1,-1) +f(0,0)) + (f(0,0) + f(-1,1))$, and $K_R = (f(1,-1) + f(0,0) + f(-1,1)) + f(0,0)$.

Thus assume $t=1$. We break the proof into five cases:

(The ``main case'': $\lambda_1-\lambda_2 - (\mu_1-\mu_2) \ge 2$):
Then $M(2) = M(-2)=M(0)$. We have
\begin{multline}
\label{eqn:Case4BoxassocAfivecases1}
F'^{(2)}(1,-1)+ F'^{(-2)}(-1,1) = f(1,-1)+f(0,0) + f(-M(0),M(0)) \\ 
+f(-M(0)-1,M(0)+1) + 2 \sum_{2\le k \le M(0)}\langle\lambda_1+\mu_1-k+1,\lambda_2+\mu_2+k-1|\bullet\rangle \mbox{,  and}
\end{multline}
\begin{equation}
\label{eqn:Case4BoxassocAfivecases2}
2F'^{(0)}(0,0) = 2F^{(0)}(1,-1) + 2F^{(0)}(0,0) - 2\sum_{2\le k\le M(0)}\langle \lambda_1+\mu_1-k+1,\lambda_2+\mu_2+k-1|\bullet\rangle.
\end{equation}
Hence combining (\ref{eqn:Case4BoxassocAfivecases1}) and (\ref{eqn:Case4BoxassocAfivecases2})
we are done (note $s=1$ in this case of the lemma). 

($\lambda_1-\lambda_2 - (\mu_1-\mu_2)=1$): Here $M(2) = M(0)$ and $M(-2)=M(0)-1$. The
change from the main case is that in (\ref{eqn:Case4BoxassocAfivecases1}), 
the term $f(-M(0)-1,M(0)+1)$ does not appear. However, in $J_{L+}$, $f(-M(0)-1,M(0)+1)$ is illegal anyway. Now we conclude as before.

($\lambda_1-\lambda_2=\mu_1-\mu_2$): Now $M(2) = M(0)$ and $M(-2)=M(0)-2$.
The change from the main case is that in (\ref{eqn:Case4BoxassocAfivecases1})
the terms $f(-M(0),M(0))$ and $f(-M(0)-1,M(0)+1)$ do not appear. But here, $s=0$ so those also terms do not appear in $J^+_L+J_{L +}$.

($\mu_1-\mu_2 - (\lambda_1-\lambda_2)=1$): So $M(2) = M(0)+1$ and $M(-2)=M(0)-2$. Here the change from
the main case is that in (\ref{eqn:Case4BoxassocAfivecases1}), the term $f(-M(0)-1,M(0)+1)$ does not appear (although for different reasons than in the second case above). 
Anyway, $f(-M(0)-1,M(0)+1)$ is illegal in $J_{L +}$.

($\mu_1-\mu_2 - (\lambda_1-\lambda_2)\ge 2$): Finally,
$M(2) = M(0)+2$ and $M(-2)=M(0)-2$. Both $J^+_L+J_{L +}$ and $K_R$ are the same as in the main case, and we are done as in that case.
\end{proof}

\noindent {\bf Case 4:} ($|\langle\lambda|\circ\rangle| +|\langle \mu|\circ\rangle|=2n-4$, $|\langle\lambda|\circ\rangle| < 2n-4$): Exactly as in Case 2, by (A) and (B) 
\[(\tableau{{\ }}\diamond\Pi(\olambda))\diamond\Pi(\omu)=J_R^{(1)}+J_{R (1)}.\]
Let
\[G'^{(0)}(a,b) = \sum_{0\le k\le M(0)-1}\langle \lambda_1+\mu_1-k+a,\lambda_2+\mu_2+k+b|\bullet\rangle + \tilde{s}\cdot f(-M(0)+a,M(0)+b),\]
\[ \mbox{where \ } \tilde{s}=\begin{cases} 0 & \mbox{ \ if $\langle n-2,n-2|\circ\rangle$ is a term in $\Pi(\olambda)\diamond\Pi(\omu)$}\\
1 & \mbox{\  otherwise.}\end{cases}\]
Also, define
\[K_L = F'^{(0)}(1,-1)+F'^{(0)}(0,0) + G'^{(0)}(0,0) + G'^{(0)}(-1,1).\]
Then by (A), 
\[\Pi(\olambda)\diamond\Pi(\omu) = \sum_{0\le k \le M(0)} \langle\lambda_1+\mu_1-k,\lambda_2+\mu_2+k|\circ\rangle.\] 
Thus, by (B), 
\[\tableau{{\ }}\diamond(\Pi(\olambda)\diamond\Pi(\omu))=K_L.\]

\begin{Claim}\label{claim:case5BoxassocA}
$\tilde{s}=s$.
\end{Claim}
\begin{proof}
By the assumption of Case 4, $\lambda_1+\lambda_2+\mu_1+\mu_2=2n-4$. By (A), if $\langle n-2,n-2|\circ\rangle$ is a term in $\Pi(\olambda)\diamond\Pi(\omu)$, it is $\langle\lambda_1+\mu_1-M(0), \lambda_2+\mu_2+M(0)|\circ\rangle$. Then
$s=0 \Leftrightarrow \lambda_1-\lambda_2=\mu_1-\mu_2 \Leftrightarrow \lambda_1+\mu_2 = \lambda_2+\mu_1 = n-2
\Leftrightarrow (\lambda_1+\mu_1-M(0),\lambda_2+\mu_2+M(0))=(n-2,n-2) \Leftrightarrow \tilde{s}=0.$
\end{proof}

\begin{Claim}\label{claim:case5BoxassocB}
$K_L = J_R^{(1)}+J_{R (1)}$.
\end{Claim}
\begin{proof}
Since by Claim~\ref{claim:case5BoxassocA} we have $\tilde{s}=s$, $K_L$ rearranges easily to obtain $K_L=J_L^+ + J_{L +}$. 
Then we use Lemma~\ref{lemma:case3BoxassocA}, which shows $J_L^+ + J_{L +} = J_R^{(1)}+J_{R (1)}$.
\end{proof}

\subsubsection{Proof of (\ref{eqn:Dboxassoc}) for non-amenable pairs $\olambda, \omu$}
We assumed throughout that $\omu$ is a non-Pieri shape.
It remains to check cases where  $\olambda$ is Pieri and $\omu$ is charged and non-Pieri. 

The following observation is clear from the definition of $\diamond$ and of $M$:
\begin{Lemma}
\label{lemma:firstrow}
If $\okappa=\langle\kappa_1,\kappa_2|{\bullet}{/\circ}\rangle$ appears in $\oalpha\diamond\obeta$ then
\[\kappa_1\geq
\begin{cases}
\max(\alpha_1+\beta_2,\alpha_2+\beta_1) &
\mbox{if  $\oalpha\diamond\obeta$ is described by (A) or (C)}\\
\max(\alpha_1+\beta_2,\alpha_2+\beta_1)-1 & \mbox{if
 $\oalpha\diamond\obeta$ is described by (B)}
\end{cases}\]
\end{Lemma}

\noindent {\bf Case 1:} ($\olambda$ is a charged Pieri shape): Then $\olambda=\langle n-2,0|\circ\rangle^{{\rm ch}(\olambda)}$, and
\begin{equation}\label{eqn:boxtimesolambda1}
\tableau{{\ }}\star\olambda=\langle n-1,0|\circ\rangle + \langle n-2,1|\circ\rangle^{{\rm ch}(\olambda)}.
\end{equation}
\noindent{\sf ($\omu$ is on):} Then $\mu_2=n-2$. We compute $\olambda\star\omu = \frac{1}{2}\eta_{\olambda,\omu}\langle 2n-4,\mu_1|\bullet\rangle$ and $\leftassboxB=\frac{1}{2}\eta_{\olambda,\omu}\langle 2n-4,\mu_1+1|\bullet\rangle$ (neutral). For $\rightassboxB$, we compute $\langle n-1,0|\circ\rangle\star\omu=0$ (Lemma~\ref{lemma:firstrow}), and $\langle n-2,1|\circ\rangle^{{\rm ch}(\olambda)}\star\omu=\frac{1}{2}\eta_{\olambda,\omu}\langle 2n-4,\mu_1+1|\bullet\rangle$. Thus $\leftassboxB=\rightassboxB$.

\noindent{\sf ($\omu$ is off):} Then $\mu_1=n-2$. Since $\mu_2>0$ ($\omu$ is non-Pieri), $\Pi(\olambda)\diamond\Pi(\omu)$ is given by (B). So,
\[\olambda\star\omu=\frac{1}{2}\eta_{\olambda,\omu}\langle 2n-4,\mu_2-1|\bullet\rangle + \sum_{1\le k \le n-2-\mu_2}\langle 2n-4-k,\mu_2-1+k|\bullet\rangle + \langle n-2+\mu_2-1,n-2|\bullet\rangle^{{\rm ch}(\omu)}.\]
Then if $\mu_2<n-2$,
\begin{multline}\label{eqn:charged100}
\leftassboxB=\frac{1}{2}\eta_{\olambda,\omu}\langle 2n-4,\mu_2|\bullet\rangle + \langle 2n-4,\mu_2|\bullet\rangle + \sum_{1\le k \le n-2-\mu_2-1}\!\!\!\!\!\!\!\!\! 2\langle  2n-4-k,\mu_2+k|\bullet\rangle \\ 
+ \langle n-2 + \mu_2, n-2|\bullet\rangle^\uparrow + \langle n-2 + \mu_2, n-2|\bullet\rangle^\downarrow + \langle n-2 + \mu_2, n-2|\bullet\rangle^{{\rm ch}(\omu)} + \langle n-2+\mu_2-1,n-1|\bullet\rangle. 
\end{multline}
If instead $\mu_2=n-2$,
\begin{multline}\label{eqn:charged101} 
\leftassboxB=\eta_{\olambda,\omu}\langle 2n-4,n-2|\bullet\rangle + \langle 2n-4,n-2|\bullet\rangle^{{\rm ch}(\omu)} + \langle 2n-5,n-1|\bullet\rangle
\end{multline}
where the first term splits if $\eta_{\olambda,\omu}=2$.
Next, compute the $\star$-product of (\ref{eqn:boxtimesolambda1}) and $\omu$:
\begin{multline}\label{eqn:charged102}
\langle n-1,0|\circ\rangle\star\omu = \sum_{0 \le k \le n-2-\mu_2-1}\langle 2n-4-k,\mu_2+k|\bullet\rangle + \langle n-2 +\mu_2,n-2|\bullet\rangle^{{\rm ch}(\omu)}.
\end{multline}
If $\mu_2<n-2$,
\begin{multline}\label{eqn:charged103}
\langle n-2,1|\circ\rangle^{{\rm ch}(\olambda)}\star\omu = \frac{1}{2}\eta_{\olambda,\omu}\langle 2n-4,\mu_2|\bullet\rangle + \sum_{1\le k\le n-2-\mu_2-1}\langle 2n-4-k,\mu_2+k|\bullet\rangle \\ 
+ \langle n-2+\mu_2,n-2|\bullet\rangle^\uparrow + \langle n-2+\mu_2,n-2|\bullet\rangle^\downarrow +  \langle n-2+\mu_2-1,n-1|\bullet\rangle. 
\end{multline}
If instead $\mu_2=n-2$,
\begin{equation}\label{eqn:charged104}
\langle n-2,1|\circ\rangle^{{\rm ch}(\olambda)}\star\omu = \eta_{\olambda,\omu}\langle 2n-4,n-2|\bullet\rangle + \langle 2n-5,n-1|\bullet\rangle
\end{equation}
where the first term splits if $\eta_{\olambda,\omu}=2$. 
Comparing (\ref{eqn:charged100}) with the sum of (\ref{eqn:charged102}) and (\ref{eqn:charged103}), and comparing (\ref{eqn:charged101}) with the sum of (\ref{eqn:charged102}) and (\ref{eqn:charged104}), we have $\leftassboxB=\rightassboxB$.

\noindent {\bf Case 2:} ($\olambda$ is a neutral Pieri shape): There are three subcases: 

\noindent \emph{Subcase 2a:} ($\lambda_1\ge n-1$): First suppose either $\omu$ is on, or $\olambda = \langle 2n-4,0|\bullet\rangle$. We have $\olambda\star\omu=0$ by Lemma~\ref{lemma:firstrow}, and thus $\leftassboxB=0$. If $\tableau{{\ }}\diamond\Pi(\olambda)$ is computed by (A) or (C), then every term $\okappa$ of $\tableau{{\ }}\diamond\Pi(\olambda)$ contains $\Pi(\olambda)$ and thus also $\rightassboxB=0$ by Lemma~\ref{lemma:firstrow}. If $\tableau{{\ }}\diamond\Pi(\olambda)$ is computed by (B), then $\olambda$ is off (so $\omu$ is on, by assumption) but every term of $\tableau{{\ }}\diamond\Pi(\olambda)$ is on, so $\rightassboxB=0$ by (D). 

Now suppose $\olambda$ and $\omu$ are off. Since both $|\lambda|>n-2$ and $|\mu|>n-2$, $\Pi(\olambda)\diamond\Pi(\omu)$ is computed by (B). We have
\[\olambda\star\omu=\sum_{\lambda_1-(n-2) \le k \le n-2-\mu_2 }\langle \lambda_1+n-2-k,\mu_2+k-1|\bullet\rangle + \langle \lambda_1+\mu_2-1,n-2|\bullet\rangle^{{\rm ch}(\omu)}.\]
Then 
\begin{multline}\label{eqn:bigpieri100}
\leftassboxB=\sum_{\lambda_1-(n-2) \le k \le n-2-\mu_2-1 }2\langle\lambda_1+n-2-k,\mu_2+k|\bullet\rangle + \langle \lambda_1+\mu_2,n-2|\bullet\rangle^{{\rm ch}(\omu)}\\  
+ \langle\lambda_1+\mu_2,n-2|\bullet\rangle^\uparrow + \langle\lambda_1+\mu_2,n-2|\bullet\rangle^\downarrow + \langle\lambda_1+\mu_2-1,n-1|\bullet\rangle.
\end{multline}
If $\olambda=\langle 2n-4,0|\circ\rangle$, then $\tableau{{\ }}\star\olambda = 2\langle 2n-4,0|\bullet\rangle+\langle 2n-5,1|\bullet\rangle$. Then $\rightassboxB = \langle 2n-5+\mu_2,n-1|\bullet\rangle$. This is equal to (\ref{eqn:bigpieri100}) when $\lambda_1=2n-4$; in particular, every term of (\ref{eqn:bigpieri100}) except possibly the last one is necessarily illegal.

Otherwise, $\lambda_1<2n-4$. Hence $\tableau{{\ }}\star\olambda = \langle \lambda_1+1,0|\circ\rangle+\langle \lambda_1,1|\circ\rangle$ and
\begin{multline}\label{eqn:bigpieri101}
\langle \lambda_1+1,0|\circ\rangle\star\omu = \!\!\!\!\!\!\sum_{\lambda_1-(n-2) \le k \le n-2-\mu_2-1 }\!\!\!\!\!\!\!\!\!\!\!\!\!\!\!\!\langle \lambda_1+n-2-k,\mu_2+k|\bullet\rangle + \langle \lambda_1+\mu_2,n-2|\bullet\rangle^{{\rm ch}(\omu)} \mbox{ \ and \ }
\end{multline}
\begin{multline}\label{eqn:bigpieri102}
\langle \lambda_1,1|\circ\rangle\star\omu = \sum_{\lambda_1-(n-2) \le k \le n-2-\mu_2-1 }\langle \lambda_1+n-2-k,\mu_2+k|\bullet\rangle \\ 
+ \langle \lambda_1+\mu_2,n-2|\bullet\rangle^\uparrow + \langle \lambda_1+\mu_2,n-2|\bullet\rangle^\downarrow + \langle \lambda_1+\mu_2-1,n-1|\bullet\rangle. 
\end{multline}
Comparing (\ref{eqn:bigpieri100}) with the sum of (\ref{eqn:bigpieri101}) and (\ref{eqn:bigpieri102}), we have $\leftassboxB=\rightassboxB$. This concludes the case where $\olambda$ is a Pieri shape with $\lambda_1 \ge n-1$.

\noindent \emph{Subcase 2b:} ($\olambda= \langle n-3,0|\circ\rangle$): Then
\begin{equation}\label{eqn:boxtimesolambda2b}
\tableau{{\ }}\star\olambda = \langle n-2,0|\circ\rangle^\uparrow + \langle n-2,0|\circ\rangle^\downarrow + \langle n-3,1|\circ\rangle.
\end{equation}

\noindent{\sf ($\omu$ is on):} Then $\mu_2=n-2$. If $\mu_1=2n-4$, we compute 
\[\leftassboxB=\rightassboxB=\langle 2n-4,2n-4|\bullet\rangle.\]
If $\mu_1=2n-5$, we compute
\[\leftassboxB=\rightassboxB=2\langle 2n-4,2n-5|\bullet\rangle.\]
If $\mu_1=n-2$, we compute 
\[\leftassboxB=\rightassboxB=\langle 2n-4,n-2|\bullet\rangle^{{\rm ch}(\omu)} + \langle 2n-5,n-1|\bullet\rangle.\] 
Otherwise, $n-2<\mu_1<2n-5$ and we compute
\[\leftassboxB=\rightassboxB = 2\langle 2n-4,\mu_1|\bullet\rangle + \langle 2n-5,\mu_1+1|\bullet\rangle.\]

\noindent{\sf ($\omu$ is off):} Then $\mu_1=n-2$. There are two cases:

\noindent
{\sf ($\mu_2=1$):} Then $\Pi(\olambda)\diamond\Pi(\omu)$ is given by (A), and
\[\olambda\star\omu = \sum_{1\le k \le n-3} \langle 2n-4-k,k|\circ\rangle + \langle n-2,n-2|\circ\rangle^{{\rm ch}(\omu)}.\]
Consider:
\begin{multline}\label{eqn:medpieri100}
\leftassboxB = \langle 2n-4,0|\bullet\rangle + 3\langle 2n-5,1|\bullet\rangle + \sum_{2\le k \le n-3}4\langle 2n-4-k,k|\bullet\rangle \\  
+ 2\langle n-2,n-2|\bullet\rangle^{{\rm ch}(\omu)} + \langle n-2,n-2|\bullet\rangle^\uparrow + \langle n-2,n-2|\bullet\rangle^\downarrow;
\end{multline}
For $\rightassboxB$, we compute the $\star$-product of (\ref{eqn:boxtimesolambda2b}) and $\omu$:
\begin{multline}\label{eqn:medpieri101}
(\langle n-2,0|\circ\rangle^\uparrow + \langle n-2,0|\circ\rangle^\downarrow)\star \omu = \langle 2n-4,0|\bullet \rangle + \sum_{1\le k \le n-3}2\langle 2n-4-k,k|\bullet\rangle \\ 
+ 2\langle n-2,n-2|\bullet\rangle^{{\rm ch}(\omu)};
\end{multline}
\begin{multline}\label{eqn:medpieri102}
\langle n-3,1|\circ\rangle\star\omu = \langle 2n-5,1|\bullet\rangle + \sum_{2\le k\le n-3}2\langle 2n-4-k,k|\bullet\rangle \\  
+ \langle n-2,n-2|\bullet\rangle^\uparrow + \langle n-2,n-2|\bullet\rangle^\downarrow. 
\end{multline}
Comparing (\ref{eqn:medpieri100}) with the sum of (\ref{eqn:medpieri101}) and (\ref{eqn:medpieri102}), we have $\leftassboxB=\rightassboxB$.

\noindent
{\sf ($\mu_2>1$):} Then $\Pi(\olambda)\diamond\Pi(\omu)$ is given by (B), and
\[\olambda\star\omu = \langle 2n-5,\mu_2-1|\bullet\rangle + \sum_{1\le k \le n-2-\mu_2}2\langle 2n-5-k, \mu_2-1+k|\bullet\rangle +  2\langle n-2+\mu_2-2,n-2|\bullet\rangle^{{\rm ch}(\omu)}.\]
If $\mu_2\neq n-2$ then
\begin{multline}\label{eqn:medpieri103}
\leftassboxB = \langle 2n-4,\mu_2-1|\bullet\rangle + 3\langle 2n-5, \mu_2|\bullet\rangle + \sum_{1 \le k \le n-3-\mu_2}\!\!\!\!\!4\langle 2n-5-k,\mu_2+k|\bullet\rangle +  \\ 
2\langle n-2+\mu_2-1,n-2|\bullet\rangle^{{\rm ch}(\omu)} + 2\langle n-2+\mu_2-1,n-2|\bullet\rangle^\uparrow + 2\langle n-2+\mu_2-1,n-2|\bullet\rangle^\downarrow\\ + 2\langle n-2+\mu_2-2,n-1|\bullet\rangle, 
\end{multline}
whereas if $\mu_2=n-2$ then
\begin{multline}\label{eqn:medpieri104}
\leftassboxB = \langle 2n-4,n-3|\bullet\rangle + 2\langle 2n-5,n-2|\bullet\rangle^{{\rm ch}(\omu)} \\ 
+ \langle 2n-5,n-2|\bullet\rangle^\uparrow + \langle 2n-5,n-2|\bullet\rangle^\downarrow + 2\langle 2n-6,n-1|\bullet\rangle; 
\end{multline}
For $\rightassboxB$, we compute the $\star$-product of (\ref{eqn:boxtimesolambda2b}) and $\omu$:
\begin{multline}\label{eqn:medpieri105}
(\langle n-2,0|\circ\rangle^\uparrow + \langle n-2,0|\circ\rangle^\downarrow)\star \omu = \langle 2n-4, \mu_2-1|\bullet\rangle + \!\!\sum_{0\le k \le n-3-\mu_2}\!\!\!\!\!\!\!\!2\langle 2n-5-k,\mu_2+k|\bullet\rangle \\ 
+ 2\langle n-2+\mu_2-1,n-2|\bullet\rangle^{{\rm ch}(\omu)}. 
\end{multline}
If $\mu_2\neq n-2$ then
\begin{multline}\label{eqn:medpieri106}
\langle n-3,1|\circ\rangle\star\omu = \langle 2n-5,\mu_2|\bullet\rangle + \sum_{1\le k \le n-3-\mu_2}2\langle 2n-5-k,\mu_2+k|\bullet\rangle \\ 
+ 2\langle n-2+\mu_2-1,n-2|\bullet\rangle^\uparrow + 2\langle n-2+\mu_2-1,n-2|\bullet\rangle^\downarrow + 2\langle n-2+\mu_2-2,n-1|\bullet\rangle 
\end{multline}
while if $\mu_2=n-2$ then
\begin{multline}\label{eqn:medpieri107}
\langle n-3,1|\circ\rangle\star\omu = \langle 2n-5,n-2|\bullet\rangle^\uparrow + \langle 2n-5,n-2|\bullet\rangle^\downarrow + 2\langle 2n-6,n-1|\bullet\rangle.
\end{multline}
Comparing (\ref{eqn:medpieri103}) with the sum of (\ref{eqn:medpieri105}) and (\ref{eqn:medpieri106}), and comparing (\ref{eqn:medpieri104}) with the sum of (\ref{eqn:medpieri105}) and (\ref{eqn:medpieri107}), we have $\leftassboxB=\rightassboxB$. 

\noindent \emph{Subcase 2c:} ($\olambda$ is a neutral Pieri shape and $\lambda_1< n-3$): Then we always have
\begin{equation}\label{eqn:boxtimesolambda2c}
\tableau{{\ }}\star\olambda = \langle\lambda_1+1,0|\circ\rangle + \langle\lambda_1,1|\circ\rangle \mbox{\ (both neutral). \ }
\end{equation}
\noindent{\sf ($\omu$ is on):} Then $\mu_2=n-2$. Here $M(i)=\min\{\lambda_1+i,\mu_1-(n-2)\}$. We compute $\olambda\star\omu=\sum_{0\le k\le M(0)}\langle \lambda_1+\mu_1-k,n-2+k|\bullet\rangle$, where the $k=0$ term, if legal, is assigned ${\rm ch}(\omu)$. Then 
\begin{multline}\label{eqn:smallpieri100}
\leftassboxB=\langle \lambda_1+\mu_1+1,n-2|\bullet\rangle^{{\rm ch}(\omu)} + \sum_{1\le k \le M(0)}2\langle \lambda_1+\mu_1+1-k,n-2+k|\bullet\rangle \\
+ \langle \lambda_1+\mu_1-M(0),n-2+M(0)+1|\bullet\rangle;
\end{multline}
For $\rightassboxB$, we compute the $\star$-product of (\ref{eqn:boxtimesolambda2c}) and $\omu$:
\begin{multline}\label{eqn:smallpieri101}
\langle\lambda_1+1,0|\circ\rangle\star\omu = \!\langle \lambda_1+\mu_1+1,n-2|\bullet\rangle^{{\rm ch}(\omu)} \!+ \!\!\!\!\!\!\sum_{1\le k \le M(1)}\!\!\!\!\!\langle \lambda_1\!+\!\mu_1+1-k,n-2+k|\bullet\rangle \mbox{ \ and \ }
\end{multline}
\begin{equation}\label{eqn:smallpieri102}
\langle\lambda_1,1|\circ\rangle\star\omu = \sum_{0\le k \le M(-1)}\langle\lambda_1+\mu_1-k,n-2+k+1|\bullet\rangle.
\end{equation}
$\leftassboxB=\rightassboxB$ holds by comparing (\ref{eqn:smallpieri100}) with the sum of (\ref{eqn:smallpieri101}) and (\ref{eqn:smallpieri102}) and noting:
\begin{itemize}
\item If $\mu_1-(n-2)<\lambda_1$, then $M(1)=M(-1)=M(0)$;
\item If $\mu_1-(n-2)=\lambda_1$, then $M(1)=M(0)$, $M(-1)=M(0)-1$ and the last term of $\leftassboxB$ is illegal;
\item If $\mu_1-(n-2)>\lambda_1$, then $M(1)=M(0)+1$ and $M(-1)=M(0)-1$. 
\end{itemize}

\noindent{\sf ($\omu$ is off):} Then $\mu_1=n-2$. There are three cases:

{\sf ($\lambda_1+\mu_2<n-2$):} Then $\Pi(\olambda)\diamond\Pi(\omu)$ is computed by (A) and $\lambda_1<n-2-\mu_2$, so
\[\olambda\star\omu = \sum_{0\le k \le \lambda_1}\langle n-2+\lambda_1-k,\mu_2+k|\circ \rangle \]
where the $k=\lambda_1$ term is assigned ${\rm ch}(\omu)$. Then
\begin{equation}\label{eqn:smallpieri103}
\leftassboxB = \sum_{0\le k \le \lambda_1} \langle n-2+\lambda_1+1-k,\mu_2+k|\circ\rangle + \sum_{0 \le k \le \lambda_1}\langle n-2+\lambda_1-k,\mu_2+k+1|\circ\rangle
\end{equation}
where the $k=\lambda_1$ term in the second summation is assigned ${\rm ch}(\omu)$. For $\rightassboxB$, we compute the $\star$-product of (\ref{eqn:boxtimesolambda2c}) and $\omu$:
\begin{equation}\label{eqn:smallpieri104}
\langle\lambda_1+1,0|\circ\rangle\star\omu = \sum_{0\le k \le \lambda_1+1}\langle n-2+\lambda_1+1-k,\mu_2+k|\circ\rangle
\end{equation}
where the $k=\lambda_1 +1$ term is assigned ${\rm ch}(\omu)$, and
\begin{equation}\label{eqn:smallpieri105}
\langle\lambda_1,1|\circ\rangle\star\omu = \sum_{0\le k \le \lambda_1-1}\langle n-2+\lambda_1-k,\mu_2+1+k|\circ\rangle.
\end{equation}
Comparing (\ref{eqn:smallpieri103}) with the sum of (\ref{eqn:smallpieri104}) and (\ref{eqn:smallpieri105}), we have $\leftassboxB=\rightassboxB$.

{\sf ($\lambda_1+\mu_2=n-2$):} Then $\Pi(\olambda)\diamond\Pi(\omu)$ is computed by (A) and
\[\olambda\star\omu = \sum_{0\le k \le \lambda_1}\langle n-2+\lambda_1-k,\mu_2+k|\circ\rangle\]
where the $k=\lambda_1$ term is assigned ${\rm ch}(\omu)$. Using (B), we compute
\begin{multline}\label{eqn:smallpieri106}
\leftassboxB= \langle n-2+\lambda_1+1,\mu_2-1|\bullet\rangle + 3\langle n-2+\lambda_1,\mu_2|\bullet\rangle +   \\ 
\sum_{1 \le k \le \lambda_1-1} 4\langle n-2+\lambda_1-k,\mu_2+k|\bullet\rangle 
+ 2\langle n-2,n-2|\bullet\rangle^{{\rm ch}(\omu)} + \langle n-2,n-2|\bullet\rangle^\uparrow + \langle n-2,n-2|\bullet\rangle^\downarrow. 
\end{multline}
For $\rightassboxB$, we compute the $\star$-product of (\ref{eqn:boxtimesolambda2c}) and $\omu$:
\begin{multline}\label{eqn:smallpieri107}
\langle\lambda_1+1,0|\circ\rangle\star\omu = \langle n-2+\lambda_1+1,\mu_2-1|\bullet\rangle + \\ 
\sum_{0\le k \le \lambda_1-1}2\langle n-2+\lambda_1-k,\mu_2+k|\bullet\rangle + 2\langle n-2,n-2|\bullet\rangle^{{\rm ch}(\omu)} \mbox{ \ and \ } 
\end{multline}
\begin{multline}\label{eqn:smallpieri108}
\langle\lambda_1,1|\circ\rangle\star\omu = \langle n-2+\lambda_1,\mu_2|\bullet\rangle + \sum_{1\le k \le \lambda_1-1}2\langle n-2+\lambda_1-k,\mu_2+k|\bullet\rangle \\ 
+ \langle n-2,n-2|\bullet\rangle^\uparrow + \langle n-2,n-2|\bullet\rangle^\downarrow. 
\end{multline}
Comparing (\ref{eqn:smallpieri106}) with the sum of (\ref{eqn:smallpieri107}) and (\ref{eqn:smallpieri108}), we have $\leftassboxB=\rightassboxB$.

{\sf ($\lambda_1+\mu_2>n-2$):} Then $\Pi(\olambda)\diamond\Pi(\omu)$ is computed by (B) and $n-2-\mu_2<\lambda_1$, so
\[\olambda\star\omu = \langle n-2+\lambda_1,\mu_2-1|\bullet\rangle + \sum_{1 \le k \le n-2-\mu_2}2\langle n-2+\lambda_1-k,\mu_2+k-1|\bullet\rangle + 2\langle \lambda_1+\mu_2-1,n-2|\bullet\rangle^{{\rm ch}(\omu)}.\]
If $\mu_2<n-2$, we compute
\begin{multline}\label{eqn:smallpieri109}
\leftassboxB = \langle n-2 + \lambda_1+1,\mu_2-1|\bullet\rangle + 3\langle n-2 + \lambda_1,\mu_2|\bullet\rangle\\ 
+ \sum_{1\le k \le n-2-\mu_2-1} 4\langle n-2+ \lambda_1-k,\mu_2+k|\bullet\rangle  + 2\langle \lambda_1+\mu_2,n-2|\bullet\rangle^{{\rm ch}(\omu)} + 2\langle \lambda_1+\mu_2,n-2|\bullet\rangle^\uparrow\\ + 2\langle \lambda_1+\mu_2,n-2|\bullet\rangle^\downarrow + 2\langle \lambda_1+\mu_2-1,n-1|\bullet\rangle. 
\end{multline}
Otherwise $\mu_2=n-2$ and we compute
\begin{multline}\label{eqn:smallpieri110}
\leftassboxB = \langle n-2+ \lambda_1+1,n-3|\bullet\rangle +  2\langle n-2+ \lambda_1,n-2|\bullet\rangle^{{\rm ch}(\omu)} \\ 
+ \langle n-2+ \lambda_1,n-2|\bullet\rangle^\uparrow + \langle n-2+ \lambda_1,n-2|\bullet\rangle^\downarrow + 2\langle n-2+ \lambda_1-1,n-1|\bullet\rangle. 
\end{multline}
For $\rightassboxB$, we compute the $\star$-product of (\ref{eqn:boxtimesolambda2c}) and $\omu$:
\begin{multline}\label{eqn:smallpieri111} 
\langle\lambda_1+1,0|\circ\rangle\star\omu = \langle n-2+\lambda_1+1,\mu_2-1|\bullet\rangle + \\ 
\sum_{0\le k \le n-2-\mu_2-1}2\langle n-2+\lambda_1-k,\mu_2+k|\bullet\rangle + 2\langle \lambda_1+\mu_2,n-2|\bullet\rangle^{{\rm ch}(\omu)}. 
\end{multline}
If $\mu_2<n-2$, we compute
\begin{multline}\label{eqn:smallpieri112}
\langle\lambda_1,1|\circ\rangle\star\omu = \langle n-2+\lambda_1,\mu_2|\bullet\rangle + \sum_{1\le k\le n-2-\mu_2-1}2\langle n-2+\lambda_1-k,\mu_2+k|\bullet\rangle \\ 
+ 2\langle \lambda_1+\mu_2,n-2|\bullet\rangle^\uparrow + 2\langle \lambda_1+\mu_2,n-2|\bullet\rangle^\downarrow + 2\langle \lambda_1+\mu_2-1,n-1|\bullet\rangle. 
\end{multline}
Otherwise $\mu_2=n-2$ and we compute
\begin{equation}\label{eqn:smallpieri113}
\langle\lambda_1,1|\circ\rangle\star\omu = \langle n-2+\lambda_1,n-2|\bullet\rangle^\uparrow + \langle n-2+\lambda_1,n-2|\bullet\rangle^\downarrow + 2\langle n-2+\lambda_1-1,n-1|\bullet\rangle.
\end{equation}

Comparing (\ref{eqn:smallpieri109}) with the sum of (\ref{eqn:smallpieri111}) and (\ref{eqn:smallpieri112}), and comparing (\ref{eqn:smallpieri110}) with the sum of (\ref{eqn:smallpieri111}) and (\ref{eqn:smallpieri113}), we have $\leftassboxB=\rightassboxB$. This concludes the proof of (\ref{eqn:Dboxassoc}).
 
\subsection{Proof of relation (\ref{eqn:Dambigassoc})} We split our argument into two main cases.

\subsubsection{At least one of $\olambda$ or $\omu$ is neutral}
Since we know (\ref{eqn:Ddominoassoc}) and (\ref{eqn:Dboxassoc}), we can show:
\begin{Lemma}
\label{lemma:partialgen}
Every neutral shape, as well as the elements
\[\alpha_{(\lambda_2)}:=\langle n-2,\lambda_2 | \circ\rangle^\uparrow + \langle n-2,\lambda_2 | \circ\rangle^\downarrow
\mbox{ \  for $0 \le \lambda_2 \le n-2$, and}
\]
\[\alpha^{(\lambda_1)}:=\langle\lambda_1,n-2 | \bullet\rangle^\uparrow + \langle\lambda_1,n-2 |\bullet\rangle^\downarrow \mbox{ \
for $n-2 \le \lambda_1 \le 2n-4$}\]
have an expression as a ${\mathbb Q}$-linear sum of $\star$-monomials in $\tableau{{\ }}$ and $\sctableau{{\ }\\{\ }}$. 
Any such expression is well-defined, i.e., all associative bracketings of these monomials are equivalent. 
\end{Lemma}
\begin{proof}
{\bf Case 1:} (generating $\alpha_{(\lambda_2)}$'s): 
In view of Lemma~\ref{lemma:diamondLR}, given $\tableau{{\ }}$ and $\sctableau{{\ }\\{\ }}$, we can generate any $\olambda$ such that
\begin{equation}\label{eqn:aaa233}
\lambda \subseteq (n-3,n-3). 
\end{equation}
We thus obtain
\[\alpha_{(\lambda_2)}=\tableau{{\ }}\star\langle n-3,\lambda_2|\circ\rangle -\langle n-3, \lambda_2+1|\circ\rangle, \mbox{ \ for all $0 \le \lambda_2 < n-3$;}\]
\[\alpha_{(n-3)}=\tableau{{\ }}\star\langle n-3,n-3|\circ\rangle \mbox{ \ and \ } \alpha_{(n-2)}=\sctableau{{\ }\\{\ }}\star\langle n-3,n-3|\circ\rangle.\]

\noindent
{\bf Case 2:} (generating neutral $\omu=\langle \mu|\circ\rangle$): Since $\tableau{{\ }} \star \alpha_{(0)} = \alpha_{(1)} + 2\langle n-1,0|\circ\rangle \mbox{\ \ and \ \ } \sctableau{{\ }\\{\ }} \star \alpha_{(0)} = 2\langle n-1,1|\circ\rangle$,
we obtain $\langle n-1,0|\circ\rangle$ and $\langle n-1,1|\circ\rangle$. Then
$\tableau{{\ }} \star \langle n-1,0|\circ\rangle = \langle n,0|\circ\rangle + \langle n-1,1|\circ\rangle$,
generates $\langle n,0|\circ\rangle$. Since
$\langle k+1,0|\circ\rangle=\tableau{{\ }}\star\langle k,0|\circ\rangle-\sctableau{{\ }\\{\ }}\star\langle k-1,0|\circ\rangle$ for $k\geq n$,
we have all
\begin{equation}
\label{eqn:aaa234}
\langle k,0|\circ\rangle \mbox{\ for $n-1 \le k \le 2n-4$.}
\end{equation}
Let $\langle \mu|\circ\rangle$ be neutral. Then if $\mu_1-\mu_2=n-2$, $\alpha_{(0)}\star \langle \mu_2,\mu_2|\circ\rangle = 2\langle \mu|\circ\rangle$. Otherwise, $\langle \mu_1-\mu_2,0|\circ\rangle \star \langle \mu_2,\mu_2|\circ\rangle$ is either $\langle \mu|\circ\rangle$ or $2\langle \mu|\circ\rangle$. Since $\langle \mu|\circ\rangle$ is neutral, $\mu_2\leq n-3$,
so we are done by Case 1, (\ref{eqn:aaa233}) and (\ref{eqn:aaa234}).

\noindent
{\bf Case 3:} (generating $\alpha^{(\lambda_1)}$ and neutral $\omu=\langle \mu|\bullet\rangle$): 
We have $\langle 2n-4,0|\bullet\rangle=
\tableau{{\ }} \star \langle 2n-4,0|\circ\rangle - \sctableau{{\ }\\{\ }} \star \langle 2n-5,0|\circ\rangle$. 
Repeatedly multiplying $\langle 2n-4,0|\bullet\rangle$ by $\tableau{{\ }}$ gives us all $\langle 2n-4,k|\bullet\rangle$ as well as $\alpha^{(2n-4)}$.

Since we have $\langle 2n-4,0|\bullet\rangle$, we also have 
$\langle 2n-5,1|\bullet\rangle= \sctableau{{\ }\\{\ }} \star \langle 2n-5,0|\circ\rangle-\langle 2n-4,0|\bullet\rangle$.
Since we have $\langle 2n-4,k|\bullet\rangle$, we can (by repeatedly multiplying $\langle 2n-5,1|\bullet\rangle$ by $\tableau{{\ }}$) obtain all $\langle 2n-5,k|\bullet\rangle$ as well as $\alpha^{(2n-5)}$. 
Now, $\langle 2n-6,2|\bullet\rangle=  \sctableau{{\ }\\{\ }} \star \langle 2n-6,1|\circ\rangle -\langle 2n-5,1|\bullet\rangle$, 
so we obtain all $\langle 2n-6,k|\bullet\rangle$ as well as $\alpha^{(2n-6)}$. 
Continuing in this way we obtain all neutral $\langle \mu| \bullet\rangle$ and all $\alpha^{(\lambda_1)}$ for $\lambda_1 \ge n-1$. 
Finally $\sctableau{{\ }\\{\ }}\star\alpha_{(n-3)} = 2\langle n-1,n-3|\bullet\rangle + 2\alpha^{(n-2)}$, thus we also obtain $\alpha^{(n-2)}$.
\end{proof}
Let $g^\uparrow:=\langle n-2,0|\circ\rangle^\uparrow \in {\mathbb Z}[{\mathbb Y}_{OG(2,2n)}]$. 
Also, let $g:=\Pi(g^\uparrow)$. Let $\delta^{(d)}=\tableau{{\ }}^{\ a}\star {\sctableau{{\ }\\{\ }}}^{ \ b}$ (for $a+b=d$).

\begin{Lemma}
\label{lemma:DambigassocA}
(\ref{eqn:Dambigassoc}) holds whenever $\olambda$ or $\omu$ is replaced by $\delta^{(d)}$.
\end{Lemma}

\begin{proof}
Without loss of generality, we may assume $a>0$ (the case $b>0$ is proved similarly). 
Let $\delta^{(d-1)}= \tableau{{\ }}^{\ a-1}\star {\sctableau{{\ }\\{\ }}}^{\ b}$.

We proceed by induction on $d$. The base case $d=1$ follows from (\ref{eqn:Dboxassoc}) and the commutativity of $\star$.
Suppose (\ref{eqn:Dambigassoc}) holds whenever $\olambda$ is replaced by any monomial in
$\tableau{{\ }}$ and $\sctableau{{\ }\\{\ }}$ of degree $d-1$. By commutativity of $\star$ and (\ref{eqn:Dboxassoc}):
\begin{multline}
(g^\uparrow \star \delta^{(d)}) \star \omu = (g^\uparrow \star (\delta^{(d-1)} \star\tableau{{\ }} )) \star \omu
= ((g^\uparrow \star \delta^{(d-1)}) \star \tableau{{\ }})\star \omu
= ( \tableau{{\ }} \star (g^\uparrow \star \delta^{(d-1)})) \star \omu \\ \nonumber
 = \tableau{{\ }}  \star ((g^\uparrow \star \delta^{(d-1)}) \star \omu) =
\tableau{{\ }} \star (g^\uparrow \star (\delta^{(d-1)} \star \omu))
=  \tableau{{\ }} \star ((\delta^{(d-1)} \star \omu) \star g^\uparrow) \\ \nonumber
=  ( \tableau{{\ }} \star (\delta^{(d-1)} \star \omu)) \star g^\uparrow
=  ( (\tableau{{\ }}\star \delta^{(d-1)}) \star \omu) \star g^\uparrow
=  (\delta^{(d)} \star \omu) \star g^\uparrow
= g^\uparrow \star (\delta^{(d)} \star \omu).
\end{multline}
\excise{
\begin{multline}
(g^\uparrow \star \delta^{(d)}) \star \omu = (g^\uparrow \star (\delta^{(d-1)} \star\tableau{{\ }} )) \star \omu
= ((g^\uparrow \star \delta^{(d-1)}) \star \tableau{{\ }})\star \omu\\ \nonumber
= ( \tableau{{\ }} \star (g^\uparrow \star \delta^{(d-1)})) \star \omu
 = \tableau{{\ }}  \star ((g^\uparrow \star \delta^{(d-1)}) \star \omu) =
\tableau{{\ }} \star (g^\uparrow \star (\delta^{(d-1)} \star \omu))\\ \nonumber
=  \tableau{{\ }} \star ((\delta^{(d-1)} \star \omu) \star g^\uparrow)
=  ( \tableau{{\ }} \star (\delta^{(d-1)} \star \omu)) \star g^\uparrow = g^\uparrow \star (\delta^{(d)} \star \omu).
\end{multline}
}
A similar argument shows (\ref{eqn:Dambigassoc}) also holds when $\omu$ is replaced by $\delta^{(d)}$.
\end{proof}

\begin{Corollary}\label{cor:Dambigassoceasy}
(\ref{eqn:Dambigassoc}) holds whenever $\olambda$ or $\omu$ is neutral. 

If $\okappa \in \mathbb{Y}_{OG(2,2n)}'$ is ambiguous,
\begin{itemize}
\item $g^\uparrow \star ((\okappa^\uparrow + \okappa^\downarrow) \star \omu) = (g^\uparrow \star (\okappa^\uparrow + \okappa^\downarrow)) \star \omu$; and
\item $g^\uparrow \star (\olambda \star (\okappa^\uparrow + \okappa^\downarrow)) = (g^\uparrow \star \olambda) \star (\okappa^\uparrow + \okappa^\downarrow)$.
\end{itemize}
\end{Corollary}
\begin{proof}
Follows immediately from Lemmas \ref{lemma:partialgen} and \ref{lemma:DambigassocA} and the distributivity of $\star$.
\end{proof}

\subsubsection{Both $\olambda$, $\omu$ are charged.} By Corollary \ref{cor:Dambigassoceasy} and the distributivity of $\star$, for any $\omu$ and any ambiguous $\okappa$:
\[g^\uparrow \star (\okappa^\uparrow \star \omu) + g^\uparrow \star (\okappa^\downarrow \star \omu) = (g^\uparrow \star \okappa^\uparrow) \star \omu + (g^\uparrow \star \okappa^\downarrow) \star \omu; \mbox{ \ and thus}\]
\[g^\uparrow \star (\okappa^\uparrow \star \omu) = (g^\uparrow \star \okappa^\uparrow) \star \omu \iff g^\uparrow \star (\okappa^\downarrow \star \omu) = (g^\uparrow \star \okappa^\downarrow) \star \omu.\]
Therefore, it suffices to prove (\ref{eqn:Dambigassoc}) for a choice of ${\rm ch}(\olambda)$. By a similar argument, we may also choose ${\rm ch}(\omu)$. Therefore, assume that $\olambda$ is down and $\omu$ is up. Furthermore, since $g^\uparrow$ is a Pieri shape, as in the proof of (\ref{eqn:Dboxassoc}) we may also assume that $\omu$ is a non-Pieri shape.

Suppose one of $\olambda$, $\omu$ is on. Then if $\olambda = \langle n-2,0|\circ\rangle^\downarrow$ and $\omu=\langle n-2,n-2|\bullet\rangle^\uparrow$, we compute $\leftassgenB=\rightassgenB=0$. Otherwise, $n-2+|\olambda|+|\omu|>4n-7$ and thus again $\leftassgenB=\rightassgenB=0$. Therefore, suppose $\olambda$, $\omu$ are off. It remains to check two cases:

\noindent {\bf Case 1:} ($\olambda=\langle n-2,\lambda_2|\circ\rangle^\downarrow$ with $\lambda_2 > 0$, $\omu = \langle n-2,\mu_2|\circ\rangle^\uparrow$ (which we assumed to have $\mu_2>0$)):

\noindent \emph{Subcase 1a:} ($\lambda_2+\mu_2 \le n-2$): Let $M=\min\{n-2-\lambda_2,n-2-\mu_2\}$. Then by (B)
\begin{multline}\nonumber
\Pi(\olambda)\diamond\Pi(\omu)= \langle 2n-4,\lambda_2+\mu_2-1|\bullet\rangle
+ 2\sum_{1 \le k \le M} \langle 2n-4-k,\lambda_2+\mu_2+k-1|\bullet\rangle\\ \nonumber
+ \langle 2n-4-M-1,\lambda_2+\mu_2+M|\bullet\rangle,
\end{multline}
By Lemma~\ref{lemma:firstrow}, $g^\uparrow \star \okappa=0$ for any $\okappa$ in $\olambda\star\omu$ with $\kappa_2>n-2$, so we may ignore such terms. In addition, $\lambda_2+\mu_2+M>n-2$ since $\lambda_2,\mu_2\neq 0$, so the last term of $\olambda\star\omu$ is ignored. Hence,
\begin{multline}
\label{eqn:assocgenjunk}
\olambda\star\omu=\frac{1}{2} \eta_{\olambda,\omu} \langle 2n-4,\lambda_2+\mu_2-1|\bullet\rangle + \!\!\!\!\!\! \sum_{1\le k\le n-2-(\lambda_2+\mu_2)} \!\!\! \langle 2n-4-k,\lambda_2+\mu_2+k-1|\bullet\rangle\\
+ (\langle n-2+\lambda_2+\mu_2-1,n-2|\bullet\rangle^\uparrow+\langle n-2+\lambda_2+\mu_2-1,n-2|\bullet\rangle^\downarrow)+\mbox{\ junk terms,}
\end{multline}
where the ``junk terms'' are those $\okappa$ we wished to ignore.

Label the non-junk terms of (\ref{eqn:assocgenjunk}) in the order given from $k=0$ to $k=m$, where $m=n-1-(\lambda_2+\mu_2)$ and the sum of the two charged shapes in (\ref{eqn:assocgenjunk}) is considered to be a single term. Let $\Delta[k]$ denote the $\star$-product of $g^\uparrow$ with the $k$th non-junk term, and let
\[Q_{\lambda_2, \mu_2}(N) = \sum_{0 \le i \le N}\langle 2n-4-i,n-2+\lambda_2+\mu_2-1+i|\bullet\rangle\in {\mathbb Z}[{\mathbb Y}_{OG(2,2n)}].\]
Then we compute:
\begin{itemize}
\item[(a)] $\Delta[k]=\Delta[m-k]=Q_{\lambda_2, \mu_2}(k)$
for $0<k\le \frac{m}{2}$;
\item[(b)] $\Delta[0]=\frac{1}{2}\eta_{\olambda,\omu} \langle 2n-4,n-2+\lambda_2+\mu_2-1|\bullet\rangle$;
\item[(c)] $\Delta[m]= \langle 2n-4,n-2+\lambda_2+\mu_2-1|\bullet\rangle$.
\end{itemize}
Since $\lambda_2+\mu_2>1$ ($\omu$ is non-Pieri), all terms of $\leftassgenB$ are neutral. Next,
\[g^\uparrow \star \olambda = \frac{1}{2} \eta_{g^\uparrow,\olambda} \langle 2n-4, \lambda_2-1|\bullet\rangle + \langle 2n-5,\lambda_2|\bullet\rangle + \ldots
+ \langle n-2+\lambda_2-1,n-2|\bullet\rangle^{{\rm ch}(\olambda)}.\]
As above, by Lemma~\ref{lemma:firstrow}, we may ignore all terms $\okappa$ of $g^\uparrow\star\olambda$ with $\kappa_1>2n-4-\mu_2$.
In particular, the first term of $g^\uparrow\star\olambda$ is ignored since $\mu_2\neq 0$. Thus
\begin{multline}\nonumber
g^\uparrow \star \olambda=\langle 2n-4-\mu_2,\lambda_2-1+\mu_2|\bullet\rangle 
+ \ldots + \langle n-2+\lambda_2-1,n-2|\bullet\rangle^{{\rm ch}(\olambda)}+\mbox{ junk terms}.\nonumber
\end{multline}
Label the non-junk terms in the order given from $j=0$ to $j=m$, and let $\nabla[j]$ denote the $\star$-product of the $j$th non-junk term with $\omu$. Then:
\begin{itemize}
\item[(a')] $\nabla[j]=\nabla[m-j]=Q_{\lambda_2, \mu_2}(j)$ for $0<j\le \frac{m}{2}$;
\item[(b')] $\nabla[m]=\frac{1}{2}\eta_{\olambda,\omu}\langle 2n-4,n-2+\lambda_2+\mu_2-1|\bullet\rangle$;
\item[(c')] $\nabla[0]=\langle 2n-4,n-2+\lambda_2+\mu_2-1|\bullet\rangle$.
\end{itemize}
So comparing (a),(b) and (c) respectively with (a'),(b') and (c') gives $\leftassgenB=\rightassgenB$.

\noindent \emph{Subcase 1b:} ($\lambda_2+\mu_2 > n-2$): If $\lambda_2+\mu_2>n-1$, then $|g^\uparrow|+|\olambda|+|\omu|>4n-7 = |\Lambda_{G/P}|$, so $\leftassgenB=\rightassgenB=0$. 
Thus assume $\lambda_2+\mu_2= n-1$. Then
\[\olambda \star \omu = \eta_{\olambda,\omu} \langle 2n-4,n-2|\bullet\rangle+\mbox{junk terms},\]
where the junk terms are those $\okappa$ with $\kappa_2> n-2$. These can be ignored in view of Lemma~\ref{lemma:firstrow}. If $\eta_{\olambda,\omu}=2$, the one non-junk term splits. Then
\[\leftassgenB=\frac{1}{2}\eta_{\olambda,\omu}\langle 2n-4,2n-4|\bullet\rangle.\]
Next,
\[g^\uparrow \star \olambda = \langle n-2+\lambda_2-1,n-2|\bullet\rangle^{{\rm ch}(\olambda)}+\mbox{junk terms}\]
where the junk terms are those $\okappa$ with $\kappa_1>2n-4-\mu_2$; they can be ignored in view of Lemma~\ref{lemma:firstrow}. Therefore,
\[\rightassgenB= \frac{1}{2} \eta_{\olambda,\omu} \langle 2n-4,2n-4|\bullet\rangle=\leftassgenB.\]

\noindent {\bf Case 2:} ($\olambda=\langle n-2,0|\circ\rangle^\downarrow$, $\omu=\langle n-2,\mu_2|\circ\rangle^\uparrow$ (where we have already assumed $\mu_2>0$)): 
If $\mu_2=n-2$, $\leftassgenB=\rightassgenB=\langle 2n-4,2n-5|\bullet\rangle$. Thus assume $\mu_2<n-2$:
\begin{multline}
\label{eqn:theolambdatimesomu}
\olambda \star \omu = \frac{1}{2} \eta_{\olambda,\omu}\langle 2n-4,\mu_2-1|\bullet\rangle + \langle 2n-5,\mu_2|\bullet\rangle + \ldots\\
 + \langle n-2+\mu_2,n-3|\bullet\rangle + \langle n-2+\mu_2-1,n-2|\bullet\rangle^{{\rm ch}(\omu)}.
\end{multline}
Label the $n-\mu_2$ terms of (\ref{eqn:theolambdatimesomu}) in the order given from $k=0$ to $k=n-\mu_2-1$. Throughout Case 2, let $\Delta[k]\in \mathbb{Z}[\mathbb{Y}'_{OG(2,2n)}]$ be the $\diamond$-product of $g$ with the $k$th term of $\Pi(\olambda\star\omu)$, and applying (i) and (ii) (but not (iii)). Then:
\begin{itemize}
\item $\Delta[0]=\frac{1}{2}\eta_{\olambda,\omu}\langle 2n-4,n-2+\mu_2-1|\bullet\rangle$;
\item $\Delta[n-\mu_2-1]=\frac{1}{2}\eta_{g^\uparrow,\omu}\langle 2n-4,n-2+\mu_2-1|\bullet\rangle$;
\item $\Delta[k]=\Delta[n-\mu_2-1-k]=\sum_{0 \le i \le k}\langle 2n-4-i,n-2+\mu_2-1+i|\bullet\rangle$
for $0<k\le \frac{n-\mu_2-1}{2}$.
\end{itemize}

\begin{Claim}\label{claim:case5genassocLHS}
$\leftassgenB=\sum_{0\le k \le n-\mu_2-1} \Delta[k]$. If $\mu_2=1$:
\begin{itemize}
\item The ambiguous term of $\Delta[0]$ is assigned $\uparrow$.
\item The ambiguous term of $\Delta[n-\mu_2-1]$ is assigned ${\rm ch}(\omu)=\uparrow$.
\item For $0<k<n-\mu_2-1$ the $i=0$ term (which is ambiguous) of $\Delta[k]$ is
assigned $\uparrow$ if $k$ is even and $\downarrow$ if $k$ is odd.
\end{itemize}
If $\mu_2>1$, there is no ambiguity.
\end{Claim}
\begin{proof}
If $\mu_2=1$, the charge assignments follow from (iii.2), (iii.3b) and (iii.3a) respectively. 
If $\mu_2>1$, then every term $\okappa$ in every $\Delta[k]$ has $|\kappa|>3n-6$ and is thus unambiguous.
\end{proof}

Define the following elements of ${\mathbb Z}[{\mathbb Y}'_{OG(2,2n)}]$:
\[S_{r}(N) = \sum_{0 \le i \le N}\langle 2n-4-2i-r,2i+r|\circ\rangle; \, \text{and} \]
\[T(N) = \sum_{0 \le l \le N}2\langle 2n-4-l,n-2+\mu_2-1+l|\bullet\rangle + \langle 2n-4-(N+1),n-2+\mu_2-1+(N+1)|\bullet\rangle.\]

Recall the expansion of $g^\uparrow\star\olambda$ is given directly in Definition~\ref{def:starproduct}. 
Let $\nabla[j]\in \mathbb{Z}[\mathbb{Y}'_{OG(2,2n)}]$ be the $\diamond$-product of the $j$th 
term of said expansion of $\Pi(g^\uparrow\star\olambda)$ 
with $\Pi(\omu)$, and applying (i) and (ii) (but not (iii)). (It will be convenient to fix the indexing $j$ below, to 
account for terms we wish to ignore in the expansion.)

\begin{Defclaim}\label{claim:case5genassocRHS2}
Suppose $n$ is even. Then

\begin{enumerate}
\item If $\mu_2$ is even,
\[\rightassgenB=\sum_{0\le j\le \frac{n-\mu_2-2}{2}}\nabla[j] \mbox{\ \ (neutral),}\]
where:
\begin{itemize}
\item $\nabla[0]=\langle 2n-4,n-2+\mu_2-1|\bullet\rangle$; and
\item $\nabla[j]=\nabla\left[\frac{n-\mu_2-2}{2}+1-j\right]=T(2j-1)$ for
$1\leq j\le \frac{1}{2}\left(\frac{n-\mu_2-2}{2}+1\right)$.
\end{itemize}

\item If $\mu_2$ is odd,
\[\rightassgenB=\sum_{0\le j\le \frac{n-\mu_2-3}{2}}\nabla[j];\]
where: 
\begin{itemize}
\item $\nabla[j]=T(2j)$ for $0\leq j\le \frac{1}{2}\left(\frac{n-\mu_2-3}{2}\right)$; and 
\item $\nabla\left[\frac{n-\mu_2-3}{2}-j\right]=T(2j+1)$ for $0\leq j< \frac{1}{2}\left(\frac{n-\mu_2-3}{2}\right)$
\end{itemize}
Also, if $\mu_2=1$, the ambiguous term in each $\nabla[j]$ splits. If $\mu_2>1$, all terms are neutral.
\end{enumerate}
\end{Defclaim}
\begin{proof}
By definition, $g^\uparrow \star \olambda = S_1(\frac{n-2}{2}-1)$ (neutral). For any term $\okappa$ of $S_1(\frac{n-2}{2}-1)$, $\okappa\star\omu=0$ if $\kappa_1>2n-4-(\mu_2-1)$ (Lemma~\ref{lemma:firstrow}).

($\mu_2$ is even): The terms of $g^\uparrow \star \olambda$ not annihilated by
 multiplication by $\omu$ are $S_{\mu_2-1}(\frac{n-\mu_2-2}{2})$ (Lemma~\ref{lemma:firstrow}).
Numbering these from $j=0$ (first) to $j=\frac{n-\mu_2-2}{2}$ (last), we compute the $\nabla[j]$'s as stated.
There is no ambiguity since $|\kappa|>3n-6$ for any $\okappa$ in $\rightassgenB$.

($\mu_2$ is odd): The terms of $g^\uparrow \star \olambda$ not annihilated by multiplication by $\omu$ are
$S_{\mu_2}(\frac{n-\mu_2-3}{2})$ (Lemma~\ref{lemma:firstrow}). Numbering these from $j=0$ (first) to $j=\frac{n-\mu_2-3}{2}$ (last), we compute the $\nabla[j]$'s as stated. If $\mu_2=1$, the splitting follows from (iii.1).
\end{proof}

\begin{Defclaim}\label{claim:case5genassocRHS4}
Suppose $n$ is odd. Then:
\begin{enumerate}
\item If $\mu_2$ is odd, 
\[\rightassgenB=\sum_{0\le j\le \frac{n-\mu_2-2}{2}}\nabla[j],\]
where:
\begin{itemize}
\item $\nabla[0]=\langle 2n-4,n-2+\mu_2-1|\bullet\rangle$;
\item $\nabla[j]=\nabla[\frac{n-\mu_2-2}{2}+1-j]=T(2j-1)$ for $1\leq j\le \frac{1}{2}\left(\frac{n-\mu_2-2}{2}+1\right)$.
\end{itemize}
Also, if $\mu_2=1$, the ambiguous term in $\nabla[0]$ is assigned ${\rm ch}(\omu)=\uparrow$, and the ambiguous term in each $\nabla[j]$ for $j>0$ splits. If $\mu_2>1$, all terms are neutral.

\item If $\mu_2$ is even,
\[\rightassgenB=\sum_{0\le j\le \frac{n-\mu_2-3}{2}}\nabla[j] \mbox{ \ \ (neutral).}\]
Here:
\begin{itemize}
\item $\nabla[j]=T(2j)$ for $0\leq j\le \frac{1}{2}\left(\frac{n-\mu_2-3}{2}\right)$; and
\item $\nabla\left[\frac{n-\mu_2-3}{2}-j\right]=T(2j+1)$ for
$0\leq j< \frac{1}{2}\left(\frac{n-\mu_2-3}{2}\right)$.
\end{itemize}

\end{enumerate}
\end{Defclaim}
\begin{proof}
Here, $g^\uparrow \star \olambda = S_0(\frac{n-3}{2})$ (neutral).

($\mu_2$ is odd): The terms of $g^\uparrow \star \olambda$ not annihilated by multiplication by $\omu$ are $S_{\mu_2-1}(\frac{n-\mu_2-2}{2})$
(Lemma~\ref{lemma:firstrow}). Numbering these from $j=0$ (first) to $j=\frac{n-\mu_2-2}{2}$ (last),
we compute the $\nabla[j]$'s as stated. If $\mu_2=1$, the charges follow from (iii.2) in $\nabla[0]$, and (iii.1) otherwise.

($\mu_2$ is even): The terms of $g^\uparrow \star \olambda$ not annihilated by multiplication by $\omu$ are $S_{\mu_2}(\frac{n-\mu_2-3}{2})$
(Lemma~\ref{lemma:firstrow}). Numbering these from $j=0$ (first) to $j=\frac{n-\mu_2-3}{2}$ (last),
we compute the $\nabla[j]$'s as stated. There is no ambiguity since $|\kappa|>3n-6$ for any $\okappa$ in $\rightassgenB$.
\end{proof}

\noindent
\emph{Conclusion of Case 2:}

\noindent \emph{Subcase 2a:}($n$ is even):
We use Claim~\ref{claim:case5genassocLHS} and Definition-Claim~\ref{claim:case5genassocRHS2}.

($\mu_2$ is even): Note that if $n-\mu_2 \equiv 0$ mod $4$, the last term of $\nabla\left[\frac{n-\mu_2}{4}\right]$ is illegal. Now,
\begin{itemize}
\item $\Delta[0]+\Delta[n-\mu_2-1]=\nabla[0]$
\item $\Delta[1]+\Delta[2]=\nabla[1]$
\item $\Delta[3]+\Delta[4]=\nabla[2]$ etc.
\end{itemize}

($\mu_2$ is odd): Here we see that
\begin{itemize}
\item $\Delta[0]+\Delta[1]+\Delta[n-\mu_2-1]=\nabla[0]$
\item $\Delta[2]+\Delta[3]=\nabla[1]$
\item $\Delta[4]+\Delta[5]=\nabla[2]$ etc.
\end{itemize}

\noindent \emph{Subcase 2b:} ($n$ is odd):
We use Claim~\ref{claim:case5genassocLHS} and Definition-Claim~\ref{claim:case5genassocRHS4}.

($\mu_2$ is odd): If $n-\mu_2 \equiv 0$ mod $4$, the last term of $\nabla\left[\frac{n-\mu_2}{4}\right]$
is illegal. Now use:
\begin{itemize}
\item $\Delta[0]+\Delta[n-\mu_2-1]=\nabla[0]$
\item $\Delta[1]+\Delta[2]=\nabla[1]$
\item $\Delta[3]+\Delta[4]=\nabla[2]$, etc.
\end{itemize}

($\mu_2$ is even): Finally, we use:
\begin{itemize}
\item $\Delta[0]+\Delta[1]+\Delta[n-\mu_2-1]=\nabla[0]$
\item $\Delta[2]+\Delta[3]=\nabla[1]$
\item $\Delta[4]+\Delta[5]=\nabla[2]$, etc.
\end{itemize}

\subsection{Conclusion of the proof of Theorem~\ref{Thm:Dmult}}

\begin{Proposition}
\label{cor:Dngenerate}
Assuming (\ref{eqn:Ddominoassoc}), (\ref{eqn:Dboxassoc}) and (\ref{eqn:Dambigassoc}), every $\olambda\in \mathbb{Y}_{OG(2,2n)}$ has a well-defined expression as a $\star$-polynomial in $\left\{\tableau{{\ }}, \sctableau{{\ }\\{\ }}, \langle n-2,0|\circ\rangle^\uparrow\right\}$ with rational coefficients.
\end{Proposition}
\begin{proof}
By Lemma~\ref{lemma:partialgen}, it is enough to check we can generate all the up charged shapes $\olambda$, i.e.,
$\langle n-2,\lambda_2|\circ\rangle^\uparrow$ and $\langle\lambda_1,n-2|\bullet\rangle^\uparrow$, using $\langle n-2,0|\circ\rangle^\uparrow$
and neutral shapes.

Indeed, $\tableau{{\ }} \star \langle n-2,k|\circ\rangle^\uparrow = \langle n-1,k|\circ\rangle+\langle n-2,k+1|\circ\rangle^\uparrow;
\mbox{\ for $0 \le k \le n-3$,}$
we inductively obtain all $\langle n-2,\lambda_2|\circ\rangle^\uparrow$.
Moreover,
$\tableau{{\ }} \star \langle n-2,n-2|\circ\rangle^\uparrow = \langle n-1,n-3|\bullet\rangle + 2\langle n-2,n-2|\bullet\rangle^\uparrow,$
giving $\langle n-2,n-2|\bullet\rangle^\uparrow$. Then 
$\langle n-2+\lambda_1,n-2|\bullet\rangle^\uparrow=\langle n-2,n-2|\bullet\rangle^\uparrow\star \langle\lambda_1,0|\circ\rangle$,
for $1 \le \lambda_1 \le n-3$. Finally, $\langle 2n-4,n-2|\bullet\rangle^\uparrow=\tableau{{\ }}\star\langle 2n-5,n-2|\bullet\rangle^\uparrow-\langle 2n-5,n-1|\bullet\rangle$ (the latter term being neutral).
\end{proof}

\begin{Lemma}\label{lemma:Dassocfinal}
The product $\star$ is associative.
\end{Lemma}
\begin{proof}
Let $\otau,\olambda, \omu \in \mathbb{Y}_{OG(2,2n)}$. We must establish
\begin{equation}\label{eqn:Dassocfinal}
\otau\star (\olambda \star \omu) =  (\otau\star \olambda) \star \omu
\end{equation}

By Proposition~\ref{cor:Dngenerate} and the distributivity of $\star$
it suffices to prove (\ref{eqn:Dassocfinal}) when $\otau$ is replaced by $\delta^{(d)}$, where $\delta^{(d)}=\tableau{{\ }}^{\ a}\star {\sctableau{{\ }\\{\ }}}^{ \ b} \star (\langle n-2,0|\circ\rangle^{\uparrow})^c$ and $a+b+c=d$. 
Assume $a>0$ (the cases $b>0$, $c>0$ are proved similarly). Let $\delta^{(d-1)}= \tableau{{\ }}^{\ a-1}\star \sctableau{{\ }\\{\ }}^{\ b} \star (\langle n-2,0|\circ\rangle^{\uparrow})^c$.

We proceed by induction on $d$. The case $d=1$ is exactly (\ref{eqn:Dboxassoc}).
Suppose (\ref{eqn:Dassocfinal}) holds if $\otau$ is replaced by any $\star$-monomial in
$\tableau{{\ }}$, $\sctableau{{\ }\\{\ }}$ and $\langle n-2,0|\circ\rangle^\uparrow$ of degree $d-1$. Then by (\ref{eqn:Dboxassoc}):
\begin{multline}
(\overline{\tau}\star\olambda)\star\omu = (\delta^{(d)}\star\olambda)\star\omu = ((\tableau{{\ }}\star \delta^{(d-1)})\star\olambda)\star\omu
= (\tableau{{\ }}\star (\delta^{(d-1)}\star\olambda))\star\omu = \tableau{{\ }}\star ((\delta^{(d-1)}\star\olambda)\star\omu)\\ \nonumber
= \tableau{{\ }}\star (\delta^{(d-1)}\star(\olambda\star\omu)) = (\tableau{{\ }}\star \delta^{(d-1)})\star(\olambda\star\omu)
= \delta^{(d)}\star(\olambda\star\omu) = \overline{\tau}\star(\olambda\star\omu). 
\end{multline}
\end{proof}

Using $\left\{\tableau{{\ }}, \sctableau{{\ }\\{\ }}, \langle n-2,0|\circ\rangle^\uparrow\right\}$ we have shown that $(R,\star)$
is an associative ring. In \cite{Searles:13+} it is shown that $\star$ agrees with a Pieri rule of \cite{BKT:Inventiones} for certain
generators of $H^*(OG(2,2n))$ indexed by shapes $\olambda$ that therefore also generate $R$. Hence, it follows $\olambda \leftrightarrow \sigma_{\olambda}$ defines an isomorphism of the rings $R$ and $H^{\star}(OG(2,2n))$.

\subsection{Integrality}

Theorems~\ref{Thm:Dintro} and~\ref{Thm:Dmult}, while not \emph{manifestly} integral, are
actually integral at each stage of the computation. 
Clearly, replacement
rule (i) is integral. Moreover:
\begin{Proposition}
\label{prop:isintegral}
Theorem~\ref{Thm:Dintro} and Theorem~\ref{Thm:Dmult}
return elements of ${\mathbb Z}[{\mathbb Y}_{OG(2,2n)}]$. More precisely:
\begin{itemize}
\item[(a)] The result of applying replacement rule (ii) is an element of ${\mathbb Z}[{\mathbb Y}'_{OG(2,2n)}]$.
\item[(b)] Replacement rule (iii) outputs an element of
${\mathbb Z}[{\mathbb Y}_{OG(2,2n)}]$.
\end{itemize}
\end{Proposition}

The case of Theorem~\ref{Thm:Dintro} follows from that of Theorem~\ref{Thm:Dmult} as
the computations are the same when $\olambda$
and $\omu$ are both non-Pieri shapes. We prove the latter case, directly from the definitions, in two lemmas:

\begin{Lemma}\label{lemma:integralpart1}
If $\Pi(\olambda)$, $\Pi(\omu)$ are not both $\langle n-2,0|\circ\rangle$ then applying (i) and (ii) to $\Pi(\olambda)\diamond\Pi(\omu)$ gives an element of ${\mathbb Z}[{\mathbb Y}'_{OG(2,2n)}]$.
\end{Lemma}
\begin{proof}
Since $\diamond$ and (i) are integral, it suffices to check (ii). Let $c\cdot\okappa$ be a term of $\Pi(\olambda)\diamond\Pi(\omu)$. 

\excise{In each of the following cases, $2^{{\rm fsh}(\okappa)-{\rm fsh}(\olambda)-{\rm fsh}(\omu)}$ is an integer.
This is obvious if ${\rm fsh}(\olambda)={\rm fsh}(\omu)=0$ or ${\rm fsh}(\olambda)={\rm fsh}(\omu)=2$. Suppose ${\rm fsh}(\olambda)+{\rm fsh}(\omu)=1$. Then if $\okappa$ is on, we have ${\rm fsh}(\okappa)\geq 1$ while if $\okappa$ is off, by (C) $\okappa$ contains both $\Pi(\olambda)$, $\Pi(\omu)$ and thus ${\rm fsh}(\okappa)=1$. Suppose ${\rm fsh}(\olambda)=2$ and ${\rm fsh}(\omu)=0$ (or vice versa). Then $\Pi(\olambda)$ (respectively $\Pi(\omu)$) is on, thus (by (C)) is contained in $\okappa$, so ${\rm fsh}(\okappa)=2$.}

We have ${\rm fsh}(\olambda), {\rm fsh}(\omu)\in\{0,1,2\}$. Of these, we only worry about the two cases where it is possible that $2^{{\rm fsh}(\okappa)-{\rm fsh}(\olambda)-{\rm fsh}(\omu)}=\frac{1}{2}$.

\noindent
{\bf Case 1:} (${\rm fsh}(\olambda)={\rm fsh}(\omu)=1$): 
Then $\lambda_1, \mu_1 \ge n-2$. By our hypothesis, $\okappa$ is on, ${\rm fsh}(\okappa)\ge 1$ and $\Pi(\olambda)\diamond\Pi(\omu)$ is given by (B). Then if $c\cdot\okappa$ is the $j$th term of (B) for $1\le j\le M$, then $c=2$, so $\okappa$'s coefficient remains integral after applying (ii). If the $(M+1)$th term is legal it has $2$ fake short roots. Finally, $\langle \lambda_1+\mu_1,\lambda_2+\mu_2-1|\bullet\rangle$ is legal only if $\lambda_1=\mu_1=n-2$, in which case (i) rescales it by $2$ or $0$, so its coefficient remains integral after applying (ii).

\noindent
{\bf Case 2:} (${\rm fsh}(\olambda)=2$ and ${\rm fsh}(\omu)=1$): (The same argument works if  ${\rm fsh}(\olambda)=1$ and ${\rm fsh}(\omu)=2$.) We may assume $\omu$ is off, otherwise $\Pi(\olambda)\diamond\Pi(\omu)=0$ by (D). Now $\lambda_1,\lambda_2,\mu_1 \ge n-2$. By Lemma~\ref{lemma:firstrow}, $\Pi(\olambda)\diamond\Pi(\omu)=0$ unless $\mu_1=\lambda_2=n-2$; in that case, $\Pi(\olambda)\diamond\Pi(\omu)=\langle 2n-4,\lambda_1+\mu_2|\bullet\rangle$. But then (i) rescales this term by $2$ or $0$ before (ii) rescales it by $\frac{1}{2}$. 
\end{proof}
\begin{Lemma}\label{lemma:integralpart2}
$\olambda\star\omu\in {\mathbb Z}[{\mathbb Y}_{OG(2,2n)}]$. 
\end{Lemma}
\begin{proof}
This is true by definition when $\olambda=\langle n-2,0|\circ\rangle^{{\rm ch}(\olambda)}$ and $\omu=\langle n-2,0|\circ\rangle^{{\rm ch}(\omu)}$. Otherwise, by Lemma~\ref{lemma:integralpart1}, it remains to check the possibilities for when (iii) splits an ambiguous  term $c\cdot\okappa$ of $\Pi(\olambda)\diamond\Pi(\omu)$.

\noindent{\bf Case 1:} ($\olambda, \omu$ are both neutral classes): Then (i) has no effect. If either ${\rm fsh}(\olambda)=2$ or ${\rm fsh}(\omu)=2$, there is no ambiguity. If ${\rm fsh}(\olambda)={\rm fsh}(\omu)=0$ then ${\rm fsh}(\okappa)\ge 1$, so $2^{{\rm fsh}(\okappa)-{\rm fsh}(\olambda)-{\rm fsh}(\omu)}$ is even. Suppose ${\rm fsh}(\olambda)=1$ and ${\rm fsh}(\omu)=0$, or vice versa. Then if $\Pi(\olambda)\diamond\Pi(\omu)$ is computed by (A) there is no ambiguity. If $\Pi(\olambda)\diamond\Pi(\omu)$ is computed by (B) or (C), then $\okappa$ is on. Then ${\rm fsh}(\okappa)=2$, so $2^{{\rm fsh}(\okappa)-{\rm fsh}(\olambda)-{\rm fsh}(\omu)}$ is even.

Suppose ${\rm fsh}(\olambda)={\rm fsh}(\omu)=1$. If either $\olambda$, $\omu$ is on, 
then $|\kappa|>3n-6$, so actually $\okappa$ is not  
ambiguous, by Lemma~\ref{cor:case1DomassocGone}. 
Otherwise, $\olambda$, $\omu$ are both off. Then $\Pi(\olambda)\diamond\Pi(\omu)$ is computed by (B), so ${\rm fsh}(\okappa)=2$. 
Hence (ii) multiplies by $1$.
Since $\olambda, \omu$ are neutral, $\lambda_1,\mu_1>n-2$, so the first term in (B) is illegal. The last term is unambiguous since $\lambda_2+\mu_2+M+1>n-2$ since $\lambda_2+\mu_2+M+1\in \{\lambda_1+\mu_2+1, \lambda_2+\mu_1+1\}$. Thus $c=2$ when we
apply (iii).

\noindent{\bf Case 2:} ($\olambda, \omu$ are both non-Pieri, at least one of which is charged): We may assume $\olambda$ is charged. If $\olambda$ is also on, then $\lambda_2=n-2$. Then since $\mu_2\neq 0$, by (C) $\kappa_2>n-2$ and $\okappa$ cannot be ambiguous. 

Assume $\olambda$ is off. Thus $\lambda_1=n-2$. If $\omu$ is on then $|\mu|\ge 2n-4$, and since $\lambda_2\neq 0$ we have $|\kappa|>3n-6$ thus $\okappa$ again cannot be ambiguous.  So suppose $\omu$ is off. If $|\olambda|+|\omu|\le 2n-4$, then $\okappa$ is off. By Lemma~\ref{lemma:firstrow}, $\kappa_1 \ge \lambda_1+\mu_2 >n-2$, so $\okappa$ can't be ambiguous. Otherwise, $|\olambda|+|\omu|> 2n-4$. Then $\Pi(\olambda)\diamond\Pi(\omu)$ is computed by (B) and $\okappa$ is on. Hence $\okappa= \langle \kappa_1,n-2|\bullet\rangle$. Assume ${\rm fsh}(\omu)=1$, since $2^{{\rm fsh}(\okappa)-{\rm fsh}(\olambda)-{\rm fsh}(\omu)}$ is even if ${\rm fsh}(\omu)=0$. Then $\okappa$ is not the last term of (B) since $\lambda_2+\mu_2+M>n-2$. If it is the first term of (B), then $\mu_1=n-2$ and $\omu$ is also charged. Then (i) multiplies $\okappa$ by $2$ or $0$, after which (ii) multiplies $\okappa$ by $1$. Otherwise, $c=2$, (i) has no effect, and (ii) multiplies $\okappa$ by $1$.

\noindent{\bf Case 3:} (One of $\olambda, \omu$ is charged and Pieri, the other is neutral and non-Pieri with size $\neq 2n-3$): Assume $\olambda=\langle n-2,0|\circ\rangle$ is the charged Pieri shape. If $|\omu|> 2n-3$ then any $\okappa$ has $|\kappa|>3n-6$, so $\okappa$ can't be ambiguous. So assume $|\omu|\le 2n-4$, i.e., $\omu$ is off. Since $\omu$ is neutral, $\mu_1 \neq n-2$. Since $\lambda_2=0$, by the definition of $\diamond$ we have $\kappa_2 \le \mu_1$.

If $|\omu|\le n-2$, then $\Pi(\olambda)\diamond\Pi(\omu)$ is computed by (A) and $\okappa$ is off. Since $|\okappa|=|\omu|+n-2$ and $\kappa_2 \le \mu_1$, we have $\kappa_1 \ge \mu_2+n-2 > n-2$, since $\mu_2\neq 0$. So $\okappa$ can't be ambiguous.

If $n-2<|\omu|\le 2n-4$, then $\Pi(\olambda)\diamond\Pi(\omu)$ is computed by (B) and $\okappa$ is on. Since $\okappa$ is ambiguous, then $\okappa=\langle \kappa_1,n-2|\bullet\rangle$. If $\mu_1<n-2$ then also $\kappa_2 < n-2$, so $\okappa$ is unambiguous. Thus assume $\mu_1>n-2$. Therefore, $\okappa=\langle \mu_1+\mu_2-1,n-2|\bullet\rangle$. This is not the last term of (B) since $\lambda_2+\mu_2+M>n-2$, and it is not the first term since $\mu_1+\mu_2-1<2n-4<\lambda_1+\mu_1$. Since applying (i) has no effect, and (ii) multiplies $\okappa$ by $1$, we are done here too, as $c=2$.

Concluding, (iii.1) falls to Cases 1 and 2, while (iii.2) falls to Case 1. Since for (iii.3) only (iii.3b) splits an ambiguous term, (iii.3) falls to Case 3.
\end{proof}

Finally:

\noindent\emph{Proof of Corollary~\ref{cor:Dintegral}:}
Since disambiguation either introduces a factor of $\frac{1}{2}$ or leaves the coefficients unchanged, it
 is now clear from the definition of $\star$ and ${\rm fsh}$ that all nonzero $C_{\olambda,\omu}^{\onu}
(OG(2,2n))$ are powers of $2$. The proof is then similar to that of Corollary~\ref{cor:BCcoefficients} 
(see the end of Section~4). 

\noindent\emph{Proof of Corollary~\ref{cor:Dnonzero}:}  The first sentence of the corollary is true by definition, so we may assume we are not in that
case. In the case $\nu_1=\frac{|\Lambda_{G/P}|-1}{2}$ if $\eta_{\olambda,\omu}=0$ then
clearly $C_{\olambda,\omu}^{\onu}(OG(2,2n))=0$, so we may assume that $\eta_{\olambda,\omu}\neq 0$ in this situation.

The idea is to reduce the problem to the analogous argument for Corollary~\ref{cor:BCnonzero}
by running a ``flattened'' argument. To do this it is easiest to use
the reformulation of the rule from the introduction given by Theorem~\ref{Thm:Dmult}. 

To be precise, let $\okappa = \Pi(\onu)$. Then $C_{\olambda,\omu}^{\onu}(OG(2,2n))\neq 0$ if and only if the 
coefficient $c$ of $\okappa$ in the expansion $\Pi(\olambda)\diamond\Pi(\omu)$ is nonzero and 
applying (i), (ii) and (iii) (of Theorem~\ref{Thm:Dmult}) yields a nonzero coefficient for $\onu$ 
(note that applying (i), (ii) and (iii) to a zero coefficient never yields a nonzero coefficient). Since we 
only need to consider  $(\olambda,\omu,\onu)$ where $\eta_{\olambda,\omu}\neq 0$, applying (i) 
multiplies $c$ by a nonzero number. By definition, (ii) multiplies the result of (i) by a nonzero number. 
Finally, since we only need to worry about non-Pieri cases, if $\onu$ is charged then (iii.1) multiplies 
the result of (ii) by $\frac{1}{2}$ and both $C_{\olambda,\omu}^{\okappa^\uparrow}(OG(2,2n))$ 
and $C_{\olambda,\omu}^{\okappa^\downarrow}(OG(2,2n))$ are equal to this resulting number. 
Thus $C_{\olambda,\omu}^{\onu}(OG(2,2n))\neq 0$ if and only if the coefficient of $\okappa$ in 
the expansion $\Pi(\olambda)\diamond\Pi(\omu)$ is nonzero, and then the proof is the same as that 
for Corollary~\ref{cor:BCnonzero} (see the end of Section~4).\qed

\section*{Acknowledgements}

We thank Nicolas Perrin, Kevin Purbhoo, Hugh Thomas and Alexander Woo for helpful 
conversations. A preliminary version of this text 
appears as an extended abstract for FPSAC 2013; we thank the anonymous
referees for their comments. DS and AY were supported by NSF grants DMS 
0901331 and DMS 1201595, and  
a grant from UIUC's Campus Research Board. This work was partially completed 
while AY was a Beckman fellow at UIUC's Center for Advanced Study and while 
supported by a Helen Corley Petit endowment.

\end{document}